\documentclass[12pt, reqno]{amsart}
\usepackage{amsfonts}
\usepackage{bbm}
\usepackage{amscd,amsfonts}
\usepackage{amssymb, eucal, amsfonts, amsmath, xypic, latexsym}
\usepackage{pifont}
\usepackage{mathrsfs,color}
\usepackage{amsthm,indentfirst,bm,fancyhdr,dsfont}
\usepackage{graphicx}
\usepackage[all]{xy}
\usepackage[CJKbookmarks=true]{hyperref}

\newtheorem{theorem}{Theorem}[section]
\newtheorem{lemma}[theorem]{Lemma}
\newtheorem{prop}[theorem]{Proposition}
\newtheorem{corollary}[theorem]{Corollary}
\theoremstyle{definition}

\newtheorem{defn}[theorem]{Definition}

\newtheorem{rem}[theorem]{Remark}

\newtheorem{conventions}[theorem]{Conventions}

\numberwithin{equation}{section}

\def\ggg{\mathfrak{g}}
\def\mmm{\mathfrak{m}}
\def\ppp{\mathfrak{p}}

\def\hhh{\mathfrak{h}}

\def\bbc{\mathbb{C}}
\def\bbf{\mathbb{F}}
\def\bbz{\mathbb{Z}}
\def\bbq{\mathbb{Q}}
\def\bbk{\mathds{k}}

\def\bo{{\bar 1}}
\def\bz{{\bar 0}}
\def\ev{{\text{ev}}}

\def\sfr{\textsf{r}}

\def\Lie{\text{Lie}}

{\vskip-\lastskip\medskip
  \noindent
  {\em #1.}\enspace
  }%
{\qed\par\medskip
  }

\begin{document}
\title{Finite $W$-superalgebras for basic Lie superalgebras}
\author{Yang Zeng and Bin Shu}
\thanks{\nonumber{{\it{Mathematics Subject Classification}} (2000 {\it{revision}})
Primary 17B35. Secondary 17B81. This work is  supported partially by the NSF of China (No. 11271130; 11201293; 111126062),  the Innovation Program of Shanghai Municipal Education Commission (No. 12zz038). }}
\address{School of Science, Nanjing Audit University, Nanjing, Jiangsu Province 211815, China}
\email{zengyang214@163.com}
\address{Department of Mathematics, East China Normal University, Shanghai 200241, China}
\email{bshu@math.ecnu.edu.cn}


\begin{abstract}
We consider the finite $W$-superalgebra $U(\mathfrak{g_\bbf},e)$ for a basic Lie superalgebra ${\ggg}_\bbf=(\ggg_\bbf)_\bz+(\ggg_\bbf)_\bo$ associated with a nilpotent element $e\in (\ggg_\bbf)_{\bar0}$ both over the field of complex numbers $\bbf=\mathbb{C}$ and over  $\bbf={\bbk}$ an algebraically closed field of positive characteristic. In this paper, we mainly present the PBW theorem for $U({\ggg}_\bbf,e)$. Then the construction of $U({\ggg}_\bbf,e)$ can be understood well, which in contrast with finite $W$-algebras, is divided into two cases in virtue of the parity of $\text{dim}\,\mathfrak{g_\bbf}(-1)_{\bar1}$. This observation will be a basis of our sequent work on the dimensional lower bounds in the super Kac-Weisfeiler property of modular representations of basic Lie superalgebras (cf. \cite[\S7-\S9]{ZS}).
\end{abstract}
\maketitle

\section*{Introduction}

\subsection{}\label{Premetintro} A finite $W$-algebra $U(\ggg,e)$ is a certain associative algebra associated to a complex semisimple Lie algebra ${\ggg}$ and a nilpotent element $e\in{\ggg}$. The study of finite $W$-algebras can be traced back to Kostant's work in the case when $e$ is regular \cite{Ko}, whose construction was generalized to arbitrary even nilpotent elements by Lynch \cite{Ly}. In the literature of mathematical physics, the finite $W$-algebras appeared in the work of de Boer and Tjin \cite{BT} from the viewpoint of BRST quantum hamiltonian reduction. The history is further complicated as there is enormous amount of work on affine $W$-algebras in the 1990's preceding the recent interest of finite $W$-algebras. Afterwards, Premet developed the finite $W$-algebras in full generality in \cite{P2}. On his way of proving the celebrated Kac-Weisfeiler conjecture for Lie algebras of reductive groups in \cite{P1}, Premet first constructed the modular version of finite $W$-algebras in \cite{P2}. By means of a complicated but natural ``admissible'' procedure, the finite $W$-algebras over the field of complex numbers were introduced, arising from the modular version, which showed that they are filtrated deformations of the coordinate rings of Slodowy slices. There the most important ingredient is the construction of the PBW basis of finite $W$-algebras (cf. \cite[\S4]{P2}).

For finite $W$-algebras over the field of complex numbers, Brundan-Kleshchev showed that those can be realized as shifted Yangians for the type $A$ case \cite{BK2}. Finite $W$-algebras theory becomes a very active area, and we refer the readers to the survey papers \cite{L1}, \cite{P6}, \cite{W} and references therein for more details.

Aside from the advances in finite $W$-algebras over $\mathbb{C}$, the modular theory of finite $W$-algebras is also in excitingly developing. As a most remarkable work, Premet proved in \cite{P7} that under the assumption $p\gg 0$ for the positive characteristic field ${\bbk}=\overline{\mathbb{F}}_p$, if the $\bbc$-algebra $U(\ggg,e)$ has a one-dimensional representation (which has been proved by Losev \cite{L3} and Goodwin-R\"{o}hrle-Ubly \cite{GRU} for all cases excluding type $E_8$ with $e$ rigid), then the reduced enveloping algebra $U_\chi(\ggg_\bbk)$ of the modular counterpart $\ggg_\bbk$ of $\ggg$ possesses a simple module whose dimension is exactly the lower bound predicted by Kac-Weisfeiler conjecture mentioned above.

\subsection{} The theory of finite $W$-superalgebras was developed in the same time.  In the work of De Sole and Kac \cite{SK}, finite $W$-superalgebras were defined in terms of BRST cohomology under the background of vertex algebras and quantum reduction. The theory of finite $W$-superalgebras for the queer Lie superalgebras (which are not basic Lie superalgebras) over an algebraically closed field of characteristic $p>2$ was first introduced and discussed by Wang and Zhao in \cite{WZ2}, then studied by Zhao over the field of complex numbers in \cite{Z2}. The connection between super Yangians and finite $W$-superalgebras was first obtained by Broit and Ragoucy in \cite{BR}. In \cite{BBG2}, the finite $W$-superalgebra associated to a principal nilpotent element was developed by Brown-Brundan-Goodwin. Some related results in this situation were also obtained independently by Poletaeva and Serganova in \cite{PS}, \cite{PS2}, \cite{PS3}. In these papers, they studied some generalities on finite $W$-superalgebras over the field of complex numbers, and details in the regular case for basic Lie superalgebras or the queer Lie superalgebras are obtained. In \cite{Peng2} and \cite{Peng3}, Peng established a connection between finite $W$-superalgebras and super Yangians explicitly in type $A$ with the Jordan type of the nilpotent element $e$ satisfying certain conditions. Now the theory of $W$-superalgebras related to super Yangians is still under investigating.

In \cite{WZ}, Wang and Zhao initiated the study of the modular representations of basic Lie superalgebras over an algebraically closed field of positive characteristic. There they formulated the super Kac-Weisfeiler property for basic Lie superalgebras, and introduced the modular $W$-superalgebras (those modular $W$-superalgebras will be called reduced $W$-superalgebras in the present paper).

\subsection{}  The main purpose of the present paper is to develop the construction theory of finite $W$-superalgebras both over the field of complex numbers and over a field in prime characteristic. Our approach is roughly generalizing the ``modular $p$ reduction'' method developed by Premet for the finite $W$-algebras in \cite{P2} and \cite{P7}, based on the results of basic Lie superalgebras given by Wang and Zhao in \cite{WZ}. As mentioned in \S \ref{Premetintro}, it becomes a crucial task to construct the PBW basis in our arguments. We successfully accomplish it by lots of nontrivial computations. Let us make further introduction below.

Let ${\ggg}={\ggg}_{\bar0}+{\ggg}_{\bar1}$ be a basic Lie superalgebra over $\mathbb{C}$ 
and $e\in{\ggg}_{\bar0}$ a nilpotent element. Fix an $\mathfrak{sl}_2$-triple $f,h,e\in{\ggg}_{\bar0}$, and denote by ${\ggg}^e:=\text{Ker}(\mbox{ad}\,e)$ a subalgebra of ${\ggg}$. The linear
operator ad\,$h$ defines a $\mathbb{Z}$-grading ${\ggg}=\bigoplus\limits_{i\in\mathbb{Z}}{\ggg}(i)$. Define the Kazhdan degree on ${\ggg}$ by declaring $x\in{\ggg}(j)$ is of $(j+2)$ in the new degree. We construct a $\mathbb{C}$-algebra (which is called a finite $W$-superalgebra)$$U({\ggg},e)=(\text{End}_{\ggg}Q_{\chi})^{\text{op}},$$
where $Q_{\chi}$ is the generalized Gelfand-Graev ${\ggg}$-module associated to $e$. Associated with a filtration of $U(\ggg,e)$ arising from the Kazhdan degree, we can define the corresponding graded superalgebra $\text{gr}\,U(\ggg,e)$.
One of the main results of the paper (Theorem \ref{PBWC} arising from its modular version Theorem \ref{reduced Wg}) presents a PBW basis of $U(\ggg,e)$ compatible with the Kazhdan grading, which is a super version of \cite[Theorem 4.6]{P2}. Then the consequent observation is the following structural information of $U(\ggg,e)$.
\begin{theorem}\label{graded W}
Keep the notations as above. The following statements hold.
\begin{itemize}
\item[(1)] gr\,$U(\ggg,e)\cong S(\ggg^e)$ as $\mathbb{C}$-algebras when $\dim\ggg(-1)_{\bar1}$ is even;
\item[(2)] gr\,$U(\ggg,e)\cong S(\ggg^e)\otimes\mathbb{C}[\Theta]$ as vector spaces over $\mathbb{C}$ when $\dim\ggg(-1)_{\bar1}$ is odd.
\end{itemize}
Here $S(\ggg^e)$ is the supersymmetric algebra on $\ggg^e$, and
$\mathbb{C}[\Theta]$ is the exterior algebra generated by one element $\Theta$ for the case when dim\,${\ggg}(-1)_{\bar1}$ is odd (under the canonical grading mapping, the element $\Theta$ is actually the image of $\Theta_{l+q+1}$ appearing in Theorem \ref{PBWC}, which only occurs when the vector space $\ggg(-1)_{\bar1}$ is odd-dimensional).
\end{theorem}

It is worth noting that after the draft of this paper has been written, we know from \cite{PS}, \cite{PS2} that Poletaeva and Serganova also noticed that the statements as  Theorem~\ref{graded W} may be true and formulated the corresponding conjecture in \cite[Conjecture 2.8]{PS2} recently. They proved that for any element $y\in{\ggg}^e$ if one can find $Y\in U({\ggg},e)$ such that $\text{gr}\,Y(1_\chi)=y$ (where $1_\chi$ denotes the image of $1$ in $Q_\chi$) under the Kazhdan grading, then Theorem~\ref{graded W} establishes.

Let us explain our approach. Let ${\ggg}_{\bbk}$ be the modular counterpart of Lie superalgebra $\ggg$ over a positive characteristic field ${\bbk}$. Some related topics on the reduced $W$-superalgebra $U_\chi({\ggg}_{\bbk},e)$ in positive characteristic are first studied in \textsection\ref{rew}, then we present the PBW theorem (Theorem~\ref{PBWC}) for the $\mathbb{C}$-algebra $U({\ggg},e)$ based on the parity of dim\,${\ggg}(-1)_{\bar1}$, respectively. Then Theorem~\ref{graded W} follows as a corollary of Theorem~\ref{PBWC}.

The key point above is that we find the construction of reduced $W$-superalgebra $U_\chi({\ggg}_{\bbk},e)$ critically depends on the parity of dim\,${\ggg}_{\bbk}(-1)_{\bar1}$.  When $r=\dim\ggg_\bbk(-1)_\bo$ is odd, $\ggg_\bbk(-1)_\bo$ is an odd-dimensional space with non-degenerate symmetric bilinear form $\langle\cdot,\cdot\rangle$, thereby satisfying $\langle v_i,v_j\rangle=\delta_{j,r-i+1}$ for $i,j=1,\cdots,r$ on a suitable basis $\{v_i\}$. In particular, $\langle v_{r+1\over 2},v_{r+1\over 2}\rangle=1$. Therefore, the leading terms of elements in $U_\chi(\ggg_\bbk,e)$ may admit a distinguished odd factor arising from $v_{r+1\over 2}$ (see the leading-terms Lemma \ref{hw}, Corollary \ref{rg} and their consequences).  This phenomenon makes the finite $W$-superalgebras strikingly different from their ordinary counterparts in the Lie algebra case. Hence the parity of $\dim\ggg(-1)_\bo$ becomes a crucial  factor deciding the change of the structure of finite $W$-superalgebras, which will be called a "detecting parity".

\subsection{} The paper is organized as follows. In \textsection\ref{prel}, some fundamental material on basic classical Lie superalgebras is recalled.
In \S\ref{FWC}, we introduce the finite $W$-superalgebra $U({\ggg},e)$ over $\mathbb{C}$, and the modular version  $U(\ggg_{\bbk},e)$ in prime characteristic $p$  along with the reduced $W$-superalgebra $U_\chi({\ggg}_{\bbk},e)$
{\sl(note that we abuse the notation $U(\ggg_\bbk,e)$ for $U_\chi(\ggg_\bbk,e)$ in \textsc{Abstract} for convenience of statements)}. In \S\ref{rew}, we introduce the generators and their relations for the ${\bbk}$-algebra $U_\chi({\ggg}_{\bbk},e)$, then  give the PBW Theorem on the basis of careful computation. In \S\ref{swc}, we formulate the PBW Theorem for finite $W$-superalgebra $U({\ggg},e)$ over $\mathbb{C}$ based on the results obtained in \textsection\ref{rew} by means of the ``admissible'' procedure. The relation between the refined finite $W$-superalgebra $Q_\chi^{\text{ad}{\mmm}'}$ in \cite[Remark 70]{W} and the finite $W$-superalgebra $U({\ggg},e)$ over $\mathbb{C}$ is discussed in the final part.

In a subsequent paper (see \cite[\S7-\S9]{ZS}), we will introduce the transition subalgebra $T({\ggg}_{\bbk},e)$ of $U(\ggg_\bbk,e)$, which can help us to better understand the construction of finite $W$-superalgebra $U(\ggg_{\bbk},e)$ in positive characteristic.
Based on the construction of the ${\bbk}$-algebra $T({\ggg}_{\bbk},e)$, we will formulate a conjecture on the minimal dimensional representations of finite $W$-superalgebras over the field of complex numbers, and prove that when the characteristic of the field $\bbk=\overline{\mathbb{F}}_p$ satisfies $p\gg 0$, the lower bounds predicted in the super Kac-Weisfeiler property \cite[Theorem 4.3]{WZ} can be reached under the assumption of this conjecture.

~

\subsection{} Throughout the paper we work with a field $\bbf$ for $\bbf=\mathbb{C}$ or $\bbf=$the algebraically closed field ${\bbk}=\overline{\mathbb{F}}_p$ in odd prime characteristic $p$ as the ground field.

Let $\mathbb{Z}_+$ be the set of all the non-negative integers in $\mathbb{Z}$, and denote by $\mathbb{Z}_2$ the residue class ring modulo $2$ in $\mathbb{Z}$. A superspace is a $\mathbb{Z}_2$-graded vector space $V=V_{\bar0}\oplus V_{\bar1}$, in which we call elements in $V_{\bar0}$ and $V_{\bar1}$ even and odd, respectively. Write $|v|\in\mathbb{Z}_2$ for the parity (or degree) of $v\in V$, which is implicitly assumed to be $\mathbb{Z}_2$-homogeneous. We will use the notations
$$\text{\underline{dim}}V=(\text{dim}V_{\bar0},\text{dim}V_{\bar1}),\quad\text{dim}V=\text{dim}V_{\bar0}+\text{dim}V_{\bar1}.$$
All Lie superalgebras ${\ggg}$ will be assumed to be finite-dimensional.  Given an ${\bbf}$-algebra $\mathscr{A}$ we denote by $\mathscr{A}$-mod the category of all finite-dimensional left $\mathscr{A}$-modules.

By vector spaces, subalgebras, ideals, modules, and submodules {\textit{etc}}. we mean in the super sense unless otherwise specified, throughout the paper.

\section{Basic classical Lie superalgebras and the corresponding algebraic supergroups}\label{prel}
In this section, we will recall some knowledge on basic classical Lie superalgebras along with the corresponding algebraic supergroups. We refer the readers to  \cite{CW, K} and \cite{M} for Lie superalgebras and to \cite{FG} and \cite{SW} for algebraic supergroups.
\subsection{Basic classical Lie superalgebras} \label{basicLiesuper}
Following  \cite[\S2.3-\S2.4]{K}, \cite[\S1]{K2}, \cite[\S1]{CW} and \cite[\S2]{WZ},  we first recall the list of basic classical Lie superalgebras over $\bbf$ for $\bbf=\bbc$ or $\bbf=\bbk$.
These Lie superalgebras with even parts being Lie
algebras of reductive algebraic groups are simple over $\bbf$ (the
general linear Lie superalgebra, though not simple, is also included), and they admit an even nondegenerate supersymmetric invariant bilinear form in the following sense.
\begin{defn}\label{form}
Let $V=V_{\bar0}\oplus V_{\bar1}$ be a $\mathbb{Z}_2$-graded space and $(\cdot,\cdot)$ be a bilinear form on $V$.
\begin{itemize}
\item[(1)] If $(a,b)=0$ for any $a\in V_{\bar0}, b\in V_{\bar1}$, then $(\cdot,\cdot)$ is called even;
\item[(2)] If $(a,b)=(-1)^{|a||b|}(b,a)$ for any homogeneous elements $a,b\in V$, then $(\cdot,\cdot)$ is called supersymmetric;
\item[(3)] If $([a,b],c)=(a,[b,c])$ for any homogeneous elements $a,b,c\in V$, then $(\cdot,\cdot)$ is called invariant;
\item[(4)] If one can conclude from $(a,V)=0$ that $a=0$, then $(\cdot,\cdot)$ is called non-degenerate.
\end{itemize}
\end{defn}

Note that when $\bbf$ is a field $\bbk$ whose characteristic equals to odd prime $p$, there are restrictions on $p$ as below (cf. \cite[Table 1]{WZ}). Then we have the following list

\begin{center}
\begin{tabular}{ccc}
\hline
 $\frak{g}$ & $\ggg_{\bar 0}$  & the restriction of $p$ when $\bbf=\bbk$\\
\hline
$\frak{gl}(m|n$) &  $\frak{gl}(m)\oplus \frak{gl}(n)$                &$p>2$            \\
$\frak{sl}(m|n)$ &  $\frak{sl}(m)\oplus \frak{sl}(n)\oplus \bbk$    & $p>2, p\nmid (m-n)$   \\
$\frak{osp}(m|n)$ & $\frak{so}(m)\oplus \frak{sp}(n)$                  & $p>2$ \\
$\text{D}(2,1,\bar a)$   & $\frak{sl}(2)\oplus \frak{sl}(2)\oplus \frak{sl}(2)$        & $p>3$   \\
$\text{F}(4)$            & $\frak{sl}(2)\oplus \frak{so}(7)$                  & $p>15$  \\
$\text{G}(3)$            & $\frak{sl}(2)\oplus \text{G}_2$                    & $p>15$     \\
\hline
\end{tabular}
\end{center}

Throughout the paper, we will simply call all $\ggg$ listed above {\sl{``basic Lie superalgebras"}}. When $\bbf$ is the field $\bbk$ of characteristic $p$, we always assume the restriction on $p$ as listed.

\subsection{Algebraic supergroups and restricted Lie superalgebras}\label{2.1}
For a given basic Lie superalgebra listed in \S\ref{basicLiesuper}, there is an algebraic supergroup $G$  satisfying $\Lie(G)=\ggg$ such that
\begin{itemize}
\item[(1)] $G$ has a subgroup scheme $G_\ev$ which is an ordinary connected reductive group with $\Lie(G_\ev)=\ggg_\bz$;
\item[(2)] There is a well-defined action of $G_\ev$ on $\ggg$, reducing to the adjoint action of $\ggg_\bz$.
    \end{itemize}
The above algebraic supergroup can be constructed as a Chevalley supergroup in \cite{FG}. The pair ($G_\ev, \ggg)$ constructed in this way is called a Chevalley super Harish-Chandra pair (cf.  \cite[Theorem 5.35]{FG} and \cite[\S3.3]{FG2}).
Partial results on $G$ and $G_\ev$ can be found in \cite[Ch. 2.2]{Ber}, \cite{FG}, \cite[\S3.3]{FG2} {\sl etc.}. In the present paper, we will call $G_\ev$ the purely even subgroup of $G$. One easily knows that $\ggg_\bbk$ is a restricted Lie superalgebra (cf. \cite[Definition 2.1]{SW} and \cite{SZ}) in the following sense when the ground field is  $\bbk$ of odd prime characteristic $p$.

\begin{defn}\label{restricted}
A Lie superalgebra ${\ggg}_{\bbk}=({\ggg}_{\bbk})_{\bar{0}}\oplus({\ggg}_{\bbk})_{\bar{1}}$ over ${\bbk}$ is called a restricted Lie superalgebra,
if there is a $p$-th power map $({\ggg}_{\bbk})_{\bar{0}}\rightarrow({\ggg}_{\bbk})_{\bar{0}}$, denoted as $(-)^{[p]}$, satisfying

(a) $(kx)^{[p]}=k^px^{[p]}$ for all $k\in{\bbk}$ and $x\in({\ggg}_{\bbk})_{\bar{0}}$;

(b) $[x^{[p]},y]=(\text{ad}x)^p(y)$ for all $x\in({\ggg}_{\bbk})_{\bar{0}}$ and $y\in{\ggg}_{\bbk}$;

(c) $(x+y)^{[p]}=x^{[p]}+y^{[p]}+\sum\limits_{i=1}^{p-1}s_i(x,y)$ for all $x,y\in({\ggg}_{\bbk})_{\bar{0}}$, where $is_i(x,y)$ is the
coefficient of $\lambda^{i-1}$ in $(\text{ad}(\lambda x+y))^{p-1}(x)$.
\end{defn}

 Let $\ggg_\bbk$ be a restricted Lie superalgebra.  For each $x\in (\ggg_\bbk)_\bz$, the element $x^p-x^{[p]}\in U(\ggg_\bbk)$ is central by the definition \ref{restricted}, and all of which generate a central subalgebra of $U(\ggg_\bbk)$.
Let $x_1,\cdots, x_s$ and $y_1,\cdots,y_t$ be the basis of $(\ggg_\bbk)_\bz$ and $(\ggg_\bbk)_\bo$ respectively.
For a given $\chi\in (\ggg_\bbk)_\bz^*$,  let $J_\chi$ be the ideal of the universal enveloping algebra $U(\ggg_\bbk)$ of $\ggg_\bbk$ generated by the even central elements $x^p-x^{[p]}-\chi(x)^p$ for all $x\in (\ggg_\bbk)_\bz$. The quotient algebra $U_\chi(\ggg_\bbk) := U(\ggg_\bbk)\slash J_\chi$ is called the reduced enveloping algebra with $p$-character $\chi$. We often regard $\chi\in \ggg_\bbk^*$ by letting $\chi((\ggg_\bbk)_\bo) = 0$. If $\chi= 0$, then $U_0(\ggg_\bbk)$ is called
the restricted enveloping algebra. It is a direct consequence from the PBW theorem that the superalgebra $U_\chi(\ggg_\bbk)$ is of dimension $p^s 2^t$, and has a basis
$$\{x_1^{a_1}\cdots  x_s^{a_s} y_1^{b_1}\cdots y_t^{b_t}
\mid 0\leq a_i < p; b_j = 0, 1 \mbox{ for  all }i, j\}.$$

\section{Finite $W$-superalgebras over $\bbc$ and their reduction versions modulo prime $p$}\label{FWC}

In this section we will introduce finite $W$-superalgebra $U(\ggg,e)$ over $\mathbb{C}$ associated with a basic superalgebra $\ggg=\ggg_\bz+\ggg_\bo$ and a nilpotent $e\in \ggg_\bz$,  along with some equivalent definitions (the second equivalent definition of finite $W$-superalgebras will be useful in our later arguments).  The super version of Skryabin's equivalence is presented here. Apart from the exploitation of Premet's treatment of finite $W$-algebras \cite[\S2]{P2} in the super case, some new phenomenon has to be dealt, which arises from the appearance of odd parts,  with aid of new techniques significantly different from the finite $W$-algebras. Then the introduction of  $\bbz$-admissible ring $A\subseteq \bbc$ associated with a Chevalley basis of $\ggg$ enables us to define the reduction version $U(\ggg_\bbk, \bar e)$ modulo a prime $p$.

\subsection{Chevalley basis for a basic Lie superalgebra and associated admissible $\bbz$-rings}\label{Chevalleybasis}

Let ${\ggg}$ be a basic Lie superalgebra over $\bbf$. The notation of Chevalley basis is an analogue of a classical theory of complex semisimple Lie algebras by Chevalley, in the super setting. In \cite{FG}, \cite{Gav} and \cite{SW}, the existence of Chevalley basis of basic Lie superalgebras is discussed.

In this section, we will take $\bbf=\bbc$.  Fix $\ggg$ to be a basic Lie superalgebra over $\bbc$ and ${\hhh}$ be a typical Cartan subalgebra of ${\ggg}$.
Let $\Phi$ be a root system of ${\ggg}$ relative to ${\hhh}$ whose simple roots system $\Delta=\{\alpha_1,\cdots,\alpha_l\}$ is distinguished (c.f. \cite[Proposition 1.5]{K2}).  Let $\Phi^+$ be the corresponding positive system in $\Phi$, and put $\Phi^-:=-\,\Phi^+$. Let ${\ggg}=\mathfrak{n}^-\oplus{\hhh}\oplus\mathfrak{n}^+$ be the corresponding triangular decomposition of ${\ggg}$. By \cite[\S3.3]{FG}, we can choose a Chevalley basis $B=\{e_\gamma\mid \gamma\in\Phi\}\cup\{h_\alpha\mid \alpha\in\Delta\}$ of ${\ggg}$ (in the case $D(2,1;a) (a\notin\mathbb{Z}$), one needs to adjust the definition of Chevalley basis by changing $\bbz$ to $\bbz[a]$ in the range of construction constants, see \cite[\S3.1]{Gav}. Here $\bbz[a]$ is the $\bbz$-algebra generated by $a$). Let ${\ggg}_\mathbb{Z}$ denote the Chevalley $\mathbb{Z}$-form in ${\ggg}$ and $U_\mathbb{Z}$ the Kostant $\mathbb{Z}$-form of $U({\ggg})$ associated to $B$. Given a $\mathbb{Z}$-module $V$ and a $\mathbb{Z}$-algebra $A$, we write $V_A:=V\otimes_\mathbb{Z}A$.

Let $G$ be an algebraic supergroup with $\Lie(G)={\ggg}$ as in \S\ref{2.1}, and $(G_\ev,\ggg)$ be a super Harish-Chandra pair. Let $e\in{\ggg}_{\bar0}$ be a nilpotent element. By the Dynkin-Kostant theory (cf. \cite{Ca} or \cite{CM}), $\text{ad}\,G_{\text{ev}}.e$ interacts with $({\ggg}_\mathbb{Z})_{\bar{0}}$ nonempty. Without loss of generality, we can assume that the nilpotent element $e$ is in $({\ggg}_\mathbb{Z})_{\bar{0}}$. Then by the same discussion as \cite[\S 4.2]{P2}, one can find $f,h\in({\ggg}_\mathbb{Q})_{\bar{0}}$ such that $(e,h,f)$ is an $\mathfrak{sl}_2$-triple in ${\ggg}$.

\begin{prop}\label{invariant bilinear}
Let ${\ggg}$ be a basic Lie superalgebra (excluding type $D(2,1;a) (a\notin\mathbb{Z}$)) over $\mathbb{C}$. Then there exists an even non-degenerate supersymmetric invariant bilinear form $(\cdot,\cdot)$ on ${\ggg}$, under which the Chevalley basis of ${\ggg}$ takes value in $\mathbb{Q}$.
\end{prop}

\begin{proof} 
 For each case listed in \S\ref{basicLiesuper}, a Chevalley basis of ${\ggg}$ (excluding the case $D(2,1;a) (a\notin\mathbb{Z}$)) was constructed by R. Fioresi and F. Gavarini in \cite[\S3.3]{FG} explicitly (the orthogonal-symplectic case was first introduced  in \cite{SW}). We will choose these vectors as a basis of ${\ggg}$. For each case one can easily certify the statements in the proposition case by case. We omit the detailed computation here.
\end{proof}

\begin{rem} \label{mathbbq}
It follows from Proposition \ref{invariant bilinear} and the discussion earlier that $(e,f)\in\mathbb{Q}$. \cite[Proposition 2.5.5(c)]{K} shows that the non-degenerate, supersymmetric and invariant bilinear form on any basic Lie superalgebra is uniquely determined up to a constant factor. Therefore, we can assume $(e,f)=1$ and $(\cdot,\cdot)$ takes value in $\mathbb{Q}$ under the Chevalley basis of ${\ggg}$ in \cite{FG}. Define $\chi\in{\ggg}^{*}$ by letting $\chi(x)=(e,x)$ for $x\in{\ggg}$, then we have $\chi({\ggg}_{\bar{1}})=0$.
\end{rem}

\begin{rem}\label{fractZa} In the case $D(2,1,a)$ with $a\notin \bbz$, there is a set of adjusted ``Chevalley basis" such that $D(2,1,a)$ has a $\bbz[a]$-lattice spanned by ``Chevalley basis" (cf. \cite[\S3.1]{Gav}). Thus we can claim in this case that $\ggg$ admits an even non-degenerate supersymmetric invariant bilinear form $(\cdot,\cdot)$ such that the adjusted ``Chevalley basis" takes value in the fraction algebra of $\bbz[a]$.
\end{rem}

Following Premet's notion (cf. \cite{P7}), we introduce admissible $\bbz$-rings associated with an (adjusted) Chevalley basis discussed previously, which will be critically useful for our definition of finite $W$-superalgebras over $\bbk$.

\begin{defn}\label{admissible}
We call a commutative (in the usual sense, not super) ring $A$ admissible if $A$ is a finitely generated $\mathbb{Z}$-subalgebra of $\mathbb{C}$, $(e,f)\in A^{\times}(=A\backslash \{0\})$, and all bad primes of the root system of ${\ggg}$ and the determinant of the Gram matrix of ($\cdot,\cdot$) relative to a Chevalley basis of ${\ggg}$ are invertible in $A$.
\end{defn}

For example, one can take $A=\mathbb{Z}[\frac{1}{N!}]$ for any sufficiently large integer $N$ for all the cases excluding type $D(2,1;a)$($a\notin\mathbb{Z}$), then $A$ is an admissible ring.

It is clear by the definition that every admissible ring is a Noetherian domain. Given a finitely generated $\mathbb{Z}$-subalgebra $A$ of $\mathbb{C}$, it is well known that for every $\mathfrak{P}\in\text{Specm}\,A$ the residue field $A/\mathfrak{P}$ is isomorphic to $\mathbb{F}_{q}$, where $q$ is a $p$-power depending on $\mathfrak{P}$. We denote by $\Pi(A)$ the set of all primes $p\in\mathbb{N}$ that occur in this way.

Since the choice of $A$ does not depend on the super structure of ${\ggg}$, it follows from the arguments in the proof of \cite[Lemma 4.4]{P7} that the set $\Pi(A)$ contains almost all primes in $\mathbb{N}$.  Let $p$ be a prime with $p\gg N$, i.e. $p\gg0$, then $p\in\Pi(A)$ for $A=\mathbb{Z}[\frac{1}{N!}]$  associated with any sufficiently large integer $N$ and  any cases except type $D(2,1;a)$($a\notin\mathbb{Z}$).
\begin{rem}\label{except}
In the case $D(2,1;a)$ ($a\notin\bbz$) with $a$ being an algebraic number, by Remark \ref{fractZa} we can enlarge the admissible ring such that the fraction algebra of $\bbz[a]$ is contained in $A$. Therefore, the basic Lie superalgebras ${\ggg}$ over $\bbc$ will be referred to all types in the article, except $D(2,1;a)$ with $a$ not being an algebraic number, such that the associated finite $W$-superalgebras are studied.
\end{rem}

\subsection{Dynkin gradings and finite $W$-superalgebras over $\mathbb{C}$}\label{defw}
 Return to the assumptions and notations in the paragraph prior to Proposition \ref{invariant bilinear}. Specially, fix a basic Lie superalgebra $\ggg$ over $\bbc$ with an even non-degenerate supersymmetric invariant bilinear form $(\cdot,\cdot)$. Let $e\in (\ggg_\bbz)_\bz$ be a given nilpotent element, then by Jacobson-Morozov theorem there is an $\frak{sl}_2$-triple $\{e,f,h\}$ with $f,h\in (\ggg_\bbq)_\bz$ (cf. \cite{CM}). Recall that a $\mathbb{Z}$-grading on $\ggg$ is called a {\sl Dynkin grading} if it is defined by $\text{ad}\,h$. Let ${\ggg}(i)=\{x\in{\ggg}\mid [h,x]=ix\}$ be the decomposition of ${\ggg}$ under the Dynkin grading, then ${\ggg}=\bigoplus\limits_{i\in\mathbb{Z}}{\ggg}(i).$ By the $\mathfrak{sl}_2$-theory, all subspaces ${\ggg}(i)$ defined are over $\mathbb{Q}$. Also, $e\in{\ggg}(2)_{\bar{0}}$ and $f\in{\ggg}(-2)_{\bar{0}}$. By \cite[Lemma 2.7(i)]{H} we know that if the integers $i$ and $j$ satisfy $i+j\neq0$, then $({\ggg}(i),{\ggg}(j))=0$. Moreover, there exist non-degenerate  symplectic and symmetric bilinear forms $\langle\cdot,\cdot\rangle$ on the $\mathbb{Z}_2$-graded subspaces ${\ggg}(-1)_{\bar{0}}$ and ${\ggg}(-1)_{\bar{1}}$ respectively,
  given by $$\langle x,y\rangle:=(e,[x,y])=\chi([x,y])$$ for all $x,y\in{\ggg}(-1)_{\bar0}~(\text{resp.}\,x,y\in{\ggg}(-1)_{\bar1})$.

It follows from \cite[\S4.1]{WZ} that dim\,${\ggg}(-1)_{\bar{0}}$ is even. Take ${\ggg}(-1)_{\bar{0}}^{\prime}\subseteq{\ggg}(-1)_{\bar{0}}$ to be a maximal isotropic subspace with respect to $\langle\cdot,\cdot\rangle$, then dim\,${\ggg}(-1)_{\bar{0}}^{\prime}=\frac{\text{dim}\,{\ggg}(-1)_{\bar{0}}}{2}
:=s$.
Let $u_{s+1},\cdots,u_{2s}$ be a basis of ${\ggg}(-1)_{\bar{0}}^{\prime}$, then we can choose a basis $u_1,\cdots,u_s$ of ${\ggg}(-1)_{\bar{0}}\cap({\ggg}(-1)_{\bar{0}}^{\prime})^{\bot }$ (where $({\ggg}(-1)_{\bar{0}}^{\prime})^{\bot }$ denotes the subspace of ${\ggg}(-1)_{\bar{0}}$ which is orthogonal to ${\ggg}(-1)_{\bar{0}}^{\prime}$ with respect to $\langle\cdot,\cdot\rangle$) such that $u_1,\cdots,u_s,u_{s+1},\cdots,u_{2s}$ is a basis of ${\ggg}(-1)_{\bar{0}}$ under which the symplectic form $\langle\cdot,\cdot\rangle$ has matrix form
\newcommand*{\adots}{\mathinner{\mkern2mu\raisebox{0.1em}{.}
   \mkern2mu\raisebox{0.4em}{.}\mkern2mu\raisebox{0.7em}{.}\mkern1mu}}
\[\left(
\begin{array}{llllll}
&&&&&-1\\
&&&&\adots&\\
&&&-1&&\\
&&1&&&\\
&\adots&&&&\\
1&&&&&
\end{array}
\right),
\]
i.e. for any $1\leqslant i\leqslant 2s$, if we define \[i^*=\left\{\begin{array}{ll}-1&\text{if}~1\leqslant i\leqslant s;\\ 1&\text{if}~s+1\leqslant i\leqslant 2s,\end{array}\right.\] then $\langle u_i, u_j\rangle =i^*\delta_{i+j,2s+1}$ for $1\leqslant i,j\leqslant 2s$, where $\delta_{i,j}$ is the Kronecker symbol.

{\sl{Significant difference of the theory of finite $W$-superalgebras from that of finite $W$-algebras is primarily resulted from the parity of $\dim\ggg(-1)_\bo$.}}
 Let us look at some beginning things.
We set $\text{dim}\,{\ggg}(-1)_{\bar1}=r$. Then we can choose a basis $v_1,\cdots,v_r$ of ${\ggg}(-1)_{\bar{1}}$ on which the symmetric form $\langle\cdot,\cdot\rangle$ has matrix form
\[\left(
\begin{array}{lll}
&&1\\
&\adots&\\
1&&
\end{array}
\right),
\]
i.e. for any $1\leqslant i,j\leqslant r$, $\langle v_i,v_j\rangle=\delta_{i+j,r+1}$.

Since the bilinear form $\langle\cdot,\cdot\rangle$ on ${\ggg}(-1)_{\bar{1}}$ is symmetric, the dimension of ${\ggg}(-1)_{\bar{1}}$ is not necessary an even number. If $r$ is even, then take ${\ggg}(-1)_{\bar{1}}^{\prime}\subseteq{\ggg}(-1)_{\bar{1}}$ to  be the subspace spanned by $v_{\frac{r}{2}+1},\cdots,v_r$. If $r$ is odd, then take ${\ggg}(-1)_{\bar{1}}^{\prime}\subseteq{\ggg}(-1)_{\bar{1}}$ to be the subspace spanned by $v_{\frac{r+3}{2}},\cdots,v_r$. Set ${\ggg}(-1)^{\prime}={\ggg}(-1)'_{\bar{0}}\oplus{\ggg}(-1)'_{\bar{1}}$ and introduce the subalgebras
$${\mmm}=\bigoplus_{i\leqslant -2}{\ggg}(i)\oplus{\ggg}(-1)^{\prime},\qquad {\ppp}=\bigoplus_{i\geqslant 0}{\ggg}(i),$$
\[{\mmm}^{\prime}=\left\{\begin{array}{ll}{\mmm}&\text{if}~r~\text{is even;}\\
{\mmm}\oplus \mathbb{C}v_{\frac{r+1}{2}}&\text{if}~r~\text{is odd.}\end{array}\right.\]
For any real number $a\in\mathbb{R}$, let $\lceil a\rceil$ denote the largest integer lower bound of $a$, and $\lfloor a\rfloor$ the least integer upper bound of $a$. In particular, $\lceil a\rceil=\lfloor a\rfloor=a$ when $a\in\mathbb{Z}$. Clearly,  $u_1,\cdots,u_s\in({\ggg}(-1)_{\bar{0}}^{\prime})^{\bot }$ and $ v_1,\cdots,v_{\lfloor\frac{r}{2}\rfloor}\in({\ggg}(-1)_{\bar{1}}^{\prime})^{\bot }$. Here ${\bot}$ is respect to the bilinear form $\langle\cdot,\cdot\rangle$ on the space  ${\ggg}(-1)$.

\begin{conventions}\label{firstconventions}
From now on, we will reset the detecting dimension in connection to the detecting parity
$$\sfr:=\dim\ggg(-1)_{\bar{1}}$$
 for  stressing its distinguished role in the  sequent arguments.   We will  denote $\lfloor\frac{\sfr}{2}\rfloor$ by $t$ hereafter for convenience which actually equals to the dimension of $({\ggg}(-1)_{\bar{1}}^{\prime})^{\bot }$.
\end{conventions}

\begin{rem}\label{centralizer}
Write ${\ggg}^e$ for the centralizer of $e$ in ${\ggg}$, and ${\ggg}^f$ the centralizer of $f$ in ${\ggg}$. For any $i\in\mathbb{Z}_2$, denote $d_i:=\text{dim}\,{\ggg}_i-\text{dim}\,{\ggg}^e_i$. It follows from \cite[Theorem 4.3]{WZ} that
$$\text{\underline{dim}}\,{\ggg}-\text{\underline{dim}}\,{\ggg}^e=\sum\limits_{k\geqslant2}
2\text{\underline{dim}}\,{\ggg}(-k)+\text{\underline{dim}}\,{\ggg}(-1).$$

In particular, $\text{dim}\,{\ggg}(-1)_{\bar1}$ and $d_1$ always have the same parity. It follows from the definition of ${\mmm}$ that either (1) $(\frac{d_0}{2},\frac{d_1}{2})=\text{\underline{dim}}\,{\mmm}$ when $\text{dim}\,{\ggg}(-1)_{\bar1}$ (or $d_1$, equivalently) is even, or (2) $(\frac{d_0}{2},\frac{d_1-1}{2})=\text{\underline{dim}}\,{\mmm}$ when $\text{dim}\,{\ggg}(-1)_{\bar1}$ (or $d_1$) is odd.
\end{rem}
By the same discussion as \cite[\S2.1]{P7}, we can assume ${\ggg}_A=\bigoplus\limits_{i\in\mathbb{Z}}{\ggg}_A(i)$ after enlarging $A$ if necessary, and each ${\ggg}_A(i):={\ggg}_A\cap{\ggg}(i)$ is freely generated over $A$ by a basis of the vector space ${\ggg}(i)$. Then $\{u_1,\cdots,u_{2s}\}$ and $\{v_1,\cdots,v_\sfr\}$ are free basis of $A$-module ${\ggg}_A(-1)_{\bar0}$ and ${\ggg}_A(-1)_{\bar1}$, respectively. By the assumptions on $A$ one can obtain that
${\mmm}_A:={\ggg}_A\cap{\mmm}$, ${\mmm}^{\prime}_A:={\ggg}_A\cap{\mmm}^{\prime}$ and ${\ppp}_A:={\ggg}_A\cap{\ppp}$ are free $A$-modules and direct summands of ${\ggg}_A$. More precisely,
$${\mmm}_A={\ggg}_A(-1)^{\prime}\oplus\bigoplus\limits_{i\leqslant -2}{\ggg}_A(i),~ \text{where}
~{\ggg}_A(-1)^{\prime}={\ggg}_A\cap{\ggg}(-1)^{\prime},\quad {\ppp}_A=\bigoplus_{i\geqslant 0}{\ggg}_A(i),$$
\[{\mmm}^{\prime}_A=\left\{\begin{array}{ll}{\mmm}_A&\text{if}~\sfr~\text{is even;}\\{\mmm}_A\oplus A
v_{\frac{\sfr+1}{2}}&\text{if}~\sfr~\text{is odd.}\end{array}\right.\]

Let ${\ggg}^*$ be the $\mathbb{C}$-module dual to ${\ggg}$ and let ${\mmm}^\perp$ denote the set of all linear functions on ${\ggg}$ vanishing on ${\mmm}$. By the discussion at the beginning of \S\ref{defw} we have  $e\in({\ggg}_\mathbb{Z})_{\bar{0}}, f\in({\ggg}_\mathbb{Q})_{\bar{0}}$. Hence one can assume $e,f\in({\ggg}_A)_{\bar0}$ after enlarging $A$ possibly (for example, if the admissible ring is chosen as $\mathbb{Z}[\frac{1}{N!}]$, then one can just select a sufficiently large positive integer $N\gg0$) and that $[e,{\ggg}_A(i)]$ and $[f,{\ggg}_A(i)]$ are direct summands of ${\ggg}_A(i+2)$ and ${\ggg}_A(i-2)$, respectively. By the $\mathfrak{sl}_2$-theory we have ${\ggg}_A(i+2)=[e,{\ggg}_A(i)]$ for each $i\geqslant -1$.

Since the vectors in ${\ggg}$ can be identified with their dual vectors in ${\ggg}^*$ by the non-degenerate bilinear form $(\cdot,\cdot)$, we will identify the functions on ${\ggg}$ naturally with the vectors in ${\ggg}$.

\begin{lemma}\label{m'}
For the subspace ${\mmm}^\perp$ of Lie superalgebra ${\ggg}$, we have
$${\mmm}^\perp=[{\mmm}',e]\oplus{\ggg}^f.$$
\end{lemma}

\begin{proof}
When $\text{dim}\,{\ggg}(-1)_{\bar1}$ is even, i.e. ${\mmm}'={\mmm}$, the proof is the same as the Lie algebra case (see e.g. \cite[Lemma 26]{W}). When $\text{dim}\,{\ggg}(-1)_{\bar1}$ is odd, i.e. ${\mmm}'\neq{\mmm}$, minor modifications are needed for the proof. We will just sketch the proof as follows:

(1) ${\ggg}^f\subseteq{\mmm}^\perp$. This follows from ${\ggg}^f\subseteq\bigoplus\limits_{i\leqslant 0}{\ggg}(i)\subseteq{\mmm}^\perp$.

(2) $[{\mmm}',e]\subseteq{\mmm}^\perp$. This can be seen by $([e,{\mmm}'],{\mmm})=(e,[{\mmm}',{\mmm}])=\chi([{\mmm}',{\mmm}])=0$.

(3) $\text{Im}(\text{ad}e)\cap{\ggg}^f=0$. This follows from the $\mathfrak{sl}_2$-representation theory.

(4) $\underline{\text{dim}}\,{\mmm}^\perp=\underline{\text{dim}}~{\mmm}'+\underline{\text{dim}}\,{\ggg}(0)+\underline{\text{dim}}~{\ggg}(-1)=\underline{\text{dim}}~[{\mmm}',e]+\underline{\text{dim}}~{\ggg}^f$. This follows by the bijection ${\mmm}'\rightarrow [{\mmm}',e], x\mapsto [x,e]$, by (2), and the $\mathfrak{sl}_2$-representation theory.
\end{proof}

\begin{lemma}\label{p}
For the subalgebra ${\ppp}$ of Lie superalgebra ${\ggg}$, we have
$${\ppp}=\bigoplus\limits_{j\geqslant 2}[f,{\ggg}(j)]\oplus{\ggg}^e.$$
\end{lemma}

\begin{proof}
The proof is straightforward and is the same as the Lie algebra case (see e.g. \cite[Lemma 2.2]{BGK}).
\end{proof}

By Lemma~\ref{p} and the assumptions on $A$, we can choose a basis $x_1,\cdots,x_l,x_{l+1},$\\$\cdots,x_m\in{\ppp}
_{\bar{0}}, y_1,\cdots,y_q,y_{q+1},$ $\cdots,y_n\in{\ppp}_{\bar{1}}$ of ${\ppp}=\bigoplus\limits_{i\geqslant 0}{\ggg}(i)$ such that

(a) $x_i\in{\ggg}(k_i)_{\bar{0}}, y_j\in{\ggg}(k'_j)_{\bar{1}}$, where $k_i,k'_j\in\mathbb{Z}_+$;

(b) $x_1,\cdots,x_l$ is a basis of ${\ggg}^e_{\bar{0}}$ and $y_1,\cdots,y_q$ is a basis of ${\ggg}^e_{\bar{1}}$;

(c) $x_{l+1},\cdots,x_m\in[f,{\ggg}_{\bar{0}}]$ and $ y_{q+1},\cdots,y_n\in[f,{\ggg}_{\bar{1}}]$;\\
then the corresponding elements of (a), (b) and (c) in $A$ form a basis of the free $A$-module ${\ppp}_A=\bigoplus\limits_{i\geqslant 0}{\ggg}_A(i)$ after enlarging admissible ring $A$ if needed. Summing up the arguments before, we have that there is a co-basis of $\ggg$ modulo $\mmm$
$$\{x_1,\cdots,x_m; u_1,\cdots,u_s; y_1,\cdots,y_n; v_1,\cdots,v_t\}$$
 with  $x_1,\cdots,x_m; u_1,\cdots,u_s\in\ggg_\bz$ and $ y_1,\cdots,y_n; v_1,\cdots,v_t\in\ggg_\bo$.

\begin{defn}\label{Gelfand-Graev}
Define a generalized Gelfand-Graev ${\ggg}$-module associated to $\chi$ by $$Q_\chi=U({\ggg})\otimes_{U({\mmm})}\mathbb{C}_\chi,$$ where $\mathbb{C}_\chi=\mathbb{C}1_\chi$ is a one-dimensional ${\mmm}$-module such that $x.1_\chi=\chi(x)1_\chi$ for  a given nonzero eigenvector $1_\chi\in\bbc_\chi$ and all $x\in{\mmm}$.
\end{defn}

For $k\in\mathbb{Z}_+$, define
 \begin{equation*}
 \begin{array}{llllll}
 \mathbb{Z}_+^k&:=&\{(i_1,\cdots,i_k)\mid i_j\in\mathbb{Z}_+,i_j\geqslant 0\},&
 \Lambda'_k&:=&\{(i_1,\cdots,i_k)\mid i_j\in\{0,1\}\}
 \end{array}
 \end{equation*}with $1\leqslant j\leqslant k$. Given $(\mathbf{a},\mathbf{b},\mathbf{c},\mathbf{d})\in\mathbb{Z}^m_+\times\Lambda'_n\times\mathbb{Z}^s_+\times\Lambda'_t$, let $x^\mathbf{a}y^\mathbf{b}u^\mathbf{c}v^\mathbf{d}$ denote the monomial element $$x_1^{a_1}\cdots x_m^{a_m}y_1^{b_1}\cdots y_n^{b_n}u_1^{c_1}\cdots u_s^{c_s}v_1^{d_1}\cdots v_t^{d_t}$$ in $U({\ggg})$.

\begin{defn}\label{W-C}
Define the finite $W$-superalgebra over $\mathbb{C}$ by $$U({\ggg},e):=(\text{End}_{\ggg}Q_{\chi})^{\text{op}},$$
where $(\text{End}_{\ggg}Q_{\chi})^{\text{op}}$ denotes the opposite algebra of the endomorphism algebra of ${\ggg}$-module $Q_{\chi}$.
\end{defn}

It can be easily concluded by the definition that if two nilpotent elements $E,E'\in{\ggg}_{\bar0}$ are conjugate under the action of $\text{Ad}\,G_{\text{ev}}$, then there is an isomorphism between finite $W$-superalgebras $U({\ggg},E)$ and $U({\ggg},E')$. Therefore, the construction of finite $W$-superalgebras only depends on the adjoint orbit Ad\,$G_{\text{ev}}.e $ of $e$ up to isomorphism.

Let $N_\chi$ denote the $\mathbb{Z}_2$-graded ideal of codimension one in $U({\mmm})$ generated by all $x-\chi(x)$ with $x\in{\mmm}$. Then $Q_\chi\cong U({\ggg})/U({\ggg})N_\chi$ as ${\ggg}$-modules. By construction, the left ideal $I_\chi:=U({\ggg})N_\chi$ of $U({\ggg})$ is a $(U({\ggg}),U({\mmm}))$-bimodule. The fixed point space $(U({\ggg})/I_\chi)^{\text{ad}{\mmm}}$ carries a natural algebra structure given by $$(x+I_\chi)\cdot(y+I_\chi):=(xy+I_\chi)$$ for all $x,y\in U({\ggg})$.

\begin{theorem}\label{W-C2}
There is an isomorphism between $\mathbb{C}$-algebras
\[\begin{array}{lcll}
\phi:&(\text{End}_{\ggg}Q_{\chi})^{\text{op}}&\rightarrow&Q_{\chi}^{\text{ad}\,{\mmm}}\\ &\Theta&\mapsto&\Theta(1_\chi),
\end{array}
\]
where $Q_{\chi}^{\text{ad}\,{\mmm}}$ is the invariant subalgebra of $U({\ggg})/I_\chi\cong Q_{\chi}$ under the adjoint action of ${\mmm}$.
\end{theorem}

\begin{proof}

Since each element in $(\text{End}_{\ggg}Q_{\chi})^{\text{op}}$ is uniquely determined by its effect on $1_\chi\in Q_{\chi}$, it is easy to verify that the mapping $\phi$ keeps the $\mathbb{Z}_2$-graded. The proof is similar to the Lie algebra case \cite[\S3.2]{W}, thus will be omitted.

\end{proof}

\begin{rem}\label{a-d}
We may regard Theorem~\ref{W-C2} as the second definition of finite $W$-superalgebras over $\mathbb{C}$. If one takes $e=0$, then the finite  $W$-superalgebra is simply the enveloping algebra $U({\ggg})$. Hence the finite $W$-superalgebra $U({\ggg},e)$ can be considered as a generalization of universal enveloping algebra $U({\ggg})$.
\end{rem}

\begin{rem}\label{isoWchi} Set $\mathfrak{\widetilde{p}}:={\ppp}\oplus \mathbb{C}\langle u_1,\cdots,u_s, v_1,\cdots,v_{\lfloor\frac{\sfr}{2}\rfloor}\rangle$ to be a subspace of ${\ggg}$. In particular, we have $\mathfrak{\widetilde{p}}={\ppp}$ when the grading ${\ggg}=\bigoplus\limits_{i\in\mathbb{Z}}{\ggg}(i)$ is even (i.e. ${\ggg}(i)=0$ unless $i$ is an even integer). In this case, it follows by the PBW theorem that  $$U({\ggg})=U({\ppp})\oplus I_\chi.$$  Let $\text{Pr}: U({\ggg})\longrightarrow U({\ppp})$ denote the corresponding linear projection.

When the grading ${\ggg}=\bigoplus\limits_{i\in\mathbb{Z}}{\ggg}(i)$ is even, we can define a subalgebra $W_\chi$ of $U({\ppp})$ over $\mathbb{C}$ by
$$W_\chi:=\{u\in U({\ppp})~\mid ~\text{Pr}([x,u])=0~\text{for any}~x\in{\mmm}\},$$
then there is an isomorphism of $\mathbb{C}$-algebras
\begin{align}\label{thethirdeqdef}\varphi:W_\chi&\longrightarrow Q_\chi^{\text{ad}\,{\mmm}}\qquad u\mapsto u(1+I_\chi).
\end{align}
The proof of this isomorphism is straightforward. So $W_\chi$ gives rise to another equivalent definition for the finite $W$-superalgebras over $\mathbb{C}$
when the grading of ${\ggg}$ is even.
\end{rem}

\subsection{Kazhdan filtration}\label{Kfil}

Let ${\ggg}=\bigoplus\limits_{i\in\mathbb{Z}}{\ggg}(i)$ denote the root decomposition of ${\ggg}$ under the action of $\text{ad}\,h$. Define the Kazhdan degree on ${\ggg}$ by declaring $x\in{\ggg}(j)$ is $(j+2)$. Let $\text{F}_iU({\ggg})$ denote the span of monomials $x_1\cdots x_n$ for $n\geqslant 0 $ with $ x_1\in{\ggg}(j_1),\cdots,x_n\in{\ggg}(j_n)$ in $U({\ggg})$ such that $(j_1+2)+\cdots+(j_n+2)\leqslant  i$. Then we get the Kazhdan filtration on $U({\ggg})$:
$$\cdots\subseteq \text{F}_iU({\ggg})\subseteq \text{F}_{i+1}U({\ggg})\subseteq\cdots.$$
The associated graded algebra $\text{gr}\,U({\ggg})$ is the supersymmetric algebra $S({\ggg})$.

The Kazhdan filtration on $U({\ggg})$ induces a filtration on the $\mathbb{Z}_2$-graded left ideal $I_\chi$ and on the quotient $Q_\chi=U({\ggg})/I_\chi$. By the definition $\text{gr}\,Q_\chi=S({\ggg})/\text{gr}\,I_\chi\cong S(\mathfrak{\widetilde{p}})$ is a super-commutative $\mathbb{Z}_+$-graded algebra under the Kazhdan grading. One can easily obtain an induced strictly positive filtration
$$\text{F}_0U({\ggg},e)\subseteq \text{F}_{1}U({\ggg},e)\subseteq\cdots$$ on $U({\ggg},e)$ with $\text{F}_0U({\ggg},e)=\mathbb{C}$.

In virtue of the bilinear form $(\cdot,\cdot)$, we can identify $S({\ggg})$ with the polynomial superalgebra $\mathbb{C}[{\ggg}]$ of regular functions on ${\ggg}$. Then $\text{gr}\,I_\chi$ is the ideal generated by the functions $\{x-\chi(x)\mid x\in{\mmm}\}$, i.e. the left ideal of all functions in $\mathbb{C}[{\ggg}]$ vanishing on $e+{\mmm}^\perp$ of ${\ggg}$. Hence $\text{gr}\,Q_\chi$ can be identified with $\mathbb{C}[e+{\mmm}^\perp]$. Since $S(\mathfrak{\widetilde{p}})\cong\text{gr}\,Q_\chi$, we have $$S(\mathfrak{\widetilde{p}})\cong\mathbb{C}[e+{\mmm}^\perp].$$
If just considering the even part, we can get an isomorphism between $\mathbb{C}$-algebras (in the usual sense, not super)$$S(\mathfrak{\widetilde{p}}_{\bar{0}})\cong\mathbb{C}[e+{\mmm}_{\bar{0}}^\perp],$$ where ${\mmm}^\perp_{\bar{0}}:=\{f\in{\ggg}^*_{\bar{0}}\mid f({\mmm}_{\bar{0}})=0\}$, ${\mmm}^\perp_{\bar{0}}$ is identified with the subalgebra of ${\ggg}_{\bar0}$ by the bilinear form $(\cdot,\cdot)$, and $\mathbb{C}[e+{\mmm}_{\bar{0}}^\perp]$ the regular functions on affine variety $e+{\mmm}_{\bar{0}}^\perp$.

\subsection{Whittaker functor and Skryabin equivalence}

In this part we will establish the connection between the representation category of finite $W$-superalgebras and the category of Whittaker modules. For the Lie algebra case, one refers to \cite[\S5]{W} for more detail.

\begin{defn}\label{Whittaker}
A ${\ggg}$-module $L$ is called a Whittaker module if $a-\chi(a),~\forall a\in{\mmm},$ acts on $L$ locally nilpotently. A Whittaker vector in a Whittaker ${\ggg}$-module $L$ is a vector $v\in L$ which satisfies $(a-\chi(a))v=0,~\forall a\in{\mmm}$.
\end{defn}

Let ${\ggg}\text{-}W\text{mod}^\chi$ denote the category of finitely generated Whittaker ${\ggg}$-modules, and assume all the morphisms are even. Write $$\text{Wh(L)}=\{v\in L\mid (a-\chi(a))v=0,\forall a\in{\mmm}\}$$ the subspace of all Whittaker vectors in $L$.

Recall the second definition of finite $W$-superalgebras (see Theorem~\ref{W-C2}) shows that $U({\ggg},e)\cong(U({\ggg})/I_\chi)^{\text{ad}{\mmm}}$. Denote by $\bar{y}\in U({\ggg})/I_\chi$ the coset associated to $y\in U({\ggg})$.

\begin{theorem}\label{W-n}
(1) Given a Whittaker ${\ggg}$-module $L$ with an action map $\rho$, $\text{Wh}(L)$ is naturally a $U({\ggg},e)$-module by letting $$\bar{y}.v=\rho(y)v$$ for $v\in\text{Wh}(L)$ and $\bar{y}\in U({\ggg})/I_\chi$.

(2) For $M\in U({\ggg},e)$, $Q_\chi\otimes_{U({\ggg},e)}M$ is a Whittaker ${\ggg}$-module by letting $$y.(q\otimes v)=(y.q)\otimes v$$ for $y\in U({\ggg})$ and $q\in Q_\chi,~v\in M$.
\end{theorem}

\begin{proof}
The proof is straightforward and is the same as the Lie algebra case (see e.g. proof of \cite[Lemma 35]{W}).
\end{proof}

Given a ${\ggg}\text{-}W\text{mod}^\chi$ $M$, define $M^{{\mmm}}$ by $$M^{{\mmm}}=\{m\in M \mid  x.m=\chi(x)m ~\text{for all}~x\in{\mmm}\},$$
then $M^{{\mmm}}$ can be considered as a $U({\ggg},e)$-module. The following theorem shows that there exists an equivalence of categories between the ${\ggg}\text{-}W\text{mod}^\chi$ and the $U({\ggg},e)$-modules.

\begin{theorem}\label{functor}
The functor $Q_\chi\otimes_{U({\ggg},e)}-:U({\ggg},e)\text{-mod}\longrightarrow
{\ggg}\text{-}W\text{mod}^\chi$ is an equivalence of categories, with $\text{Wh}:{\ggg}\text{-}W\text{mod}^\chi\longrightarrow U({\ggg},e)\text{-mod}$ as its quasi-inverse.
\end{theorem}

The theorem generalizes the situation of the Lie algebra case by Skryabin. The proof will be omitted for the similarity, and for more specific details one refers to \cite[Therorem 1]{S2} (Note that some of the details need to be improved; see the proof of \cite[Proposition 4.2]{WZ} by Wang-Zhao).

\subsection{Finite $W$-superalgebras modulo $p$ via an admissible ring $A$}\label{modulo p}
From now on to the end of this section, the ground field will be $\bbk$ in turn, an algebraically closed field of prime characteristic $p$. The purpose here is to introduce a ``modular $p$ reduction" version from finite $W$-superalgebras over $\bbc$, and some reduced $W$-superalgebras. The approach is to make use of the admissible ring $A$  (strictly speaking, the prime $p$ is dependent on the choice of $A$), analogous of the arguments in finite $W$-algebras (cf. \cite{P2}).

Given an admissible ring $A$, set $Q_{\chi,A}:=U({\ggg}_A)\otimes_{U({\mmm}_A)}A_\chi$, where $A_\chi=A1_\chi$. It follows by the definition that $Q_{\chi,A}$ is a ${\ggg}_A$-stable $A$-lattice in $Q_{\chi}$ with $$\{x^\mathbf{a}y^\mathbf{b}u^\mathbf{c}v^\mathbf{d}\otimes1_\chi\mid (\mathbf{a},\mathbf{b},\mathbf{c},\mathbf{d})\in\mathbb{Z}^m_+\times\Lambda'_n\times\mathbb{Z}^s_+\times\Lambda'_t\}$$
as a free basis. Let $N_{\chi,A}$ denote the homogeneous ideal of codimension one in $A$-subalgebra $U({\mmm}_A)$ generated by all $x-\chi(x)$ with $x\in{\mmm}_A$. Set $I_{\chi,A}:=U({\ggg}_A)N_{\chi,A}$, the left ideal of $U({\ggg}_A)$. Then $Q_{\chi,A}\cong U({\ggg}_A)/I_{\chi,A}$ as ${\ggg}_A$-modules.

Pick a prime $p\in\Pi(A)$ and denote by ${\bbk}=\overline{\mathbb{F}}_p$ the algebraic closure of $\mathbb{F}_p$. By Definition~\ref{admissible} and the discussion thereafter, we can assume that $(\cdot,\cdot)$ is $A$-valued on ${\ggg}_A$ after enlarging $A$, possibly. Set $\ggg_\bbz$ to be the Chevalley $\bbz$-form as in \S\ref{Chevalleybasis}.
  Then we can have
   $$\ggg_\bbk:=\ggg_\bbz\otimes_{\bbz}\bbk$$
   (note that the $\bbz$-form should be adjusted into $\bbz[a]$-form when $D(2,1,a)$ ($a\notin\bbz$, with $a$ being an algebraic number) is considered, see Remark \ref{fractZa}). The bilinear form $(\cdot,\cdot)$ induces a bilinear form on the Lie superalgebra ${\ggg}_{\bbk}\cong{\ggg}_A\otimes_A{\bbk}$. In the following we still denote this bilinear form by $(\cdot,\cdot)$.

If we denote by $G_{\bbk}$ the algebraic ${\bbk}$-supergroup of distribution algebra $U_{\bbk}=U_\mathbb{Z}\otimes_\mathbb{Z}{\bbk}$, then ${\ggg}_{\bbk}=\text{Lie}(G_{\bbk})$ (c.f. \cite[\S2.2]{SW} and \cite[\S{I}.7.10]{J}). Note that the bilinear form $(\cdot,\cdot)$ is non-degenerate and $\text{Ad}\,(G_{\bbk})_{\text{ev}}$-invariant. For $x\in{\ggg}_A$, set $\bar{x}:=x\otimes1$, an element of ${\ggg}_{\bbk}$. Especially, we have $\bar{e}=e\otimes1,~\bar{f}=f\otimes1$ and $\bar{h}=h\otimes1$ in ${\ggg}_{\bbk}$. Identify $\chi$ with the linear function $(\bar{e},\cdot)$ on ${\ggg}_{\bbk}$. 
Set ${\mmm}_{\bbk}:=\mmm_A\otimes_A{\bbk}$, $\mmm'_{\bbk}:=\mmm'_A\otimes_A{\bbk}$, ${\ppp}_{\bbk}:={\ppp}_A\otimes_A{\bbk}$, $\mathfrak{\widetilde{p}}_{\bbk}:=\mathfrak{\widetilde{p}}_A\otimes_A{\bbk}$.

Recall that ${\ggg}_{\bbk}$ is a restricted Lie superalgebra. For each $\bar x\in({\ggg}_{\bbk})_{\bar{0}}$ it is obvious that $\bar x^p-\bar x^{[p]}\in U({\ggg}_{\bbk})$ is contained in the center of $U({\ggg}_{\bbk})$ by the definition. The subalgebra ${\bbk}\langle \bar x^p-\bar x^{[p]}\mid \bar x\in({\ggg}_{\bbk})_{\bar{0}}\rangle$ of $U({\ggg}_{\bbk})$ is called the $p$-center of $U({\ggg}_{\bbk})$ and denote $Z_p({\ggg}_{\bbk})$ for short. It follows from the PBW theorem of $U({\ggg}_{\bbk})$ that $Z_p({\ggg}_{\bbk})$ is isomorphic to a polynomial algebra (in the usual sense, not super) in $\text{dim}\,({\ggg}_{\bbk})_{\bar{0}}$ variables. For every maximal ideal $H$ of $Z_p({\ggg}_{\bbk})$ there is a unique linear function $\eta=\eta_H\in({\ggg}_{\bbk})_{\bar{0}}^*$ such that $$H=\langle \bar x^p-\bar x^{[p]}-\eta(\bar x)^p\mid \bar x\in({\ggg}_{\bbk})_{\bar{0}}\rangle.$$
Since the Frobenius map of ${\bbk}$ is bijective, this enables us to identify the maximal spectrum $\text{Specm}(Z_p({\ggg}_{\bbk}))$ of $Z_p({\ggg}_{\bbk})$ with $({\ggg}_{\bbk})_{\bar{0}}^*$.

For any $\xi\in({\ggg}_{\bbk})_{\bar{0}}^*$ we write $J_\xi$ the two-sided ideal of $U({\ggg}_{\bbk})$ generated by the even central elements $$\{\bar x^p-\bar x^{[p]}-\xi(\bar x)^p\mid \bar x\in({\ggg}_{\bbk})_{\bar{0}}\}.$$ Then the quotient algebra $U_\xi({\ggg}_{\bbk}):=U({\ggg}_{\bbk})/J_\xi$ is a ${\ggg}_{\bbk}$-module, which is called the reduced enveloping algebra with $p$-character $\xi$. We often regard $\xi\in{\ggg}_{\bbk}^*$ by letting $\xi(({\ggg}_{\bbk})_{\bar{1}})=0$. By the classical theory of Lie superalgebras, we have $$\text{dim}\,U_\xi({\ggg}_{\bbk})=p^{\text{dim}\,({\ggg}_{\bbk})_{\bar{0}}}2^{\text{dim}\,({\ggg}_{\bbk})_{\bar{1}}}.$$

For $i\in\mathbb{Z}$, define the graded subspaces of ${\ggg}_{\bbk}$ over ${\bbk}$ by $${\ggg}_{\bbk}(i):={\ggg}_A(i)\otimes_A{\bbk},\quad {\mmm}_{\bbk}(i):={\mmm}_A(i)\otimes_A{\bbk}.$$
Due to our assumptions on $A$, the elements $\bar{x}_1,\cdots,\bar{x}_l$ and $\bar{y}_1,\cdots,\bar{y}_q$ form bases of the centralizers $({\ggg}^{\bar{e}}_{\bbk})_{\bar{0}}$ and $({\ggg}^{\bar e}_{\bbk})_{\bar{1}}$ of $\bar{e}$ in $(\ggg_\bbk)_\bz$ and $(\ggg_\bbk)_{\bar 1}$, respectively.

It follows from \cite[\S4.1]{WZ} that the subalgebra ${\mmm}_{\bbk}$ is $p$-nilpotent, and the linear function $\chi$ vanishes on the $p$-closure of $[{\mmm}_{\bbk},{\mmm}_{\bbk}]$. More precisely,  we have the following consequence:

\begin{lemma}\label{restricted-m} The Lie superalgebra $\mmm_\bbk$ is a restricted subalgebra of $\ggg_\bbk$.
\end{lemma}

\begin{proof} Recall that under the adjoint action of $(G_\bbk)_\ev$ on $\ggg_\bbk$, each homomorphism of algebraic group $\tau:\bbk^\times\rightarrow (G_\bbk)_\ev$ defines a grading $\ggg_\bbk=\bigoplus_{i\in\bbz}\ggg_\bbk(i;\tau)$ with
$$\ggg_\bbk(i;\tau)=\{v\in\ggg_\bbk\mid \mbox{Ad}(\tau(t))(v)=t^iv\mbox{ for all }t\in \bbk^\times\}.$$
  Note that $\ggg_\bbk$ admits a non-degenerate invariant even bilinear form $(\cdot,\cdot)$.
 As discussed in \cite[\S3.2-\S3.4]{WZ} for the Lie superalgebras of type $\mathfrak{gl}$, $\mathfrak{sl}$ and $\mathfrak{osp}$, and  \cite[\S3.5]{WZ} for exceptional basic Lie superalgebras, one can define a cocharacter $\tau$ as in \cite[Remarks 3.4 and \S3.5]{WZ}, such that the $\tau$-grading of $\ggg_\bbk$ coincides with the grading arising from the admissible structure, i.e.
  $$\ggg_\bbk(i)=\{X\in\ggg_\bbk\mid \text{Ad}\,(\tau(t))(X)=t^iX;~\forall t\in \bbk^\times\}. $$
 Therefore, this graded structure satisfies the property $(\ggg_\bbk(i)_\bz)^{[p]}\subset \ggg_\bbk(pi)_\bz$. From all the discussion above and the definition of $\mmm_\bbk$, the conclusion follows.
\end{proof}

From the above lemma, we can set $$Q_{\chi,{\bbk}}:=U({\ggg}_{\bbk})\otimes_{U({\mmm}_{\bbk})}{\bbk}_\chi,$$ where ${\bbk}_\chi=A_\chi\otimes_{A}{\bbk}={\bbk}1_\chi$. Clearly, ${\bbk}1_\chi$ is a one-dimensional ${\mmm}_{\bbk}$-module with the property $\bar x.1_\chi=\chi(\bar x)1_\chi$ for all $\bar x\in{\mmm}_{\bbk}$ and it is obvious that $Q_{\chi,{\bbk}}\cong Q_{\chi,A}\otimes_A{\bbk}$. Define $N_{\chi,{\bbk}}:=N_{\chi,A}\otimes_A{\bbk}$ and $I_{\chi,{\bbk}}:=I_{\chi,A}\otimes_A{\bbk}$.

\begin{defn}\label{W-k}
Define the finite $W$-superalgebra over ${\bbk}$ by $$U(\ggg_{\bbk},\bar{e}):=(\text{End}_{\ggg_{\bbk}}Q_{\chi,\bbk})^{\text{op}},$$
where $(\text{End}_{\ggg_\bbk} Q_{\chi,{\bbk}})^{\text{op}}$ denotes the opposite algebra of $\text{End}_{\ggg_\bbk}Q_{\chi,{\bbk}}$.
\end{defn}

\subsection{Reduced $W$-superalgebras over $\bbk$}\label{modulo p2} Keep the notations as above. Especially take $A$ to be an admissible ring, and $\bbk$ an algebraically closed field of prime characteristic $p\in \Pi(A)$.
Let ${\ggg}_A^*$ be the $A$-module dual to ${\ggg}_A$, so that ${\ggg}^*={\ggg}_A^*\otimes_A\mathbb{C}, ~{\ggg}_{\bbk}^*={\ggg}_A^*\otimes_A{\bbk}$. Let $({\mmm}_A^\perp)_{\bar{0}}$ denote the set of all linear functions on $({\ggg}_A)_{\bar0}$ vanishing on $({\mmm}_A)_{\bar{0}}$. By the assumptions on $A$,  $({\mmm}_A^\perp)_{\bar{0}}$ is a free $A$-submodule and a direct summand of ${\ggg}^*_A$. Note that
$({\mmm}_A^\perp\otimes_A\mathbb{C})_{\bar{0}}$ and $({\mmm}_A^\perp\otimes_A{\bbk})_{\bar{0}}$ can be identified with the annihilators ${\mmm}^\perp_{\bar{0}}
:=\{f\in{\ggg}^*_{\bar{0}}\mid f({\mmm}_{\bar{0}})=0\}$ and $({\mmm}_{\bbk}^\perp)_{\bar{0}}:=\{f\in({\ggg}_{\bbk})^*_{\bar{0}}\mid f(({\mmm}_{\bbk})_{\bar{0}})=0\}$, respectively.

Given a linear function $\eta\in\chi+({\mmm}_{\bbk}^\perp)_{\bar{0}}$, set ${\ggg}_{\bbk}$-module $$Q_{\chi}^\eta:=Q_{\chi,{\bbk}}/J_\eta Q_{\chi,{\bbk}},$$ where $J_\eta$ is the homogeneous ideal of $U({\ggg}_{\bbk})$ generated by all $\{\bar x^p-\bar x^{[p]}-\eta(\bar x)^p\mid \bar x\in({\ggg}_{\bbk})_{\bar0}\}$. Evidently $Q_{\chi}^\eta$ is a ${\ggg}_{\bbk}$-module with $p$-character $\eta$, and there exists a ${\ggg}_{\bbk}$-module isomorphism
\begin{equation}\label{qqq}
Q_{\chi}^\eta\cong U_\eta({\ggg}_{\bbk})\otimes_{U_\eta({\mmm}_{\bbk})}\bar 1_\chi.
\end{equation}

\begin{defn}\label{reduced W}
Define a reduced $W$-superalgebra $U_\eta({\ggg}_{\bbk},\bar{e})$ associated to $\ggg_\bbk$ and to $p$-character $\eta\in\chi+({\mmm}_{\bbk}^\bot)_{\bar0}$ by $$U_\eta({\ggg}_{\bbk},\bar{e}):=(\text{End}_{{\ggg}_{\bbk}}Q_{\chi}^\eta)^{\text{op}}.$$
\end{defn}

It is immediate by the definition that the restriction of $\eta$ coincides with that of $\chi$ on $({\mmm}_{\bbk})_{\bar{0}}$. If we let $\eta(({\mmm}_{\bbk})_{\bar{1}})=0$, then the ideal of $U({\mmm}_{\bbk})$ generated by all $\{\bar x-\eta(\bar x)\mid \bar x\in{\mmm}_{\bbk}\}$ equals $N_{\chi,{\bbk}}=N_{\chi,A}\otimes_A{\bbk}$, and ${\bbk}_\chi={\bbk}_\eta$ as ${\mmm}_{\bbk}$-modules. Let $N_{{\mmm}_{\bbk}}$ denote the Jacobson radical of $U_\eta(\mmm_\bbk)$, i.e. the ideal of codimensional one in $U_\eta(\mmm_\bbk)$ generated by all $\langle x-\eta(x)\,\mid \,x\in\mmm_\bbk\rangle$, and define $I_{{\mmm}_{\bbk}}:=U_\eta({\ggg}_{\bbk})N_{{\mmm}_{\bbk}}$ be the ideal of $U_\eta({\ggg}_{\bbk})$.

Motivated by \cite[Theorem 2.3(iv)]{P2}, one can expect the following result. It is notable that the first isomorphism in Proposition~\ref{invariant} was proposed in \cite{W} for the special case $\chi=\eta$ (without proof).

\begin{prop}\label{invariant}
There exist isomorphisms of ${\bbk}$-algebras:
\[\begin{array}{lllll}U_\eta({\ggg}_{\bbk},\bar{e})&\longrightarrow&(Q_\chi^\eta)^{\text{ad}\,{\mmm}_{\bbk}}&\longrightarrow&U_\eta({\ggg}_{\bbk})^{\text{ad}\,{\mmm}_{\bbk}}/I_{{\mmm}_{\bbk}}\cap U_\eta({\ggg}_{\bbk})^{\text{ad}\,{\mmm}_{\bbk}}.
\end{array}\]
\end{prop}

For the proof of the above result, we first need the following observation.

\begin{lemma}\label{m_k free}
Let ${\ggg}_{\bbk}$ be a basic Lie superalgebra over ${\bbk}$ and $\chi\in({\ggg}_{\bbk})^*_{\bar0}$. Then for any $\eta\in\chi+(\mathfrak{m}_{\bbk}^\bot)_{\bar{0}}\subseteq({\ggg}_{\bbk})^*_{\bar0}$, every $U_\eta({\ggg}_{\bbk})$-module is $U_\eta(\mathfrak{m}_{\bbk})$-free.
\end{lemma}

\begin{proof} We define a decreasing filtration $\{\ggg_\bbk^{(k)}\mid k\in\bbz\}$ of $\ggg_\bbk$
by setting $\ggg_\bbk^{(k)}:=\sum_{l\geq k}\ggg_\bbk(-l)$. As $\chi|_{\mathfrak{m}_{\bbk}}=\eta|_{\mathfrak{m}_{\bbk}}$,  by a straightforward verification one knows that the conditions in \cite[\S1(b1)-(b6)]{Sk} satisfy with respect to $\eta$. The remaining part for the proof can be fulfilled in the same way as that for \cite[Theorem 1.3]{Sk} and \cite[Proposition 4.2]{WZ}, thus will be omitted here.
\end{proof}

As an immediate corollary of Lemma~\ref{m_k free}, we can conclude that

\begin{lemma}\label{free}
$Q_{\chi}^\eta$ is a free $U_\eta({\mmm}_{\bbk})$-module under the action of ad\,${\mmm}_{\bbk}$.
\end{lemma}

\begin{proof} It is easy to verify that
\begin{equation}\label{[]tomult}
[\bar x,\bar u]\equiv(\bar x-\eta(\bar x))\bar u\quad(\text{mod}\,I_{{\mmm}_{\bbk}})
\end{equation}
for all $\bar x\in{\mmm}_{\bbk}$ and $\bar u\in Q_{\chi}^\eta$.
Since Lemma~\ref{m_k free} shows that every $U_\eta({\ggg}_{\bbk})$-module is $U_\eta({\mmm}_{\bbk})$-free under the action of left multiplication, it is immediate from \eqref{[]tomult} that $U_\eta({\ggg}_{\bbk})$-module $Q_\chi^\eta$ is $U_\eta({\mmm}_{\bbk})$-free under the action of ad\,${\mmm}_{\bbk}$.
\end{proof}

\textsc{Proof} of Proposition \ref{invariant}: With Lemma~\ref{free}, the proof follows the same strategy in the Lie algebra case (cf. \cite[Theorem 2.3(iv)]{P2}).
Take $B$ to be the normalizer of $I_{\mmm_\bbk}$ in $U_\eta(\ggg_\bbk)$, i.e.  $B:=\{\bar u\in U_\eta({\ggg}_{\bbk})\mid I_{{\mmm}_{\bbk}} \bar u\subseteq I_{{\mmm}_{\bbk}}\}=\{\bar u\in U_\eta({\ggg}_{\bbk})\mid [I_{{\mmm}_{\bbk}},\bar u]\subseteq I_{{\mmm}_{\bbk}}\}$. It is immediate by the definition and \eqref{qqq} that $B/I_{{\mmm}_{\bbk}}\cong (U_\eta({\ggg}_{\bbk})/I_{{\mmm}_{\bbk}})^{\text{ad}\,{\mmm}_{\bbk}}\cong(Q_{\chi}^\eta)^{\text{ad}\,{\mmm}_{\bbk}}$ as $\bbk$-algebras.  Moreover,  we can construct the following isomorphisms of ${\bbk}$-algebras:
\begin{equation}\label{psi}
\begin{array}{lcll}
\varphi:& \text{End}_{{\ggg}_{\bbk}}Q_{\chi}^\eta&\longrightarrow& (B/I_{{\mmm}_{\bbk}})^{\text{op}};\\
\psi:& B/I_{{\mmm}_{\bbk}}&\longrightarrow& U_\eta({\ggg}_{\bbk})^{\text{ad}\,{\mmm}_{\bbk}}/I_{{\mmm}_{\bbk}}\cap U_\eta({\ggg}_{\bbk})^{\text{ad}\,{\mmm}_{\bbk}}.
\end{array}
\end{equation}
From all above, the proposition follows. The proof is completed.
\vskip0.3cm

\subsection{The Morita equivalence theorem}

In the concluding subsection, we will introduce the Morita equivalence theorem between reduced enveloping algebras and the corresponding reduced $W$-superalgebras for basic Lie superalgebras, of which the basic version has been provided in \cite{WZ}.
This Morita equivalence will be important to our sequent arguments.  Given a left $U_{\eta}({\ggg}_{\bbk})$-module $M$, define
$$M^{{\mmm}_{\bbk}}:=\{v\in M\mid I_{{\mmm}_{\bbk}}.v=0\}.$$
It follows from Proposition~\ref{invariant} that $U_{\eta}({\ggg}_{\bbk},\bar{e})$ can be identified with the $\bbk$-algebra $U_\eta({\ggg}_{\bbk})^{\text{ad}{\mmm}_{\bbk}}/ U_\eta({\ggg}_{\bbk})^{\text{ad}{\mmm}_{\bbk}}\cap I_{{\mmm}_{\bbk}}$. Therefore, any left $U_\eta({\ggg}_{\bbk})^{\text{ad}{\mmm}_{\bbk}}$-module can be considered as a $U_{\eta}({\ggg}_{\bbk},\bar{e})$-module with the trivial action of the ideal $U_\eta({\ggg}_{\bbk})^{\text{ad}{\mmm}_{\bbk}}\cap I_{{\mmm}_{\bbk}}$.

Now we are in a position to introduce the main result of this part.

\begin{theorem} \label{matrix} The following statements hold.
\begin{itemize}
\item[(1)] (\cite[Theorem 4.4]{WZ} and its generalization)
$Q_\chi^\eta$ is a projective $U_\eta(\ggg_\bbk)$-module
and there is an isomorphism of $\bbk$-algebras
$$U_\eta({\ggg}_{\bbk})\cong\text{Mat}_\delta(U_\eta({\ggg}_{\bbk},\bar{e})),$$
where $\delta:=\dim U_{\eta}({\mmm}_{\bbk})=\dim U_{\chi}({\mmm}_{\bbk})$, and $\text{Mat}_\delta(U_\eta({\ggg}_{\bbk},\bar{e}))$ denotes the matrix algebra of $U_\eta({\ggg}_{\bbk},\bar{e})$.
\item[(2)] \label{reducedfunctors}
The functors
$$U_\eta({\ggg}_{\bbk})\text{-mod}\longrightarrow U_{\eta}({\ggg}_{\bbk},\bar{e})\text{-mod},\qquad M\mapsto M^{{\mmm}_{\bbk}}$$
and
$$U_{\eta}({\ggg}_{\bbk},\bar{e})\text{-mod}\longrightarrow U_\eta({\ggg}_{\bbk})\text{-mod},\qquad V\mapsto U_\eta({\ggg}_{\bbk})\otimes_{U_\eta({\ggg}_{\bbk})^
{\text{ad}\,{\mmm}_{\bbk}}}V$$
are mutually inverse category equivalences.
\end{itemize}
\end{theorem}

\begin{proof} (1) \cite[Theorem 4.4]{WZ} is the special version when the $p$-character $\eta$ is taken to be  $\chi$.  Thanks to Lemma \ref{m_k free}, the arguments in the proof of \cite[Theorem 4.4]{WZ} can be carried out in our case.

(2)  Since every $U_\eta({\ggg}_{\bbk})$-module is $U_\eta({\mmm}_{\bbk})$-free under the action of left multiplication by Lemma \ref{m_k free}, the second statement can be proved in the same way as \cite[Theorem 2.4]{P2} for the Lie algebra case after substituting the discussion in \cite[\S2.2]{P2} for \cite[Proposition 4.2]{WZ}.
\end{proof}

\section{The PBW structure of reduced $W$-superalgebras over $\bbk$}\label{rew}

This section is one of the main parts of the paper. In this section we will study the PBW construction theory of reduced $W$-superalgebra $U_\chi({\ggg}_{\bbk},e)$ associated to a basic Lie superalgebra ${\ggg}_{\bbk}$ over $\bbk$, an algebraically closed field of prime characteristic $p$.

Via lots of computations, we obtain some commutating relations, and also the structure of leading terms of the PBW basis of $U_\chi({\ggg}_{\bbk},e)$, which enables us to establish the PBW theorem (see Theorem \ref{reduced Wg}).  From such a PBW structure of $U_\chi(\ggg_\bbk,e)$, we will understand more on the structure of finite $W$-superalgebras over $\bbc$ presented in the next section.  
By these structural information on reduced $W$-superalgebras,  one can conclude that the construction of $U_\chi({\ggg}_{\bbk},e)$ critically depends on the parity of dim\,${\ggg}_{\bbk}(-1)_{\bar1}$, which is significantly different from the finite $W$-algebra case.

We mainly follow Premet's strategy on finite $W$-algebras \cite[\S3]{P2}, with a few modifications. Compared with the Lie algebra case, one will see that the emergence of odd parts in the Lie superalgebra ${\ggg}_{\bbk}$ makes the situation much more complicated.

\subsection{Notations and conventions}\label{modpconventions} We maintain the notations as the previous section, especially as in \S\ref{modulo p} and \S\ref{modulo p2}. 
Here we still list some conventions frequently used afterwards.
\subsubsection{} Recall that the elements in the Lie superalgebra ${\ggg}_{\bbk}$ are obtained by ``modular $p$ reduction'' from the ones in the $A$-algebra ${\ggg}_A$. Denote by
$\bar{x}=x\otimes1\in{\ggg}_{\bbk}$ for each $x\in{\ggg}_A$.
To ease notation we identify $e,f,h$ with the nilpotent elements $\bar{e}=e\otimes1,~\bar{f}=f\otimes1$ and $\bar{h}=h\otimes1$ in ${\ggg}_{\bbk}$.
\subsubsection{} Take typical basis elements in the corresponding subspaces of $\ggg_\bbk$ as follows
\[\begin{array}{ll}
\bar x_1,\cdots,\bar x_l\in({\ggg}^e_{\bbk})_{\bar{0}},&\qquad \bar x_{l+1},\cdots,\bar x_m\in\bigoplus\limits_{j\geqslant 2}[f,({\ggg}_{\bbk}(j))_{\bar0}];\\
\bar y_1,\cdots,\bar y_q\in({\ggg}^e_{\bbk})_{\bar{1}},&\qquad \bar y_{q+1},\cdots,\bar y_n\in\bigoplus\limits_{j\geqslant 2}[f,({\ggg}_{\bbk}(j))_{\bar1}];\\
\bar u_1,\cdots,\bar u_s\in({\ggg}_{\bbk}(-1)_{\bar{0}}^{\prime})^{\bot}\subseteq{\ggg}_{\bbk}(-1)_{\bar{0}},
&\qquad \bar u_{s+1},\cdots,\bar u_{2s}\in{\ggg}_{\bbk}(-1)_{\bar{0}}';\\
\bar v_1,\cdots,\bar v_t\in({\ggg}_{\bbk}(-1)_{\bar{1}}^{\prime})^{\bot}\subseteq{\ggg}_{\bbk}(-1)_{\bar{1}},
&\qquad \bar v_{t+1},\cdots,\bar v_{\sfr}\in{\ggg}_{\bbk}(-1)_{\bar{1}}'.
\end{array}
\]
Here we still take the notation $t=\lfloor\frac{\sfr}{2}\rfloor$.
\subsubsection{}
Given an element $\bar x\in{\ggg}_{\bbk}(i)$, we denote its weight (with the action of $\text{ad}\,h$) by $\text{wt}(\bar x)=i$. For $k\in\mathbb{Z}_+$, define$$\Lambda_k:=\{(i_1,\cdots,i_k)\mid i_j\in\mathbb{Z}_+,~0\leqslant  i_j\leqslant  p-1\},~\Lambda'_k:=\{(i_1,\cdots,i_k)\mid i_j\in\{0,1\}\}$$ with $1\leqslant j\leqslant k$.
Set $\mathbf{e}_i=(\delta_{i1},\cdots,\delta_{ik})$. For $\mathbf{i}=(i_1,\cdots,i_k)$ in $\Lambda_k\,$ or $\Lambda'_k$, set $|\mathbf{i}|=i_1+\cdots+i_k$.
\subsubsection{}
Given  $\mathbf{a}=(a_1,\cdots,a_m)\in\Lambda_m,~\mathbf{b}=(b_1,\cdots,b_n)\in\Lambda'_n,~\mathbf{c}=(c_1,\cdots,c_s)\in\Lambda_s,~\mathbf{d}=(d_1,\cdots,d_t)\in\Lambda'_t$, define$$\bar x^{\mathbf{a}}\bar y^\mathbf{b}\bar u^\mathbf{c}\bar v^\mathbf{d}:=\bar x_1^{a_1}\cdots \bar x_m^{a_m}\bar y_1^{b_1}\cdots \bar y_n^{b_n}\bar u_1^{c_1}\cdots \bar u_s^{c_s}\bar v_1^{d_1}\cdots \bar v_t^{d_t}.$$
By the discussion preceding Definition~\ref{reduced W}, one knows that there is an isomorphism of ${\ggg}_{\bbk}$-modules $Q_\chi^\chi\cong U_\chi({\ggg}_{\bbk})\otimes_{U_\chi({\mmm}_{\bbk})}\bar 1_\chi$, then all tensor vectors from monomials $\bar x^{\mathbf{a}}\bar y^\mathbf{b}\bar u^\mathbf{c}\bar v^\mathbf{d}\otimes\bar 1_\chi$ form a basis of $Q_\chi^\chi$.
\subsubsection{}
We can assume that all the elements given above are homogeneous under the action of ad\,$h$, i.e. $\bar x_1\in{\ggg}_{\bbk}(k_1)_{\bar{0}},~\cdots,~\bar x_m\in{\ggg}_{\bbk}(k_m)_{\bar{0}},~\bar y_1\in{\ggg}_{\bbk}(k'_1)_{\bar{1}}, \cdots,~\bar y_n\in{\ggg}_{\bbk}(k'_n)_{\bar{1}}$. Define
$$|(\mathbf{a},\mathbf{b},\mathbf{c},\mathbf{d})|_e=\sum_{i=1}^ma_i(k_i+2)+\sum_{i=1}^nb_i(k'_i+2)+\sum_{i=1}^sc_i+\sum_{i=1}^td_i.$$
Say that $\bar x^{\mathbf{a}}\bar y^\mathbf{b}\bar u^\mathbf{c}\bar v^\mathbf{d}$ has $e$-degree $|(\mathbf{a},\mathbf{b},\mathbf{c},\mathbf{d})|_e$ and write $\text{deg}_e(\bar x^{\mathbf{a}}\bar y^\mathbf{b}\bar u^\mathbf{c}\bar v^\mathbf{d})=|(\mathbf{a},\mathbf{b},\mathbf{c},\mathbf{d})|_e$.  The $e$-degree defined above is compatible with the Kazhdan degree in \S\ref{Kfil}. Note that
\begin{equation}\label{dege}
\text{deg}_e(\bar x^{\mathbf{a}}\bar y^\mathbf{b}\bar u^\mathbf{c}\bar v^\mathbf{d})=\text{wt}(\bar x^{\mathbf{a}}\bar y^\mathbf{b}\bar u^\mathbf{c}\bar v^\mathbf{d})+2\text{deg}(\bar x^{\mathbf{a}}\bar y^\mathbf{b}\bar u^\mathbf{c}\bar v^\mathbf{d}),
\end{equation}
where $\text{wt}(\bar x^{\mathbf{a}}\bar y^\mathbf{b}\bar u^\mathbf{c}\bar v^\mathbf{d})=(\sum\limits_{i=1}^mk_ia_i)+(\sum\limits_{i=1}^nk'_ib_i)-|\mathbf{c}|-|\mathbf{d}|$\, and $\text{deg}(\bar x^{\mathbf{a}}\bar y^\mathbf{b}\bar u^\mathbf{c}\bar v^\mathbf{d})=|\mathbf{a}|+|\mathbf{b}|+|\mathbf{c}|+|\mathbf{d}|$ are the weight and the standard degree of $\bar x^{\mathbf{a}}\bar y^\mathbf{b}\bar u^\mathbf{c}\bar v^\mathbf{d}$, respectively.
\subsubsection{}\label{maximalandweight}
Recall that any non-zero element $\bar h\in U_\chi({\ggg}_{\bbk},e)$ is uniquely determined by its value on $\bar h(\bar 1_\chi)\in Q_\chi^\chi$. Write$$\bar h(\bar 1_\chi)=(\sum\limits_{|(\mathbf{a},\mathbf{b},\mathbf{c},\mathbf{d})|_e\leqslant  n(\bar h)}\lambda_{\mathbf{a},\mathbf{b},\mathbf{c},\mathbf{d}}
\bar x^{\mathbf{a}}\bar y^\mathbf{b}\bar u^\mathbf{c}\bar v^\mathbf{d})\otimes\bar 1_\chi,$$ where $n(\bar h)$ is the highest $e$-degree of the terms in the linear expansion of $\bar h(\bar 1 _\chi)$, and $\lambda_{\mathbf{a},\mathbf{b},\mathbf{c},\mathbf{d}}\neq0$ for at least one $(\mathbf{a},\mathbf{b},\mathbf{c},\mathbf{d})$ with $|(\mathbf{a},\mathbf{b},\mathbf{c},\mathbf{d})|_e=n(\bar h)$.

For $k\in\mathbb{Z}_+$, put $\Lambda^{k}_{\bar h}=\{(\mathbf{a},\mathbf{b},\mathbf{c},\mathbf{d})\mid\lambda_{\mathbf{a},\mathbf{b},\mathbf{c},\mathbf{d}}\neq0\mbox{ and }|(\mathbf{a},\mathbf{b},\mathbf{c},\mathbf{d})|_e=k\}$ and set
\begin{align}\label{MaxLambda}
\Lambda^{\text{max}}_{\bar h}:=
\{ (\mathbf{a},\mathbf{b},\mathbf{c},\mathbf{d})\in\Lambda^{n(\bar h)}_{\bar h}\mid
    \text{wt}(\bar x^{\mathbf{a}}\bar y^\mathbf{b}\bar u^\mathbf{c}\bar v^\mathbf{d}) \mbox{ takes its maximum value}\}.
    \end{align}

 This maximum value mentioned in (\ref{MaxLambda}) will be denoted by $N(\bar h)$ (simply by $N$ if the context is clear).

 A  monomial in the ${\bbk}$-span of  $\bar h(\bar 1_\chi)\in Q_\chi^\chi$ both with highest $e$-degree and with maximum weight will be called {\bf the leading term} of $\bar h(\bar 1_\chi)$.

\subsection{Some commutating relations}
Before starting the discussion on the construction theory of reduced $W$-superalgebra $U_\chi({\ggg}_{\bbk},e)$, we will first formulate some lemmas on commutating relations for the elements in the basis of $U_\chi({\ggg}_{\bbk})$.

\begin{lemma}\label{commutative relations k1}
Let $\bar w\in U_\chi({\ggg}_{\bbk})_{i}$ ($i\in\mathbb{Z}_2$) be a $\mathbb{Z}_2$-homogeneous element, then we have
$$\bar w\cdot \bar x^{\mathbf{a}}\bar y^\mathbf{b}\bar u^\mathbf{c}\bar v^\mathbf{d}=\sum_{\mathbf{i}\in\Lambda_{m}}\sum_{\mathbf{j}\in\Lambda_{n}'}\left(\begin{array}{@{\hspace{0pt}}c@{\hspace{0pt}}} \mathbf{a}\\ \mathbf{i}\end{array}\right)\bar x^{\mathbf{a}-\mathbf{i}}\bar y^{\mathbf{b}-\mathbf{j}}\cdot[\bar w\bar x^{\mathbf{i}}\bar y^{\mathbf{j}}]\cdot \bar u^\mathbf{c}\bar v^\mathbf{d},$$
where $\mathbf{a}\choose\mathbf{i}$$=\prod\limits_{l'=1}^m$$a_{l'}\choose i_{l'}$ and $$[\bar w\bar x^{\mathbf{i}}\bar y^{\mathbf{j}}]=k_{1,b_1,j_1}\cdots k_{n,b_n,j_n}(-1)^{|\mathbf{i}|}(\text{ad}\bar y_n)^{j_n}\cdots(\text{ad}\bar y_1)^{j_1}(\text{ad}\bar x_m)^{i_m}\cdots(\text{ad}\bar x_1)^{i_1}(\bar w),$$
in which the coefficients $k_{1,b_1,j_1},\cdots,k_{n,b_n,j_n}\in {\bbk}$ (note that $b_1,\cdots,b_n,j_1,\cdots,j_n\in\{0,1\}$) are defined by
$$k_{t',0,0}=1, k_{t',0,1}=0, k_{t',1,0}=(-1)^{|\bar w|+j_1+\cdots+j_{{t'}-1}}, k_{t',1,1}=(-1)^{|\bar w|+1+j_1+\cdots+j_{{t'}-1}}$$
with $1\leqslant {t'}\leqslant n$ ($j_0$ is interpreted as $0$).
\end{lemma}

\begin{proof}The lemma can be proved by induction. Let $\bar w$ be any $\mathbb{Z}_2$-homogeneous element in $U_\chi({\ggg}_{\bbk})$ and denote its $\mathbb{Z}_2$-degree by $|\bar w|$. For each $\bar y_i\in ({\ggg}_{\bbk})_{\bar1} (1\leqslant i\leqslant n)$, let $R_{\bar y_j}^{i}$ be the $i$-th right multiplication by $\bar y_j(1\leqslant j\leqslant n)$, i.e. $R_{\bar y_j}^{i}(\bar u)=\bar u\bar y_j^{i}$ for any $\bar u\in U_\chi({\ggg}_{\bbk})$. For each $0\leqslant  s'\leqslant  n-1$, by the definition one can conclude that
\[
\begin{array}{ll}
&R_{\bar y_{s'+1}}((\text{ad}\bar y_{s'})^{k_{s'}}(\text{ad}\bar y_{s'-1})^{k_{s'-1}}\cdots(\text{ad}\bar y_{1})^{k_1}(\bar w))\\
=&(-1)^{|\bar w|+1+k_1+\cdots+k_{s'}}\bar y_{s'+1}^0(\text{ad}\bar y_{s'+1})^1(\text{ad}\bar y_{s'})^{k_{s'}}(\text{ad}\bar y_{s'-1})^{k_{s'-1}}\cdots(\text{ad}\bar y_{1})^{k_1}(\bar w)\\
+&(-1)^{|\bar w|+k_1+\cdots+k_{s'}}\bar y_{s'+1}^1(\text{ad}\bar y_{s'+1})^0(\text{ad}\bar y_{s'})^{k_{s'}}(\text{ad}\bar y_{s'-1})^{k_{s'-1}}\cdots(\text{ad}\bar y_{1})^{k_1}(\bar w).
\end{array}
\] For any monomial $\bar x^{\mathbf{a}}\bar y^\mathbf{b}\bar u^\mathbf{c}\bar v^\mathbf{d}$ in the basis of $U_\chi({\ggg}_{\bbk})$, note that all the indices of the odd elements of ${\ggg}_{\bbk}$ (i.e. the indices of $\bar y_i$'s and $\bar v_i$'s) are in the set $\{0,1\}$ by the PBW theorem. Let $j_1,\cdots,j_n\in\{0, 1\}$, and define$$k_{j_i,0,0}:=1,~k_{j_i,0,1}:=0, k_{j_i,1,0}:=(-1)^{|w|+k_1+\cdots+k_{j_i-1}},~k_{j_i,1,1}:=(-1)^{|w|+1+ k_1+\cdots+k_{j_i-1}}$$ for $1\leqslant i\leqslant n$ (where $j_0$ is interpreted as $0$). One can check that
\begin{equation}\label{wdot}
\begin{split}
\bar w\cdot \bar y_1^{j_1}\cdots \bar y_n^{j_n}=&R_{\bar y_n}^{j_n}R_{\bar y_{n-1}}^{j_{n-1}}\cdots R_{\bar y_1}^{j_1}(\bar w)\\
=&R_{\bar y_n}^{j_n}R_{\bar y_{n-1}}^{j_{n-1}}\cdots R_{\bar y_2}^{j_2}(\sum\limits_{i_1=0}^{j_1}k_{1,j_1,i_1}\bar y_1^{j_1-i_1}(\text{ad}\bar y_1)^{i_1}(\bar w))=\cdots\cdots\\
=&\sum\limits_{i_1=0}^{j_1}\sum\limits_{i_2=0}^{j_2}\cdots\sum\limits_{i_n=0}^{j_n}k_{1,j_1,i_1}k_{2,j_2,i_2}\cdots k_{n,j_n,i_n}\bar y_1^{j_1-i_1}\bar y_2^{j_2-i_2}\cdots\bar  y_n^{j_n-i_n}\\
&(\text{ad}\bar y_n)^{i_n}\cdot(\text{ad}\bar y_{n-1})^{i_{n-1}}\cdots(\text{ad}\bar y_{1})^{i_{1}}(\bar w)
\end{split}
\end{equation}
by induction.

Since all the $\bar x_i$'s for $1\leqslant i \leqslant m$ are even elements in ${\ggg}_{\bbk}$, one can obtain that
\begin{equation}\label{wdot-1}
\bar w\cdot \bar x_1^{a_1}\cdots \bar x_m^{a_m}=\sum_{\mathbf{i}\in\Lambda_{m}}(-1)^{|\mathbf{i}|}\left(\begin{array}{@{\hspace{0pt}}c@{\hspace{0pt}}} \mathbf{a}\\ \mathbf{i}\end{array}\right)\bar x^{\mathbf{a}-\mathbf{i}}\cdot(\text{ad}\bar x_m)^{i_m}\cdots (\text{ad}\bar x_1)^{i_1}(\bar w)
\end{equation}
by \cite[\S3.1(2)]{P2}, where $\mathbf{a}\choose\mathbf{i}$$=\prod\limits_{l'=1}^m$$a_{l'}\choose i_{l'}$.

Write $[\bar w\bar x^{\mathbf{i}}]=(-1)^{|\mathbf{i}|}(\text{ad}\bar x_m)^{i_m}\cdots (\text{ad}\bar x_1)^{i_1}(\bar w)$.
Since all the elements $\bar x_1,\cdots,\bar x_m$ are even, $[\bar w\bar x^{\mathbf{i}}]$ is also a $\mathbb{Z}_2$-homogeneous element with the same parity as $\bar w$. It can be inferred from \eqref{wdot} and \eqref{wdot-1} that
\begin{equation}\label{wxyuv}
\begin{split}
\bar w\cdot \bar x^{\mathbf{a}}\bar y^\mathbf{b}\bar u^\mathbf{c}\bar v^\mathbf{d}=&\sum\limits_{\mathbf{i}\in\Lambda_{m}}\left(\begin{array}{@{\hspace{0pt}}c@{\hspace{0pt}}} \mathbf{a}\\ \mathbf{i}\end{array}\right)\bar x^{\mathbf{a}-\mathbf{i}}\cdot[\bar w\bar x^{\mathbf{i}}]\cdot
\bar y^{\mathbf{b}}\cdot \bar u^\mathbf{c}\bar v^\mathbf{d}\\
=&\sum\limits_{\mathbf{i}\in\Lambda_{m}}\sum\limits_{j_1=0}^{b_1}\cdots\sum\limits_{j_n=0}^{b_n}
\left(\begin{array}{@{\hspace{0pt}}c@{\hspace{0pt}}} \mathbf{a}\\ \mathbf{i}\end{array}\right)\bar x^{\mathbf{a}-\mathbf{i}}k_{1,b_1,j_1}\cdots k_{n,b_n,j_n}\bar y_1^{b_1-j_1}\bar y_2^{b_2-j_2}\cdots\\
&\bar y_n^{b_n-j_n}(\text{ad}\bar y_n)^{j_n}(\text{ad}\bar y_{n-1})^{j_{n-1}}\cdots(\text{ad}\bar y_{1})^{j_{1}}([\bar w\bar x^{\mathbf{i}}])\cdot\bar  u^\mathbf{c}\bar v^\mathbf{d}\\
=&\sum\limits_{\mathbf{i}\in\Lambda_{m}}\sum\limits_{j_1=0}^{b_1}\cdots\sum\limits_{j_n=0}^{b_n}(-1)^{|\mathbf{i}|}\left(\begin{array}{@{\hspace{0pt}}c@{\hspace{0pt}}} \mathbf{a}\\ \mathbf{i}\end{array}\right)k_{1,b_1,j_1}\cdots k_{n,b_n,j_n}\bar x^{\mathbf{a}-\mathbf{i}}\bar y^{\mathbf{b}-\mathbf{j}}\\
&(\text{ad}\bar y_n)^{j_n}\cdots(\text{ad}\bar y_{1})^{j_{1}}(\text{ad}\bar x_m)^{i_m}\cdots(\text{ad}\bar x_1)^{i_1}(\bar w)\cdot \bar u^\mathbf{c}\bar v^\mathbf{d},
\end{split}
\end{equation}
where the coefficients $k_{1,b_1,j_1},\cdots, k_{n,b_n,j_n}\in{\bbk}$ in \eqref{wxyuv} are defined by:
$$k_{t',0,0}=1, k_{t',0,1}=0, k_{t',1,0}=(-1)^{|\bar w|+j_1+\cdots+j_{{t'}-1}}, k_{t',1,1}=(-1)^{|\bar w|+1+j_1+\cdots+j_{{t'}-1}}$$for  $1\leqslant {t'}\leqslant n$, and $j_0$ is interpreted as $0$.

Set$$[\bar w\bar x^{\mathbf{i}}\bar y^{\mathbf{j}}]=k_{1,b_1,j_1}\cdots k_{n,b_n,j_n}(-1)^{|\mathbf{i}|}(\text{ad}\bar y_n)^{j_n}\cdots(\text{ad}\bar y_1)^{j_1}(\text{ad}\bar x_m)^{i_m}\cdots(\text{ad}\bar x_1)^{i_1}(\bar w),$$
then \eqref{wxyuv} can be written as
$$\bar w\cdot \bar x^{\mathbf{a}}\bar y^\mathbf{b}\bar u^\mathbf{c}\bar v^\mathbf{d}=\sum_{\mathbf{i}\in\Lambda_{m}}\sum_{\mathbf{j}\in\Lambda_{n}'}\left(\begin{array}{@{\hspace{0pt}}c@{\hspace{0pt}}} \mathbf{a}\\ \mathbf{i}\end{array}\right)\bar x^{\mathbf{a}-\mathbf{i}}\bar y^{\mathbf{b}-\mathbf{j}}\cdot[\bar w\bar x^{\mathbf{i}}\bar y^{\mathbf{j}}]\cdot \bar u^\mathbf{c}\bar v^\mathbf{d}.$$
\end{proof}

Let $\rho_\chi$ denote the natural representation of $U_\chi({\ggg}_{\bbk})$ in $\text{End}_{\bbk}Q_\chi^\chi$. We can get the following result:

\begin{lemma}\label{commutative relations k2}
Let $(\mathbf{a},\mathbf{b},\mathbf{c},\mathbf{d}),~(\mathbf{a}',\mathbf{b}',\mathbf{c}',\mathbf{d}')\in\Lambda_m\times
\Lambda'_n\times\Lambda_s\times\Lambda'_t$ with  $|(\mathbf{a},\mathbf{b},\mathbf{c},\mathbf{d})|_e=A,~|(\mathbf{a}',\mathbf{b}',\mathbf{c}',\mathbf{d}')|_e=B$. Then
\[\begin{array}{ccl}(\rho_\chi(\bar x^{\mathbf{a}}\bar y^\mathbf{b}\bar u^\mathbf{c}\bar v^\mathbf{d}))(\bar x^{\mathbf{a}'}\bar y^{\mathbf{b}'}\bar u^{\mathbf{c}'}\bar v^{\mathbf{d}'}\otimes\bar 1_\chi)&=&
(K\bar x^{\mathbf{a}+\mathbf{a}'}\bar y^{\mathbf{b}+\mathbf{b}'}\bar u^{\mathbf{c}+\mathbf{c}'}\bar v^{\mathbf{d}+\mathbf{d}'}+\text{terms of}~e\text{-degree}\\ &&\leqslant A+B-2)\otimes\bar 1_\chi,
\end{array}\]
where the coefficient $K\in \{-1,0,1\}\subseteq \bbk$ is subject to the following items:
\begin{itemize}
\item[(1)] $K=0$ if and only if $(\mathbf{a}+\mathbf{a}',\mathbf{b}+\mathbf{b}',\mathbf{c}+\mathbf{c}',\mathbf{d}+\mathbf{d}')\notin\Lambda_m\times\Lambda'_n\times\Lambda_s\times\Lambda'_t$;
\item[(2)] If $K\neq0$, then one can further determine $K$ as follows. For the sequence $(\mathbf{b},\mathbf{d},\mathbf{b}',\mathbf{d}')$, denote it by $(z_1,\cdots,z_{2n+2t})$. Define
    $$\epsilon(z_i)=\begin{cases} i &\mbox{ if }z_i=1;\cr
     \mbox{naught} &\mbox{ if }z_i= 0.
       \end{cases}$$
       Since one can get a new sequence$$(b_{1},b'_{1},b_{2},b'_{2},\cdots,b_{n},b'_{n},d_{1},d'_{1},d_{2},d'_{2}\cdots,d_{t},d'_{t})$$from $(\mathbf{b},\mathbf{d},\mathbf{b}',\mathbf{d}')$
by transposition, let $\tau(\mathbf{b},\mathbf{d},\mathbf{b}',\mathbf{d}')$ denote the times by which we do the transpositions changing $ (\epsilon(z_1),\cdots,\epsilon(z_{2n+2t}))$ to
$$(\epsilon(b_{1}),\epsilon(b'_{1}),\epsilon(b_{2}),\epsilon(b'_{2}),\cdots,\epsilon(b_{n}),\epsilon(b'_{n}),\epsilon(d_{1}),\epsilon(d'_{1}),\epsilon(d_{2}),\epsilon(d'_{2})\cdots,\epsilon(d_{t}),\epsilon(d'_{t})),$$
  then $K=(-1)^{\tau(\mathbf{b},\mathbf{d},\mathbf{b}',\mathbf{d}')}$.
\end{itemize}
\end{lemma}

\begin{proof}
(1) Firstly suppose that $(\mathbf{a},\mathbf{b},\mathbf{c})=\mathbf{0}$ and $|\mathbf{d}|=1$, so that $A=1$. Then $\bar v^\mathbf{d}=\bar v_S$ for some $1\leqslant  S\leqslant  t$. Applying Lemma~\ref{commutative relations k1} one obtains
\begin{equation}\label{vs}
\begin{split}
(\rho_\chi(\bar v_S))(\bar x^{\mathbf{a}'}\bar y^{\mathbf{b}'}\bar u^{\mathbf{c}'}\bar v^{\mathbf{d}'}\otimes\bar 1_\chi)
=&((-1)^{|\mathbf{b}'|}\bar x^{\mathbf{a}'}\bar y^{\mathbf{b}'}
\cdot\rho_\chi(\bar v_S)\bar u^{\mathbf{c}'}\bar v^{\mathbf{d}'}+\sum\limits_{(\mathbf{i},\mathbf{j})\neq\mathbf{0}}\alpha_{\mathbf{i}\mathbf{j}}
\bar x^{\mathbf{a}'-\mathbf{i}}\\
&\cdot \bar y^{\mathbf{b}'-\mathbf{j}}\cdot\rho_\chi([\bar v_S\bar x^{\mathbf{i}}\bar y^{\mathbf{j}}])\cdot \bar u^{\mathbf{c}'}\bar v^{\mathbf{d}'})\otimes\bar 1_\chi
\end{split}
\end{equation}
for some $\alpha_{\mathbf{i}\mathbf{j}}\in{\bbk}$. Since $\rho_\chi({\mmm}_{\bbk})$ stabilises the line ${\bbk}\bar 1_\chi$, the first summand on the right equals $(-1)^{|\mathbf{b}'|+\sum\limits_{l=1}^{S-1}d'_l}\bar x^{\mathbf{a}'}\bar y^{\mathbf{b}'}\bar u^{\mathbf{c}'}\bar v^{\mathbf{d}'+\mathbf{e}_S}\otimes\bar 1_\chi$ (where $\sum\limits_{l=1}^{S-1}d'_l$ is interpreted as $0$ for the case $S=1$) modulo terms of lower $e$-degree in \eqref{vs} (if $d'_S+1=2$, then $\bar v^{\mathbf{d}'+\mathbf{e}_S}$ is interpreted as $0$). For the second summand on the right of \eqref{vs}, we have:

(i) suppose $(\mathbf{i},\mathbf{j})\neq\mathbf{0}$ is such that $\text{wt}([\bar v_S\bar x^{\mathbf{i}}\bar y^{\mathbf{j}}])\leqslant -1$. Then $\rho_\chi([\bar v_S\bar x^{\mathbf{i}}\bar y^{\mathbf{j}}])\bar u^{\mathbf{c}'}\bar v^{\mathbf{d}'}\otimes\bar 1_\chi$ is a linear combination of $\bar u^\mathbf{f}\bar v^\mathbf{g}$ with $|\mathbf{f}|+|\mathbf{g}|\leqslant |\mathbf{c}'|+|\mathbf{d}'|+1$ as $\rho_\chi(\mathfrak{m_{\bbk}})$ stabilises the line ${\bbk}\bar1_\chi$. As a consequence, $\bar x^{\mathbf{a}'-\mathbf{i}}\bar y^{\mathbf{b}'-\mathbf{j}}\cdot\rho_\chi([\bar v_S\bar x^{\mathbf{i}}y^{\mathbf{j}}])\bar u^{\mathbf{c}'}\bar v^{\mathbf{d}'}$ is a linear combination of $
\bar x^{\mathbf{a}'-\mathbf{i}}\bar y^{\mathbf{b}'-\mathbf{j}}\bar u^\mathbf{f}\bar v^\mathbf{g}$. Since $(\mathbf{i},\mathbf{j})\neq\mathbf{0}$, and the weights $k_{s'}$ and $k'_{t'}$ of elements $\bar x_{s'}$ and $\bar y_{t'}$ for each $1\leqslant s'\leqslant m$ and $1\leqslant t'\leqslant n$ are all non-negative integers, then
\[\begin{array}{ccl}
\text{deg}_e(\bar x^{\mathbf{a}'-\mathbf{i}}\bar y^{\mathbf{b}'-\mathbf{j}}\bar u^\mathbf{f}\bar v^\mathbf{g})&=&\sum\limits_{s'=1}^m(a'_{s'}-i_{s'})(k_{s'}+2)+\sum\limits_{t'=1}^n(b'_{t'}-j_{t'})(k'_{t'}+2)
+|\mathbf{f}|+|\mathbf{g}|\\
&\leqslant& \sum\limits_{s'=1}^ma'_{s'}(k_{s'}+2)+\sum\limits_{t'=1}^nb'_{t'}(k'_{t'}+2)+(|\mathbf{f}|+|\mathbf{g}|-2|\mathbf{i}|-2|\mathbf{j}|)
\\&\leqslant& \sum\limits_{s'=1}^ma'_{s'}(k_{s'}+2)+\sum\limits_{t'=1}^nb'_{t'}(k'_{t'}+2)+(|\mathbf{c}'|+|\mathbf{d}'|+1-2|\mathbf{i}|\\
&&-2|\mathbf{j}|)
\\&\leqslant &A+B-2.
\end{array}\]

(ii) Suppose $(\mathbf{i},\mathbf{j})\neq\mathbf{0}$ is such that $\text{wt}([\bar v_S\bar x^{\mathbf{i}}y^{\mathbf{j}}])\geqslant 0$. Since ${\ggg}_{\bbk}=\bigoplus\limits_{i\in\mathbb{Z}}{\ggg}_{\bbk}(i)$ is the Dynkin grading of ${\ggg}_{\bbk}$, the image of $\mathfrak{p}_{\bbk}$ is still in $\mathfrak{p}_{\bbk}$ under the action of $\text{ad}h$. This implies that $\bar x^{\mathbf{a}'-\mathbf{i}}\bar y^{\mathbf{b}'-\mathbf{j}}\cdot[\bar v_S\bar x^{\mathbf{i}}\bar y^{\mathbf{j}}]$ is a linear combination of $\bar x^\mathbf{f}\bar y^\mathbf{g}$ with $\text{wt}(\bar x^\mathbf{f}\bar y^\mathbf{g})=\text{wt}(\bar x^{\mathbf{a}'}\bar y^{\mathbf{b}'})-1$ and $|\mathbf{f}|+|\mathbf{g}|\leqslant
|\mathbf{a}'|+|\mathbf{b}'|-|\mathbf{i}|-|\mathbf{j}|+1$. Therefore, $\bar x^{\mathbf{a}'-\mathbf{i}}\bar y^{\mathbf{b}'-\mathbf{j}}\cdot\rho_\chi([\bar v_S\bar x^{\mathbf{i}}\bar y^{\mathbf{j}}])\bar u^{\mathbf{c}'}\bar v^{\mathbf{d}'}$ is a linear combination of $\bar x^{\mathbf{f}}\bar y^{\mathbf{g}}\bar u^{\mathbf{c}'}\bar v^{\mathbf{d}'}$ with
\[\begin{array}{ccl}
\text{deg}_e(\bar x^{\mathbf{f}}\bar y^{\mathbf{g}}\bar u^{\mathbf{c}'}\bar v^{\mathbf{d}'})&=&\text{wt}(\bar x^{\mathbf{f}}\bar y^{\mathbf{g}}\bar u^{\mathbf{c}'}\bar v^{\mathbf{d}'})+
2\text{deg}(\bar x^{\mathbf{f}}\bar y^{\mathbf{g}}\bar u^{\mathbf{c}'}\bar v^{\mathbf{d}'})\\
&=&\text{wt}(\bar x^{\mathbf{f}}\bar y^{\mathbf{g}})-(|\mathbf{c}'|+|\mathbf{d}'|)
+2(|\mathbf{f}|+|\mathbf{g}|+|\mathbf{c}'|+|\mathbf{d}'|)\\
&=&\text{wt}(\bar x^{\mathbf{a}'}\bar y^{\mathbf{b}'})-1+2(|\mathbf{f}|+|\mathbf{g}|)+(|\mathbf{c}'|+|\mathbf{d}'|)\\
&\leqslant& \text{wt}(\bar x^{\mathbf{a}'}\bar y^{\mathbf{b}'})-1+2(|\mathbf{a}'|+|\mathbf{b}'|-|\mathbf{i}|-|\mathbf{j}|+1)+(|\mathbf{c}'|+|\mathbf{d}'|)\\
&=&\text{wt}(\bar x^{\mathbf{a}'}\bar y^{\mathbf{b}'}\bar u^{\mathbf{c}'}\bar v^{\mathbf{d}'})+2\text{deg}(\bar x^{\mathbf{a}'}\bar y^{\mathbf{b}'}\bar u^{\mathbf{c}'}\bar v^{\mathbf{d}'})-2(|\mathbf{i}|+|\mathbf{j}|)+1\\
&\leqslant &A+B-2.
\end{array}\]

By (i) and (ii) we have
\[\begin{array}{ccl}
(\rho_\chi(\bar v_S))(\bar x^{\mathbf{a}'}\bar y^{\mathbf{b}'}\bar u^{\mathbf{c}'}\bar v^{\mathbf{d}'}\otimes\bar 1_\chi)&=&((-1)^{|\mathbf{b}'|+\sum\limits_{l=1}^{S-1}d'_l}\bar x^{\mathbf{a}'}\bar y^{\mathbf{b}'}\bar u^{\mathbf{c}'}\bar v^{\mathbf{d}'+\mathbf{e}_S}+\text{terms of}~e\text{-degree}\\
&&\leqslant A+B-2)\otimes\bar 1_\chi.
\end{array}\]

(2) Induction on $|\mathbf{d}|=|(\mathbf{0},\mathbf{0},\mathbf{0},\mathbf{d})|_e=A$ now shows that
$$(\rho_\chi(\bar v^\mathbf{d}))(\bar x^{\mathbf{a}'}\bar y^{\mathbf{b}'}\bar u^{\mathbf{c}'}\bar v^{\mathbf{d}'}\otimes\bar 1_\chi)=(K^{''}\bar x^{\mathbf{a}'}\bar y^{\mathbf{b}'}\bar u^{\mathbf{c}'}\bar v^{\mathbf{d}+\mathbf{d}'}+\text{terms of}~e\text{-degree}\leqslant A+B-2)\otimes\bar 1_\chi,$$
where the coefficient $K^{''}\in{\bbk}$ is a power of $-1$. If $\mathbf{d}+\mathbf{d}'\notin\Lambda'_t$, then set $K^{''}=0$.

(3) Notice that $\bar u^\mathbf{c}$ is a product of even elements in ${\ggg}_{\bbk}$. Combining the formula displayed in step (2) and discussing in the same way as (1) and (2), it is now easy to derive that
\[\begin{array}{ccl}(\rho_\chi(\bar u^\mathbf{c}\bar v^\mathbf{d}))(\bar x^{\mathbf{a}'}\bar y^{\mathbf{b}'}\bar u^{\mathbf{c}'}\bar v^{\mathbf{d}'}\otimes\bar 1_\chi)&=&(K^{''}\bar x^{\mathbf{a}'}\bar y^{\mathbf{b}'}\bar u^{\mathbf{c}+\mathbf{c}'}\bar v^{\mathbf{d}+\mathbf{d}'}+\text{terms
of}~e\text{-degree}\\
&&\leqslant A+B-2)\otimes\bar 1_\chi.
\end{array}\]
If $(\mathbf{c}+\mathbf{c}',\mathbf{d}+\mathbf{d}')\notin \Lambda_s\times\Lambda'_t$, then the first summand on the right hand is interpreted as $0$.

(4) Since the image of $\mathfrak{p}_{\bbk}$ is still in $\mathfrak{p}_{\bbk}$ under the action of $\text{ad}\,h$, the PBW theorem for $U_\chi(\mathfrak{p}_{\bbk})$ implies that
$$\bar x^{\mathbf{a}}\bar y^\mathbf{b}\cdot \bar x^{\mathbf{a}'}\bar y^{\mathbf{b}'}=K^{'''}\bar x^{\mathbf{a}+\mathbf{a}'}\bar y^{\mathbf{b}+\mathbf{b}'}+\sum\limits_{|\mathbf{i}|+|\mathbf{j}|<
|\mathbf{a}|+|\mathbf{a}'|+|\mathbf{b}|+|\mathbf{b}'|}\beta_{\mathbf{i},\mathbf{j}}\bar x^{\mathbf{i}}\bar y^{\mathbf{j}},$$
where $K^{'''}\in{\bbk}$ is a power of $-1$. If $(\mathbf{a}+\mathbf{a}',\mathbf{b}+\mathbf{b}')\notin \Lambda_m\times\Lambda'_n$, then set $K^{'''}=0$, and $\beta_{\mathbf{i},\mathbf{j}}=0$ unless $\text{wt}(\bar x^{\mathbf{i}}\bar y^{\mathbf{j}})=\text{wt}(\bar x^{\mathbf{a}}\bar y^\mathbf{b})+\text{wt}(\bar x^{\mathbf{a}'}\bar y^{\mathbf{b}'})$.

(5) It can be inferred from (3) and (4) that
\[\begin{array}{ccl}(\rho_\chi(\bar x^{\mathbf{a}}\bar y^\mathbf{b}\bar u^\mathbf{c}\bar v^\mathbf{d}))(\bar x^{\mathbf{a}'}\bar y^{\mathbf{b}'}\bar u^{\mathbf{c}'}\bar v^{\mathbf{d}'}\otimes\bar 1_\chi)&=&
(K'\bar x^{\mathbf{a}+\mathbf{a}'}\bar y^{\mathbf{b}+\mathbf{b}'}\bar u^{\mathbf{c}+\mathbf{c}'}\bar v^{\mathbf{d}+\mathbf{d}'}+\text{terms of}~e\text{-degree}\\&&\leqslant A+B-2)\otimes\bar1_\chi,
\end{array}\]
where $K'\in{\bbk}$ is a power of $-1$. If $(\mathbf{a}+\mathbf{a}',\mathbf{b}+\mathbf{b}',\mathbf{c}+\mathbf{c}',\mathbf{d}+\mathbf{d}')\notin\Lambda_m\times\Lambda'_n\times\Lambda_s\times\Lambda'_t$, set $K'=0$.

(6) Finally we will discuss the value of $K'$ in (5). Given any homogeneous elements $\bar u,\bar v\in{\ggg}_{\bbk}$, it can be deduced from the definition of $e$-degree that
\[\bar u\bar v\equiv\left\{\begin{array}{ll}\bar v\bar u&\text{if at least one of}~\bar u, \bar v~\text{is even;}\\-\bar v\bar u&\text{if}~\bar u~\text{and}~\bar v~\text{are both odd}\end{array}\right.\]
modolo terms of lower $e$-degree in $U_\chi({\ggg}_{\bbk})$. Therefore, in order to determine the value of $K'$, one just needs to deal with the odd elements. For each case we will consider separately:

(i) by (1)-(5), we know $K'=0$ if and only if $(\mathbf{a}+\mathbf{a}',\mathbf{b}+\mathbf{b}',\mathbf{c}+\mathbf{c}',\mathbf{d}+\mathbf{d}')\notin\Lambda_m\times\Lambda'_n\times\Lambda_s\times\Lambda'_t$.

(ii) If $(\mathbf{a}+\mathbf{a}',\mathbf{b}+\mathbf{b}',\mathbf{c}+\mathbf{c}',\mathbf{d}+\mathbf{d}')\in\Lambda_m\times\Lambda'_n\times\Lambda_s\times\Lambda'_t$, it follows from the definition of $\Lambda'_n$ and $\Lambda'_t$ that each entry in the sequence $(\mathbf{b},\mathbf{d},\mathbf{b}',\mathbf{d}')$ is in the set $\{0,1\}$. By earlier discussion one knows that if two odd elements exchange their positions in the product of $U_\chi({\ggg}_{\bbk})$, there is a sign change modulo terms of lower $e$-degree. After one deleting all the zero entries in $(\mathbf{b},\mathbf{d},\mathbf{b}',\mathbf{d}')$, since the position exchange for the odd elements of ${\ggg}_{\bbk}$ in $U_\chi({\ggg}_{\bbk})$ of (5) corresponds to the transpositions from the sequence $(b_{1},b_{2},\cdots,b_{n},d_{1},d_{2},\cdots,d_{t},b'_{1},b'_{2},\cdots,b'_{n},d'_{1},d'_{2},\cdots,d'_{t})$ to
$$(b_{1},b'_{1},b_{2},b'_{2},\cdots,b_{n},b'_{n},d_{1},d'_{1},d_{2},d'_{2}\cdots,d_{t},d'_{t}),$$
 it follows that the constant $K'$ in step (5) coincides with the constant $K$ defined in the lemma.
\end{proof}

\subsection{The leading-terms lemma}

\begin{lemma} \label{hw}
Let $\bar h\in U_\chi({\ggg}_{\bbk},e)\backslash\{0\}$ and $(\mathbf{a},\mathbf{b},\mathbf{c},\mathbf{d})\in \Lambda^{\text{max}}_{\bar h}$. Then $\mathbf{c}=\mathbf{0}$ and $\mathbf{a}\in\Lambda_{l}\times\{\mathbf{0}\},~\mathbf{b}\in\Lambda'_{q}\times\{\mathbf{0}\}$. Moreover, the sequence $\mathbf{d}$ satisfies

(1) $\mathbf{d}=\mathbf{0}$ when $\sfr=\dim\ggg_{\bbk}(-1)_{\bar{1}}$ is even;

(2) $\mathbf{d}\in\{\mathbf{0}\}_{\frac{\sfr-1}{2}}\times\Lambda'_{1}$ when $\sfr=\dim\ggg_{\bbk}(-1)_{\bar{1}}$  is odd.
\end{lemma}

\begin{proof} We prove the lemma by reductio ad absurdum. This is to say, under the assumption that  one of the following situations happens:

(I)  if $\sfr$ is even, then $$(a_{l+1},\cdots,a_m,b_{q+1},\cdots,b_n,c_1,\cdots,c_s,d_1\cdots,d_{\frac{\sfr}{2}})\neq\{\mathbf{0}\};$$

(II) if $\sfr$ is odd, then $$(a_{l+1},\cdots,a_m,b_{q+1},\cdots,b_n,c_1,\cdots,c_s,d_1\cdots,d_{\frac{\sfr-1}{2}})\neq\{\mathbf{0}\},$$
some contradiction comes out. Our arguments proceed by steps.
 \vskip0.1cm
 Step 1: We begin  with interpreting the above assumption into the following   precise circumstance.

(S1) if $a_k\neq0$ for some $k>l$ set $\bar x_k\in{\ggg}_{\bbk}(n_k)_{\bar0}$. Since $\bar x_k\notin({\ggg}_{\bbk}^e)_{\bar{0}}$ and the bilinear form $(\cdot,\cdot)$ is non-degenerate, there is $\bar w=\bar w_k\in{\ggg}_{\bbk}(-n_k-2)_{\bar{0}}$ such that $\chi([\bar w_k,\bar x_i])=\delta_{ki}$ for all $i>l$.

(S2) if all $a_i$'s are zero for $i>l$, and there is $b_k\neq0$ for some $k>q$, then set $\bar y_k\in{\ggg}_{\bbk}(n'_k)_{\bar1}$. Since $\bar y_k\notin({\ggg}_{\bbk}^e)_{\bar{1}}$ and the bilinear form $(\cdot,\cdot)$ is non-degenerate, there is $\bar w=\bar w'_k\in{\ggg}_{\bbk}(-n'_k-2)_{\bar{1}}$ such that $\chi([\bar w'_k,\bar y_i])=\delta_{ki}$ for all $i>q$.

(S3) if all $a_i$'s and $b_j$'s are zero for $i>l$ and $j>q$, and there is $\bar u_k\neq0$ for some $1\leqslant  k\leqslant  s$, choose $\bar w=\bar z_k\in{\ggg}_{\bbk}(-1)'_{\bar{0}}$ such that $\langle \bar z_k,\bar u_i\rangle=\delta_{ki}$ for all $1\leqslant  k\leqslant  s$.

(S4) if all $a_i$'s, $b_j$'s and $c_k$'s are zero for $i>l$, $j>q$ and $1\leqslant k\leqslant s$, respectively, then

(S4-a) when dim\,${\ggg}_{\bbk}(-1)_{\bar{1}}$ is even, there exists $\bar w=\bar z'_k\in{\ggg}_{\bbk}(-1)'_{\bar{1}}$ ($1\leqslant k\leqslant \frac{\sfr}{2}$) such that $\langle \bar z'_k,\bar v_i\rangle=\delta_{ki}$ for all $1\leqslant i\leqslant \frac{\sfr}{2}$;

(S4-b) when dim\,${\ggg}_{\bbk}(-1)_{\bar{1}}$ is odd, there exists $\bar w=\bar z^{''}_k\in{\ggg}_{\bbk}(-1)'_{\bar{1}}$ ($1\leqslant  k\leqslant \frac{\sfr-1}{2}$) such that $\langle \bar z^{''}_k,\bar v_i\rangle=\delta_{ki}$ for all $1\leqslant  i\leqslant \frac{\sfr+1}{2}$.

Take $\bar w\in\mmm_\bbk$ as in the above items, and write $\nu:=\text{wt}(\bar w)$.
\vskip0.1cm
   Step 2: Next let $(\mathbf{a},\mathbf{b},\mathbf{c},\mathbf{d})\in\Lambda^{d'}_{\bar h}$ where $d'\in\mathbb{Z}_+$. Under the assumption,  $\bar w$ is $\mathbb{Z}_2$-homogeneous. It is immediate from Lemma~\ref{commutative relations k1} and the definition of $Q_\chi^\chi$ that
\begin{equation}\label{rhow}
(\rho_\chi(\bar w)) (\bar x^{\mathbf{a}}\bar y^\mathbf{b}\bar u^\mathbf{c}\bar v^\mathbf{d}\otimes\bar 1_\chi)=\sum_{\mathbf{i}\in\Lambda_{m}}\sum_{\mathbf{j}\in\Lambda'_{n}}\left(\begin{array}{@{\hspace{0pt}}c@{\hspace{0pt}}} \mathbf{a}\\ \mathbf{i}\end{array}\right)\bar x^{\mathbf{a}-\mathbf{i}}\bar y^{\mathbf{b}-\mathbf{j}}\cdot\rho_\chi([\bar w\bar x^{\mathbf{i}}\bar y^{\mathbf{j}}])\cdot \bar u^\mathbf{c}\bar v^\mathbf{d}\otimes\bar 1_\chi,
\end{equation}
where the summation on the right side of  \eqref{rhow} runs over all $(\mathbf{i},\mathbf{j})\in\Lambda_m\times\Lambda'_n$ such that $[\bar w\bar x^{\mathbf{i}}\bar y^{\mathbf{j}}]$ is nonzero and $\text{wt}([\bar w\bar x^{\mathbf{i}}\bar y^{\mathbf{j}}])\geqslant -2$.

We continue the arguments case by case, according to the values of $\text{wt}([\bar w\bar x^{\mathbf{i}}y^{\mathbf{j}}])$ with respect to  summation parameters  $(\mathbf{i},\mathbf{j})\in \Lambda_m\times\Lambda'_n$.

 (Case 1)  wt$([\bar w\bar x^{\mathbf{i}}\bar y^{\mathbf{j}}])\geqslant 0$. 
   We have $|\mathbf{i}|+|\mathbf{j}|\geqslant 1$. Recall that the decomposition  ${\ggg}_{\bbk}=\bigoplus\limits_{i\in\mathbb{Z}}{\ggg}_{\bbk}(i)$ is the Dynkin grading of ${\ggg}_{\bbk}$, and the action of ad\,$h$ keeps $\mathfrak{p}_{\bbk}$ invariant. This implies that $\bar x^{\mathbf{a}-\mathbf{i}}\bar y^{\mathbf{b}-\mathbf{j}}\cdot\rho_\chi([\bar w\bar x^{\mathbf{i}}\bar y^{\mathbf{j}}]) \bar u^\mathbf{c}\bar v^\mathbf{d}\otimes\bar 1_\chi$ is a linear combination of $\bar x^{\mathbf{i'}}\bar y^{\mathbf{j'}}\bar u^\mathbf{c}\bar v^\mathbf{d}\otimes\bar 1_\chi$ with
$$\text{wt}(\bar x^{\mathbf{i'}}\bar y^{\mathbf{j'}}\bar u^\mathbf{c}\bar v^\mathbf{d})=\nu+\text{wt}(\bar x^{\mathbf{a}}\bar y^{\mathbf{b}}\bar u^\mathbf{c}\bar v^\mathbf{d}),$$
and\[\begin{array}{ccl}
\text{deg}_e(\bar x^{\mathbf{i'}}\bar y^{\mathbf{j'}} \bar u^\mathbf{c}\bar v^\mathbf{d})&=&\text{wt}(\bar x^{\mathbf{i'}}\bar y^{\mathbf{j'}} \bar u^\mathbf{c}\bar v^\mathbf{d})+2\text{deg}(\bar x^{\mathbf{i'}}\bar y^{\mathbf{j'}} \bar u^\mathbf{c}\bar v^\mathbf{d})\\
&=&\nu+\text{wt}(\bar x^{\mathbf{a}}\bar y^{\mathbf{b}}\bar u^\mathbf{c}\bar v^\mathbf{d})
+2(|\mathbf{i'}|+|\mathbf{j'}|+|\mathbf{c}|+|\mathbf{d}|).
\end{array}\]

As
$$|\mathbf{i'}|+|\mathbf{j'}|\leqslant |\mathbf{a}|+|\mathbf{b}|-|\mathbf{i}|-|\mathbf{j}|+1,$$
then\[\begin{array}{cl}
&\text{deg}_e(\bar x^{\mathbf{i'}}\bar y^{\mathbf{j'}}\bar u^\mathbf{c}\bar v^\mathbf{d})=\nu+\text{wt}(\bar x^{\mathbf{a}}\bar y^{\mathbf{b}}\bar u^\mathbf{c}\bar v^\mathbf{d})
+2(|\mathbf{i}'|+|\mathbf{j}'|+|\mathbf{c}|+|\mathbf{d}|)\\
\leqslant&\text{wt}(\bar x^{\mathbf{a}}\bar y^{\mathbf{b}}\bar u^\mathbf{c}\bar v^\mathbf{d})+2(|\mathbf{a}|+|\mathbf{b}|+|\mathbf{c}|+|\mathbf{d}|)+
\nu-2(|\mathbf{i}|+|\mathbf{j}|)+2\\
=&2+\nu+d'-2(|\mathbf{i}|+|\mathbf{j}|).
\end{array}\]

(Case 2) wt$([\bar w\bar x^{\mathbf{i}}\bar y^{\mathbf{j}}])=-1$. 
For $k,g\in\mathbb{Z}_+$, set $\bar x_k\in{\ggg}_{\bbk}(n_k)_{\bar{0}},~\bar y_g\in{\ggg}_{\bbk}(n'_g)_{\bar{1}}$, then $\sum\limits_{1\leqslant k\leqslant  m}i_kn_k+\sum\limits_{1\leqslant g\leqslant  n}j_gn'_g=-\nu-1$. Since $\rho_\chi({\mmm}_{\bbk}\cap{\ggg}_{\bbk}(-1))$ annihilates $\bar 1_\chi$, the vector $\bar x^{\mathbf{a}-\mathbf{i}}\bar y^{\mathbf{b}-\mathbf{j}}\cdot\rho_\chi([\bar w\bar x^{\mathbf{i}}\bar y^{\mathbf{j}}])\cdot \bar u^\mathbf{c}\bar v^\mathbf{d}\otimes\bar 1_\chi$ is a linear combination of $\bar x^{\mathbf{a}-\mathbf{i}}\bar y^{\mathbf{b}-\mathbf{j}}\bar u^\mathbf{i'}\bar v^\mathbf{j'}\otimes\bar 1_\chi$ with $|\mathbf{i'}|=|\mathbf{c}|\pm1, \mathbf{j'}=\mathbf{d}$, or $\mathbf{i'}=\mathbf{c}, |\mathbf{j'}|=|\mathbf{d}|\pm1$.

(a) If $|\mathbf{i'}|=|\mathbf{c}|+1, \mathbf{j'}=\mathbf{d}$, or $\mathbf{i'}=\mathbf{c}, |\mathbf{j'}|=|\mathbf{d}|+1$, then $|\mathbf{i}|+|\mathbf{j}|\geqslant 1$,
$$\text{wt}(\bar x^{\mathbf{a}-\mathbf{i}}\bar y^{\mathbf{b}-\mathbf{j}}\bar u^\mathbf{i'}\bar v^\mathbf{j'})= \text{wt}(\bar x^{\mathbf{a}-\mathbf{i}}\bar y^{\mathbf{b}-\mathbf{j}}\bar u^\mathbf{c}\bar v^\mathbf{d})-1
=\nu+\text{wt}(\bar x^{\mathbf{a}}\bar y^{\mathbf{b}}\bar u^\mathbf{c}\bar v^\mathbf{d}),$$
and\[\begin{array}{ccl}
\text{deg}_e(\bar x^{\mathbf{a}-\mathbf{i}}\bar y^{\mathbf{b}-\mathbf{j}}\bar u^\mathbf{i'}\bar v^\mathbf{j'})&=&\text{wt}(\bar x^{\mathbf{a}-\mathbf{i}}\bar y^{\mathbf{b}-\mathbf{j}}\bar u^\mathbf{i'}\bar v^\mathbf{j'})+
2(|\mathbf{a}|-|\mathbf{i}|+|\mathbf{b}|-|\mathbf{j}|+|\mathbf{i'}|+|\mathbf{j'}|)\\
&=&  d'+\nu+2(-|\mathbf{i}|-|\mathbf{j}|+|\mathbf{i'}|+|\mathbf{j'}|-|\mathbf{c}|-|\mathbf{d}|)\\
&=&2+d'+\nu-2(|\mathbf{i}|+|\mathbf{j}|).
\end{array}\]

(b) If $|\mathbf{i'}|=|\mathbf{c}|-1, \mathbf{j'}=\mathbf{d}$, or $\mathbf{i'}=\mathbf{c}, |\mathbf{j'}|=|\mathbf{d}|-1$, then
$$\text{wt}(\bar x^{\mathbf{a}-\mathbf{i}}\bar y^{\mathbf{b}-\mathbf{j}}\bar u^\mathbf{i'}\bar v^\mathbf{j'})
=2+\nu+\text{wt}(\bar x^{\mathbf{a}}\bar y^{\mathbf{b}}\bar u^\mathbf{c}\bar v^\mathbf{d}),$$
and\[\begin{array}{ccl}\text{deg}_e(\bar x^{\mathbf{a}-\mathbf{i}}\bar y^{\mathbf{b}-\mathbf{j}}\bar u^\mathbf{i'}\bar v^\mathbf{j'})&=&\text{wt}(\bar x^{\mathbf{a}-\mathbf{i}}\bar y^{\mathbf{b}-\mathbf{j}}\bar u^\mathbf{i'}\bar v^\mathbf{j'})+
2(|\mathbf{a}|-|\mathbf{i}|+|\mathbf{b}|-|\mathbf{j}|+|\mathbf{i'}|+|\mathbf{j'}|)\\
&=&\nu+\text{wt}(\bar x^{\mathbf{a}}\bar y^{\mathbf{b}}\bar u^\mathbf{c}\bar v^\mathbf{d})+2+
2(|\mathbf{a}|+|\mathbf{b}|+|\mathbf{c}|+|\mathbf{d}|)\\
&&+2(-|\mathbf{c}|-|\mathbf{d}|-|\mathbf{i}|-|\mathbf{j}|+|\mathbf{i'}|+|\mathbf{j'}|)\\
&=&\nu+d'-2(|\mathbf{i}|+|\mathbf{j}|).
\end{array}\]

(Case 3) wt$([\bar w\bar x^{\mathbf{i}}\bar y^{\mathbf{j}}])=-2$. Then $$\bar x^{\mathbf{a}-\mathbf{i}}\bar y^{\mathbf{b}-\mathbf{j}}\cdot\rho_\chi([\bar w\bar x^{\mathbf{i}}\bar y^{\mathbf{j}}])\cdot \bar u^\mathbf{c}\bar v^\mathbf{d}\otimes\bar 1_\chi=\chi([\bar w\bar x^{\mathbf{i}}\bar y^{\mathbf{j}}])\bar x^{\mathbf{a}-\mathbf{i}}\bar y^{\mathbf{b}-\mathbf{j}}\bar u^\mathbf{c}\bar v^\mathbf{d}\otimes\bar 1_\chi.$$
For $k,g\in\mathbb{Z}_+$, set $\bar x_k\in{\ggg}_{\bbk}(n_k)_{\bar{0}},~\bar y_g\in{\ggg}_{\bbk}(n'_g)_{\bar{1}}$. As $\sum\limits_{1\leqslant k\leqslant  m}i_kn_k+\sum\limits_{1\leqslant g\leqslant  n}j_gn'_g=-\nu-2$, one has
$$\text{wt}(\bar x^{\mathbf{a}-\mathbf{i}}\bar y^{\mathbf{b}-\mathbf{j}}\bar u^\mathbf{c}\bar v^\mathbf{d})=2+\nu+\text{wt}(\bar x^{\mathbf{a}}\bar y^{\mathbf{b}}\bar u^\mathbf{c}\bar v^\mathbf{d}),\,\text{deg}_e(\bar x^{\mathbf{a}-\mathbf{i}}\bar y^{\mathbf{b}-\mathbf{j}}\bar u^\mathbf{c}\bar v^\mathbf{d})=2+\nu+d'-2(|\mathbf{i}|+|\mathbf{j}|).$$
\vskip0.1cm
 Step 3:  For concluding  our arguments, we adopt an auxiliary endomorphism.
For $i,j\in\mathbb{Z}$, take $\pi_{ij}$ to be an endomorphism of $Q_\chi^\chi$ defined  via
\begin{align}\label{pi}
\pi_{ij}(\bar x^{\mathbf{a}}\bar y^{\mathbf{b}}\bar u^\mathbf{c}\bar v^\mathbf{d}\otimes\bar1_\chi)=\begin{cases}\bar x^{\mathbf{a}}\bar y^{\mathbf{b}}\bar u^\mathbf{c}\bar v^\mathbf{d}\otimes\bar 1_\chi &\mbox{ if }\text{deg}_e(\bar x^{\mathbf{a}}\bar y^{\mathbf{b}}\bar u^\mathbf{c}\bar v^\mathbf{d})=i\\ &\mbox{ and } \text{wt}(\bar x^{\mathbf{a}}\bar y^{\mathbf{b}}\bar u^\mathbf{c}\bar v^\mathbf{d})=j;\\0 &\mbox{ otherwise}.
\end{cases}
\end{align}

Now we proceed to complete the arguments by reducing contradictions under the beginning assumption. Firstly recall that $\bar w\in{\mmm}_{\bbk}$ by Step 1.

(i) If (S1) or (S2) happens, then $\nu\leqslant -2$. Set $\bar h=\bar h_{\bar{0}}+\bar h_{\bar{1}}\in U_\chi({{\ggg}_{\bbk},e})$. As $\bar w$ is $\mathbb{Z}_2$-homogeneous and $\chi(({\ggg}_{\bbk})_{\bar1})=0$, then $\chi(\bar w)=0$ if $\bar w\in({\ggg}_{\bbk})_{\bar1}$. It follows from the definition of $U_\chi({{\ggg}_{\bbk},e})$ that
\begin{equation}\label{rhochi}
\begin{split}
(\rho_\chi(\bar w)-\chi(\bar w)\text{id}).\bar h(1_\chi)=&(\rho_\chi(\bar w)-\chi(\bar w)\text{id}).
(\bar h_{\bar{0}}+\bar h_{\bar{1}})(\bar 1_\chi)\\
=&\bar h_{\bar{0}}((\bar w-\chi(\bar w)).\bar 1_\chi)
+(-1)^{|\bar w|}\bar h_{\bar{1}}(\bar w.\bar 1_\chi)-\chi(\bar w)\bar h_{\bar{1}}(\bar 1_\chi)\\
=&\bar h_{\bar{0}}((\bar w-\chi(\bar w)).\bar 1_\chi)
+(-1)^{|\bar w|}\cdot \bar h_{\bar{1}}((\bar w-\chi(\bar w)).\bar 1_\chi)\\
=&0.
\end{split}
\end{equation}
For $a\in\mathbb{Z}$ we let $\bar{a}$ denote the residue of $a$ in $\mathbb{F}_p\subseteq\overline{\mathbb{F}}_p={\bbk}$. By the arguments in Step 2 one knows that the terms with $e$-degree $n(\bar h)+\nu$ and weight $N(\bar h)+\nu+2$ only occur as in (Case 3) with wt$([\bar w\bar x^{\mathbf{i}}\bar y^{\mathbf{j}}])=-2$ when $(\rho_\chi(\bar w)-\chi(\bar w)\text{id}).\bar h(\bar 1_\chi)$ is written as a linear combination of the canonical basis of $Q_\chi^\chi$. It follows from \eqref{rhochi} that with an appointment $b_{n+1}=0$,
\begin{equation}\label{0pi}
\begin{split}
0=&\pi_{n(\bar h)+\nu,N(\bar h)+\nu+2}((\rho_\chi(\bar w)-\chi(\bar w)\text{id}). \bar h(\bar 1_\chi))\\
=&(\sum\limits_{(\mathbf{a},\mathbf{b},\mathbf{c},\mathbf{d})\in\Lambda_{\bar h}^\text{max}}\lambda_{\mathbf{a},\mathbf{b},\mathbf{c},\mathbf{d}}
\sum\limits_{i=1}^m(-1)^{|\bar w|(\sum\limits_{j=1}^nb_j)}\bar{a}_i \bar x^{\mathbf{a}-\mathbf{e}_i}\bar y^{\mathbf{b}}\cdot\chi([\bar w,\bar x_i]) \bar u^\mathbf{c}\bar v^\mathbf{d})\otimes\bar 1_\chi+\\
&(\sum\limits_{(\mathbf{a},\mathbf{b},\mathbf{c},\mathbf{d})\in\Lambda_{\bar h}^\text{max}}\lambda_{\mathbf{a},\mathbf{b},\mathbf{c},\mathbf{d}}
\sum\limits_{i=1}^n(-1)^{|\bar w|(1+\sum\limits_{j=1}^nb_j)+\sum\limits_{j=i+1}^nb_{j}} \bar x^{\mathbf{a}}\bar y^{\mathbf{b}-\mathbf{e}_i}\cdot\chi([\bar w,\bar y_i])\\
&\bar u^\mathbf{c}\bar v^\mathbf{d})\otimes\bar 1_\chi.
\end{split}
\end{equation}

On the other hand, we have

(a) if $\bar w=\bar w_k$, then the most-right side of \eqref{0pi} becomes $$\sum\limits_{(\mathbf{a},\mathbf{b},\mathbf{c},\mathbf{d})\in\Lambda_{\bar h}^\text{max}}\lambda_{\mathbf{a},\mathbf{b},\mathbf{c},\mathbf{d}}
\bar{a}_k\bar x^{\mathbf{a}-\mathbf{e}_k}\bar y^{\mathbf{b}}\bar u^\mathbf{c}\bar v^\mathbf{d}\otimes\bar 1_\chi\neq0;$$

(b) if $\bar w=\bar w'_k$, then the most-right side of \eqref{0pi} becomes $$\sum\limits_{(\mathbf{a},\mathbf{b},\mathbf{c},\mathbf{d})\in\Lambda_{\bar h}^\text{max}}\lambda_{\mathbf{a},\mathbf{b},\mathbf{c},\mathbf{d}}
(-1)^{1+\sum\limits_{j=1}^nb_j+\sum\limits_{j=k+1}^nb_{j}}\bar x^{\mathbf{a}}\bar y^{\mathbf{b}-\mathbf{e}_k}\cdot \bar u^\mathbf{c}\bar v^\mathbf{d}\otimes\bar 1_\chi\neq0,$$
a contradiction. This is to say,  both (S1) and (S2) impossibly  happen.

(ii) If (S3) or (S4) happens, then $\nu=-1$ and $\chi(\bar w)=0$. It follows from the definition of $U_\chi({{\ggg}_{\bbk},e})$ that
\begin{equation}\label{rhowchi}
\rho_\chi(\bar w).\bar h(\bar 1_\chi)=\rho_\chi(\bar w).(\bar h_{\bar{0}}+\bar h_{\bar{1}})(\bar 1_\chi)=\bar h_{\bar{0}}(\bar w.\bar 1_\chi)
+(-1)^{|\bar w|}\bar h_{\bar{1}}(\bar w.\bar 1_\chi)=0.
\end{equation}
With the appointment that $b_{n+1}=0, d_{0}=0$, we have
\begin{equation}\label{exten}
\begin{split}
\rho_\chi(\bar w).\bar h(\bar 1_\chi)=&(\sum\limits_{(\mathbf{a},\mathbf{b},\mathbf{c},\mathbf{d})\in\Lambda_{\bar h}^{n(\bar h)}}\lambda_{\mathbf{a},\mathbf{b},\mathbf{c},\mathbf{d}}(
\sum\limits_{i=1}^m(-1)^{|\bar w|(\sum\limits_{j=1}^nb_j)}\bar{a}_i\bar x^{\mathbf{a}-\mathbf{e}_i}\bar y^{\mathbf{b}}\cdot\rho_\chi([\bar w,\bar x_i])\\
&\cdot \bar u^\mathbf{c}\bar v^\mathbf{d}+\sum\limits_{i=1}^n(-1)^{|\bar w|(1+\sum\limits_{j=1}^nb_j)+\sum\limits_{j=i+1}^nb_{j}}\bar x^{\mathbf{a}}\bar y^{\mathbf{b}-\mathbf{e}_i}\cdot\rho_\chi([\bar w,\bar y_i])\cdot \bar u^\mathbf{c}\bar v^\mathbf{d}+\\
&\sum\limits_{i=1}^s(-1)^{|\bar w|(\sum\limits_{j=1}^nb_j)}\bar{c}_i\bar x^{\mathbf{a}}\bar y^{\mathbf{b}}\cdot\chi([\bar w,\bar u_i])\bar u^{\mathbf{c}-\mathbf{e}_i}\bar v^{\mathbf{d}}+\\&\sum\limits_{i=1}^{\lfloor \frac{\sfr}{2}\rfloor}(-1)^{|\bar w|(\sum\limits_{j=1}^nb_j)+\sum\limits_{j=1}^{i-1}d_{j}} \bar x^{\mathbf{a}}\bar y^{\mathbf{b}}\cdot\chi([\bar w,\bar v_i])\bar u^{\mathbf{c}}\bar v^{\mathbf{d}-\mathbf{e}_i})+\\
&\sum\limits_{|(\mathbf{i},\mathbf{j},\mathbf{k},\mathbf{l})|_e\leqslant  n(\bar h)-2}\beta_{\mathbf{i},\mathbf{j},\mathbf{k},\mathbf{l}}\cdot \bar x^{\mathbf{i}}\bar y^{\mathbf{j}}\bar u^\mathbf{k}\bar v^\mathbf{l})\otimes\bar1_\chi.
\end{split}
\end{equation}
When $(\rho_\chi(\bar w)-\chi(\bar w)\text{id})\cdot\bar h(\bar 1_\chi)$ is written as a linear combination of the canonical basis of $Q_\chi^\chi$, it is immediate from the arguments in Step 2 that the terms with $e$-degree $n(\bar h)-1$ and weight $N(\bar h)+1$ only occur as in (Case 2)(b) with wt$([\bar w\bar x^{\mathbf{i}}\bar y^{\mathbf{j}}])=-1$. By \eqref{exten} we have

(a) if $\bar w=\bar z'_k$, then
$$\pi_{n(\bar h)-1,N(\bar h)+1}(\rho_\chi(\bar w).\bar h(\bar 1_\chi))=\sum
\limits_{(\mathbf{a},\mathbf{b},\mathbf{c},\mathbf{d})\in\Lambda_{\bar h}^{n(\bar h)}}\lambda_{\mathbf{a},\mathbf{b},\mathbf{c},\mathbf{d}}
\bar{c}_k\bar x^{\mathbf{a}}\bar y^{\mathbf{b}}\bar u^{\mathbf{c}-\mathbf{e}_k}\bar v^{\mathbf{d}}\otimes\bar 1_\chi\neq0;$$

(b) if $\bar w=\bar z^{''}_k$, then
\[\begin{array}{ll}
&\pi_{n(\bar h)-1,N(\bar h)+1}(\rho_\chi(\bar w).\bar h(\bar 1_\chi))\\
=&\sum
\limits_{(\mathbf{a},\mathbf{b},\mathbf{c},\mathbf{d})\in\Lambda_{\bar h}^{n(\bar h)}}\lambda_{\mathbf{a},\mathbf{b},\mathbf{c},\mathbf{d}}
(-1)^{\sum\limits_{j=1}^nb_j+\sum\limits_{j=1}^{k-1}d_{j}}\bar x^{\mathbf{a}}\bar y^{\mathbf{b}}\bar u^{\mathbf{c}}\bar v^{\mathbf{d}-\mathbf{e}_k}\otimes\bar1_\chi
\neq0,
\end{array}\]
which contradicts to \eqref{rhowchi}. This is to say, both (S3) and (S4) impossibly happen.

Summing up, we find the beginning assumption is absurd. The proof is completed.
\end{proof}

\begin{rem} From Lemma \ref{hw} on, we can see clearly that the parity of $\sfr=\text{dim}\,\ggg_\bbf(-1)_\bo$ is decisive to the variation of  structure of the finite $W$-superalgebra $U(\ggg_\bbf,e)$ for $\bbf=\bbc$ and the reduced $W$-superalgebra $U_\chi(\ggg_\bbk,e)$ for $\bbf=\bbk$.
\end{rem}

\subsection{The construction theory of reduced $W$-superalgebras in positive characteristic}

In this part we will finally  obtain the PBW structure theorem for the reduced $W$-superalgebra $U_\chi({\ggg}_{\bbk},e)$ which presents both the generators and relations arising from the Lie superalgebra $\ggg_\bbk^e$, and a PBW basis of $U_\chi(\ggg_\bbk,e)$.

Let us recall some information already known on the structure of $U_\chi(\ggg_\bbk,e)$. For $k\in\mathbb{Z}_+$ let $H^k$ denote the ${\bbk}$-linear span of all $0\neq \bar h\in U_\chi({\ggg}_{\bbk},e)$ with $n(\bar h)\leqslant  k$ in $U_\chi({\ggg}_{\bbk},e)$. It follows readily from Lemma~\ref{commutative relations k2} that $H^i\cdot H^j\subseteq H^{i+j}$ for all $i,j\in\mathbb{Z}_+$. In other words, $\{H^i|i\in\mathbb{Z}_+\}$ is a filtration of the algebra $U_\chi({\ggg}_{\bbk},e)$ and obviously $U_\chi({\ggg}_{\bbk},e)=H^k$ for all $k\gg0$. We set $H^{-1}=0$ and let $\text{gr}\,(U_\chi({\ggg}_{\bbk},e))=\sum\limits_{i\geqslant 0}H^i/H^{i-1}$ denote the corresponding graded algebra. Lemma~\ref{commutative relations k2} implies that the ${\bbk}$-algebra $\text{gr}\,(U_\chi({\ggg}_{\bbk},e))$ is supercommutative.

\begin{prop}\label{reduced basis} Keep the notations as above. The following statements hold.

(1) If $\sfr$ is even, then for any $(\mathbf{a},\mathbf{b})\in\Lambda_l\times\Lambda'_q$ there is $\bar h_{\mathbf{a},\mathbf{b}}\in U_\chi(\mathfrak{g}_\mathds{k},e)$ such that $\Lambda_{\bar h_{\mathbf{a},\mathbf{b}}}^{\text{max}}=\{(\mathbf{a},\mathbf{b})\}$. The vectors $\{\bar h_{\mathbf{a},\mathbf{b}}\mid(\mathbf{a},\mathbf{b})\in\Lambda_l\times\Lambda'_q\}$ form a basis of $U_\chi(\mathfrak{g}_\mathds{k},e)$ over $\mathds{k}$.

(2) If $\sfr$ is odd, then for any $(\mathbf{a},\mathbf{b},c)\in\Lambda_l\times\Lambda'_q\times\Lambda'_1$ there is $\bar h_{\mathbf{a},\mathbf{b},c}\in U_\chi(\mathfrak{g}_\mathds{k},e)$ such that $\Lambda_{\bar h_{\mathbf{a},\mathbf{b},c}}^{\text{max}}=\{(\mathbf{a},\mathbf{b},c)\}$. The vectors $\{\bar h_{\mathbf{a},\mathbf{b},c}\mid(\mathbf{a},\mathbf{b},c)\in\Lambda_l\times\Lambda'_q\times\Lambda'_1\}$ form a basis of $U_\chi(\mathfrak{g}_\mathds{k},e)$ over $\mathds{k}$.
\end{prop}

\begin{proof} We will give the proof for the second situation. The first one can be proved  similarly but more simply.

Assume $\sfr$ is odd.
Given $(a,b)\in\mathbb{Z}_+^2$, let $H^{a,b}$ denote the subspace of $U_\chi({\ggg}_{\bbk},e)$ spanned by $H^{a-1}$ and all $\bar h\in U_\chi({\ggg}_{\bbk},e)$ with $n(\bar h)=a,~N(\bar h)\leqslant  b$. Order the elements in $\mathbb{Z}_+^2$ lexicographically. By construction, $H^{a,b}\subseteq H^{c,d}$ whenever $(a,b)\prec(c,d)$. Note that $U_\chi(\ggg_\bbk,e)$ is finite-dimensional.  We know that $U_\chi({\ggg}_{\bbk},e)$ has basis $B:=\bigsqcup_{(i,j)}B_{i,j}$ such that $n(\mu)=i,~N(\mu)=j$ whenever $\mu\in B_{i,j}$.

Recall the mapping $\pi_{ij}:Q_\chi^\chi\rightarrow Q_\chi^\chi$ in \eqref{pi}.
Define the linear map $\pi_B: U_\chi({\ggg}_{\bbk},e)\longrightarrow Q_\chi^\chi$ via $\pi_B(\mu)=\pi_{i,j}(\mu(\bar1_\chi))$ for any  $\mu\in B_{i,j}$. It follows  from Lemma~\ref{hw} that $\pi_B$ maps $U_\chi({\ggg}_{\bbk},e)$ into the subspace of $(U_\chi({\ggg}_{\bbk}^e)\oplus U_\chi({\ggg}_{\bbk}^e)\bar v_{\frac{\sfr+1}{2}})\otimes\bar 1_\chi$ of $Q_\chi^\chi$. By construction, $\pi_B$ is injective. On the other hand,
\[\begin{array}{ccl}\text{\underline{dim}}\,U_\chi({\ggg}_{\bbk},e)&=&(\frac{p^{\text{dim}\,({\ggg}_{\bbk})_{\bar{0}}}}
{(p^{\frac{\text{dim}\,({\ggg}_{\bbk})_{\bar{0}}-\text{dim}\,({\ggg}_{\bbk}^e)_{\bar{0}}}{2}})^2},~
\frac{2^{\text{dim}\,({\ggg}_{\bbk})_{\bar{1}}}}{(2^{\frac{\text{dim}\,({\ggg}_{\bbk})_{\bar{1}}-\text{dim}\,({\ggg}_{\bbk}^e)_{\bar{1}}-1}{2}})^2})\\
&=&(p^{\text{dim}\,({\ggg}_{\bbk}^e)_{\bar{0}}},~2^{\text{dim}\,({\ggg}_{\bbk}^e)_{\bar{1}}+1})\\
&=&\text{\underline{dim}}\,
(U_\chi({\ggg}_{\bbk}^e)\oplus U_\chi({\ggg}_{\bbk}^e)\bar v_{\frac{\sfr+1}{2}})\otimes\bar 1_\chi,
\end{array}\]due to Remark~\ref{centralizer} and Theorem~\ref{matrix}(1).  Thus $\pi_B: U_\chi({\ggg}_{\bbk},e)\longrightarrow (U_\chi({\ggg}_{\bbk}^e)\oplus U_\chi({\ggg}_{\bbk}^e)\bar v_{\frac{\sfr+1}{2}})\otimes\bar 1_\chi$ is a linear isomorphism. For $(\mathbf{a},\mathbf{b},c)=(a_1,\cdots,a_l,b_1,\cdots,b_q,c)\in\Lambda_l\times\Lambda'_q\times\Lambda'_1$ set $$h_{\mathbf{a},\mathbf{b},c}=\pi_B^{-1}(\bar x_1^{a_1}\cdots \bar x_l^{a_l}\bar y_1^{b_1}\cdots \bar y_q^{b_q}\bar v_{\frac{\sfr+1}{2}}^c\otimes\bar1_\chi).$$ By the bijectivity of $\pi_B$ and the PBW theorem (applied to $U_\chi({\ggg}_{\bbk}^e)$), the vectors $h_{\mathbf{a},\mathbf{b},c}$ with $(\mathbf{a},\mathbf{b},c)\in\Lambda_l\times\Lambda'_q\times\Lambda'_1$ form a basis of $U_\chi({\ggg}_{\bbk},e)$, while from the definition of $\pi_B$ it follows that $\Lambda_{\bar h_{\mathbf{a},\mathbf{b},c}}^{\text{max}}=\{(\mathbf{a},\mathbf{b},c)\}$ for any $(\mathbf{a},\mathbf{b},c)\in\Lambda_l\times\Lambda'_q\times\Lambda'_1$.
\end{proof}

\begin{corollary}\label{rg}\begin{itemize}
\item[(1)]
There exist even elements $\theta_1,\cdots,\theta_l\in U_\chi({\ggg}_{\bbk},e)_{\bar0}$ and odd elements $ \theta_{l+1},\cdots,\theta_{l+q}\in U_\chi({\ggg}_{\bbk},e)_{\bar1}$ such that
\begin{itemize}
\item[(a)]\begin{equation*}
\begin{array}{ll}
&\theta_k(\bar 1_\chi)\\=&(\bar x_k+\sum\limits_{\mbox{\tiny $\begin{array}{c}|\mathbf{a},\mathbf{b},\mathbf{c},\mathbf{d}|_e=m_k+2,\\|\mathbf{a}|
+|\mathbf{b}|+|\mathbf{c}|+|\mathbf{d}|\geqslant 2\end{array}$}}\lambda^k_{\mathbf{a},\mathbf{b},\mathbf{c},\mathbf{d}}\bar x^{\mathbf{a}}
\bar y^{\mathbf{b}}\bar u^{\mathbf{c}}\bar v^{\mathbf{d}}+\sum\limits_{|\mathbf{a},\mathbf{b},\mathbf{c},\mathbf{d}|_e<m_k+2}
\lambda^k_{\mathbf{a},\mathbf{b},\mathbf{c},\mathbf{d}}\bar x^{\mathbf{a}}
\bar y^{\mathbf{b}}\bar u^{\mathbf{c}}\bar v^{\mathbf{d}})\otimes\bar 1_\chi,
\end{array}
\end{equation*}
where $\bar x_k\in{\ggg}^e_{\bbk}(m_k)_{\bar0}$ for $1\leqslant  k\leqslant  l$.
\item[(b)]\begin{equation*}
\begin{array}{ll}
&\theta_{l+k}(\bar 1_\chi)\\=&(\bar y_k+\sum\limits_{\mbox{\tiny $\begin{array}{c}|\mathbf{a},\mathbf{b},\mathbf{c},\mathbf{d}|_e=n_k+2,\\|\mathbf{a}|
+|\mathbf{b}|+|\mathbf{c}|+|\mathbf{d}|\geqslant 2\end{array}$}}\lambda^k_{\mathbf{a},\mathbf{b},\mathbf{c},\mathbf{d}}\bar x^{\mathbf{a}}
\bar y^{\mathbf{b}}\bar u^{\mathbf{c}}\bar v^{\mathbf{d}}+\sum\limits_{|\mathbf{a},\mathbf{b},\mathbf{c},\mathbf{d}|_e<n_k+2}
\lambda^k_{\mathbf{a},\mathbf{b},\mathbf{c},\mathbf{d}}\bar x^{\mathbf{a}}
\bar y^{\mathbf{b}}\bar u^{\mathbf{c}}\bar v^{\mathbf{d}})\otimes\bar 1_\chi,
\end{array}
\end{equation*}
where $\bar y_k\in{\ggg}^e_{\bbk}(n_k)_{\bar1}$ for $1\leqslant  k\leqslant  q$.
\end{itemize}
\item[(2)] If additionally $\sfr=\dim{\ggg}_{\bbk}(-1)_{\bar{1}}$ is odd, then there is an odd element $\theta_{l+q+1}\in U_\chi({\ggg}_{\bbk},e)_{\bar1}$ such that$$\theta_{l+q+1}(\bar 1_\chi)=\bar v_{\frac{\sfr+1}{2}}\otimes\bar 1_\chi.$$

\end{itemize}
All the coefficients $\lambda^k_{\mathbf{a},\mathbf{b},\mathbf{c},\mathbf{d}}$ above are in ${\bbk}$. Moreover, $\lambda^k_{\mathbf{a},\mathbf{b},\mathbf{c},\mathbf{d}}=0$ if $(\mathbf{a},\mathbf{b},\mathbf{c},\mathbf{d})$ is such that $a_{l+1}=\cdots=a_m=b_{q+1}=\cdots=b_n=c_1=\cdots=c_s=d_1=\cdots=d_{\lceil\frac{\sfr}{2}\rceil}=0$.
\end{corollary}

\begin{proof}
The existence of all the elements in the corollary is an immediate consequence of Proposition~\ref{reduced basis}. The remaining thing is to prove the $\mathbb{Z}_2$-homogeneity for the elements in Statements (a) and (b) of (1). In fact, it can be obtained by applying Proposition~\ref{invariant} (with the case $\eta=\chi$) directly. Firstly note that each element $\theta\in U_\chi({\ggg}_{\bbk},e)$ can be written as $\theta=\theta_{\bar0}+\theta_{\bar1}$ where $\theta_{\bar0}(\bar1_\chi)\in (Q_\chi^\chi)_{\bar0}$ and $\theta_{\bar1}(\bar1_\chi)\in (Q_\chi^\chi)_{\bar1}$. Since both the ${\bbk}$-algebras $U_\chi({\ggg}_{\bbk},e)$ and $(Q_\chi^\chi)^{\text{ad}{\mmm}_{\bbk}}$ are $\mathbb{Z}_2$-graded and the mapping $$\varphi:~U_\chi({\ggg}_{\bbk},e)\overset{\sim}{\longrightarrow}(Q_\chi^\chi)^{\text{ad}{\mmm}_{\bbk}}$$ in Proposition~\ref{invariant} is even, it follows from $$\varphi(\theta)=\varphi(\theta_{\bar0})+\varphi(\theta_{\bar1})= \theta_{\bar0}(\bar 1_\chi)+\theta_{\bar1}(\bar 1_\chi)\in(Q_\chi^\chi)^{\text{ad}{\mmm}_{\bbk}}$$
that $\varphi(\theta_{\bar0})\in(Q_\chi^\chi)^{\text{ad}{\mmm}_{\bbk}}_{\bar0}$ and $\varphi( \theta_{\bar1})\in(Q_\chi^\chi)^{\text{ad}{\mmm}_{\bbk}}_{\bar1}$, i.e. $\theta_{\bar0}\in U_\chi({\ggg}_{\bbk},e)_{\bar0}$ and $\theta_{\bar1}\in U_\chi({\ggg}_{\bbk},e)_{\bar1}$.

Thus, for any $\theta\in U_\chi({\ggg}_{\bbk},e)$, if the leading term of $\theta(\bar 1_\chi)$ is $\bar x_i$ with $1\leqslant  i\leqslant  l$ (thereby in $(\ggg_\bbk)_\bz$), one can choose $\theta_{\bar0}$ as the desired element in $U_\chi({\ggg}_{\bbk},e)_{\bar0}$. Similarly, if the leading term of $\theta(\bar 1_\chi)$ is $\bar y_i$ with $1\leqslant  i\leqslant  q$ (thereby in $({\ggg}_{\bbk})_{\bar1}$), one can choose $\theta_{ \bar1}$ as the desired element in $U_\chi({\ggg}_{\bbk},e)_{\bar1}$.
\end{proof}

Recall that $\{\bar x_1,\cdots,\bar x_l\}$ and $\{\bar y_1,\cdots,\bar y_q\}$ are ${\bbk}$-basis of $({\ggg}^e_{\bbk})_{\bar{0}}$ and $({\ggg}^e_{\bbk})_{\bar{1}}$, respectively; and there is another element $\bar v_{\frac{\sfr+1}{2}}\in({\ggg}_{\bbk}(-1)_{\bar{1}}^{\prime})^{\bot }\subseteq {\ggg}_{\bbk}(-1)_{\bar{1}}$ when $\sfr$ is odd. Set
\[\bar Y_i:=\left\{
\begin{array}{ll}
\bar x_i&\text{if}~1\leqslant  i\leqslant  l;\\
\bar y_{i-l}&\text{if}~l+1\leqslant  i\leqslant  l+q;\\
\bar v_{\frac{\sfr+1}{2}}&\text{if}~i=l+q+1 \mbox{ whenever } \sfr\mbox{ is odd}.
\end{array}
\right.
\]
   With the new notations, the subalgebra $\ggg_\bbk^e$ admits a basis
$$\bar Y_1,\cdots,\bar Y_{l+q}.$$
Moreover, the element $\bar Y_{l+q+1}\notin\ggg_\bbk^e$ occurs only in the case when $\sfr$ is odd. We can assume that the homogenous element $\bar Y_i$ is in $\ggg_{\bbk}(m_i)$ for  $1\leqslant  i\leqslant  l+q+1$.   Let $\bar{\theta}_i$ denote the image of $\theta_i\in U_\chi({\ggg}_{\bbk},e)$ in $\text{gr}\,(U_\chi({\ggg}_{\bbk},e))$.

Now we are in the position to introduce the main result of this section, which provides the generators and relations for the reduced $W$-superalgebra $U_\chi({\ggg}_{\bbk},e)$,  along with the PBW basis. We will exploit the arguments of \cite[Theorem 3.4]{P2} in our super case. We will further see variation arising from the change of the parity of $\sfr=\dim\ggg_\bbk(-1)_\bo$. In the remaining part of the paper,  we always set $q'=q$ if $\sfr$ is even, and $q'=q+1$ if $\sfr$ is odd.

\begin{theorem}\label{reduced Wg} All the elements given in Corollary~\ref{rg} constitute a set of generators of $U_\chi({\ggg}_{\bbk},e)$. Their relations are presented as below.
\begin{itemize}
\item[(1)] For $1\leqslant  i,j\leqslant  l+q'$ , we have
$$[\theta_i,\theta_j]=\theta_i\cdot\theta_j-(-1)^{|\theta_i||\theta_j|}\theta_j\cdot\theta_i=(-1)^{|\theta_i|\cdot|\theta_j|}\theta_j\circ\theta_i-\theta_i\circ\theta_j
\in H^{m_i+m_j+2}.$$
Here and thereinafter  $\theta_i\cdot\theta_j$ is just the algebra multiplication of two elements in $U_\chi(\ggg_\bbk,e)$ while $\theta_i\circ\theta_j$ means the composition of two transformations on $Q_\chi^\chi$.

\item[(2)] The Lie bracket relations in the Lie superalgebra $\ggg_\bbk^e$ with
$$[\bar Y_i,\bar Y_j]=\sum\limits_{k=1}^{l+q}\alpha_{ij}^k\bar Y_k\quad \text{for}~1\leqslant i,j\leqslant l+q$$ imply
\begin{align*}\label{thetacom}
[\theta_i,\theta_j]&\equiv\sum\limits_{k=1}^{l+q}\alpha_{ij}^k\theta_k+q_{ij}
(\theta_1,\cdots,\theta_{l+q'})~(\mbox{mod }H^{m_i+m_j+1}),
\end{align*}
  where $q_{ij}$ with $1\leqslant i,j\leqslant l+q$ is a truncated super-polynomial
  in $l+q'$ variables whose constant term and linear part are both zero. When $\sfr$ is odd, we have additionally
  \begin{equation}\label{explainforeq0}[\theta_{l+q+1},\theta_{l+q+1}]=\text{id}.
\end{equation}
\item[(3)] The monomials $\bar{\theta}_1^{a_1}\cdots\bar{\theta}_l^{a_l}\bar{\theta}_{l+1}^{b_1}\cdots\bar{\theta}_{l+q'}^{b_{q'}}$ and $\theta_1^{a_1}\cdots\theta_l^{a_l}\theta_{l+1}^{b_1}\cdots\theta_{l+q'}^{b_{q'}}$ form bases of $\text{gr}\,(U_\chi({\ggg}_{\bbk},e))$ and $U_\chi({\ggg}_{\bbk},e)$ respectively, where $0\leqslant  a_i\leqslant  p-1$ for $1\leqslant i\leqslant l$, and $b_i\in\{0,1\}$ for $1\leqslant i\leqslant q'$.
    \end{itemize}
\end{theorem}

\begin{proof} The arguments for both cases of $\sfr$ being even and odd are almost the same, with the latter case being more complicated. We will only consider the latter case, i.e.  the case when $\sfr$ is odd. Firstly, it can be easily verified that
\begin{equation}\label{relaitionsss}
[\theta_i,\theta_j]=\theta_i\cdot\theta_j-(-1)^{|\theta_i||\theta_j|}\theta_j\cdot\theta_i=(-1)^{|\theta_i|\cdot|\theta_j|}\theta_j\circ\theta_i-\theta_i\circ\theta_j
\end{equation}by the definition of the $\bbk$-algebra $U_\chi(\ggg_\bbk,e)$.

 Now let us next prove (3).  Firstly recall the elements $\theta_1,\cdots,\theta_l\in U_\chi({\ggg}_{\bbk},e)_{\bar0}$ and $\theta_{l+1},\cdots,\theta_{l+q+1}\in U_\chi({\ggg}_{\bbk},e)_{\bar1}$ in Corollary \ref{rg}.
Combing with \eqref{relaitionsss}, the arguments by induction on $|\mathbf{a}|+|\mathbf{b}|+c$ show that
\[\begin{array}{cl}
&\theta_1^{a_1}\cdots\theta_l^{a_l}\theta_{l+1}^{b_1}\cdots\theta_{l+q}^{b_q}\theta_{l+q+1}^{c}(\bar 1_\chi)\\
=&(\bar Y_1^{a_1}\cdots \bar Y_l^{a_l}\bar Y_{l+1}^{b_1}\cdots \bar Y_{l+q}^{b_q}\bar Y_{l+q+1}^c+\sum\limits_{\mbox{\tiny $\begin{array}{c}|(\mathbf{i},\mathbf{j},\mathbf{f},\mathbf{g})|_e=|(\mathbf{a},\mathbf{b},\mathbf{0},
c\mathbf{e}_{\frac{\sfr+1}{2}})|_e,\\|\mathbf{i}|+|\mathbf{j}|+|\mathbf{f}|+|\mathbf{g}|>|\mathbf{a}|+|\mathbf{b}|+c\end{array}$}}
\lambda^{\mathbf{a},\mathbf{b},\mathbf{0},c\mathbf{e}_{\frac{\sfr+1}{2}}}_{\mathbf{i},\mathbf{j},\mathbf{f},\mathbf{g}}\bar x^{\mathbf{i}}
\bar y^{\mathbf{j}}\bar u^{\mathbf{f}}\bar v^{\mathbf{g}}\\
&+\text{terms of lower}~e\text{-degree})\otimes\bar 1_\chi
\end{array}\]for any $(\mathbf{a},\mathbf{b},c)=(a_1,\cdots,a_l,b_1,\cdots,b_q,c)\in\Lambda_l\times\Lambda'_q\times\Lambda'_1$ (the induction step is based on Corollary~\ref{rg} and Lemma~\ref{commutative relations k2}). Due to Proposition~\ref{reduced basis} we have that
\begin{equation}\label{thetamu}
\theta_1^{a_1}\cdots\theta_{l+q+1}^c=\mu_{\mathbf{a},\mathbf{b},c}\bar h_{\mathbf{a},\mathbf{b},c}+\sum
\limits_{(\mathbf{i},\mathbf{j},k)\in\Lambda_l\times\Lambda'_q\times\Lambda'_1}\mu_{\mathbf{i},\mathbf{j},k}
\bar h_{\mathbf{i},\mathbf{j},k},\qquad\mu_{\mathbf{a},\mathbf{b},c}\neq0,
\end{equation}where $\mu_{\mathbf{i},\mathbf{j},k}=0$ unless $(n(\bar h_{\mathbf{i},\mathbf{j},k}),N(\bar h_{\mathbf{i},\mathbf{j},k}))\prec (n(\bar h_{\mathbf{a},\mathbf{b},c}),N(\bar h_{\mathbf{a},\mathbf{b},c}))$.

Since this establishes for any $(\mathbf{a},\mathbf{b},c)\in\Lambda_l\times\Lambda'_q\times\Lambda'_1$, the monomials $\theta_1^{a_1}\cdots\theta_{l+q+1}^c$ with $(\mathbf{a},\mathbf{b},c)=(a_1,\cdots,a_l,b_1,\cdots,b_q,c)\in\Lambda_l\times\Lambda'_q\times\Lambda'_1$ form a basis of $U_\chi({\ggg}_{\bbk},e)$. It follows from the proof of Proposition~\ref{reduced basis} that for any positive integer $M$, the cosets $\bar h_{\mathbf{i},\mathbf{j},k}+H^{M-1}$ with $(\mathbf{i},\mathbf{j},k)\in\Lambda_l\times\Lambda'_q\times\Lambda'_1$ and $\sum\limits_{1\leqslant f\leqslant  l}i_f(m_f+2)+\sum\limits_{1\leqslant g\leqslant  q}j_g(m_{g+l}+2)+k=M$ (recall that $m_f's,~m_{g+l}'s$ are the weights of $\bar Y_f$'s ($1\leqslant f\leqslant l$) and  $\bar Y_{g+l}$'s ($1\leqslant g\leqslant q$), respectively) form a basis of $\text{gr}(U_\chi({\ggg}_{\bbk},e))$. Due to \eqref{thetamu} and Lemma~\ref{commutative relations k2} the cosets $\theta_1^{i_1}\cdots\theta_l^{i_l}\theta_{l+1}^{j_1}\cdots\theta_{l+q}^{j_q}\theta_{l+q+1}^{k}+H^{M-1}$ with $(\mathbf{i},\mathbf{j},k)\in\Lambda_l\times\Lambda'_q\times\Lambda'_1$ and $\sum\limits_{1\leqslant f\leqslant  l}i_f(m_f+2)+\sum\limits_{1\leqslant g\leqslant  q}j_g(m_{g+l}+2)+k=M$ have the same property.

Now for the completion of the proof (3), we note that
$$\bar{\theta}_1^{i_1}\cdots\bar{\theta}_l^{i_l}
\bar{\theta}_{l+1}^{j_1}\cdots\bar{\theta}_{l+q}^{j_q}\bar{\theta}_{l+q+1}^{k}
= \theta_1^{i_1}\cdots\theta_l^{i_l}\theta_{l+1}^{j_1}\cdots\theta_{l+q}^{j_q}\theta_{l+q+1}^{k}+H^{M-1}
$$
 for any $(\mathbf{i},\mathbf{j},k)\in\Lambda_l\times\Lambda'_q\times\Lambda'_1$  with $\sum\limits_{1\leqslant f\leqslant  l}i_f(m_f+2)+\sum\limits_{1\leqslant g\leqslant  q}j_g(m_{g+l}+2)+k=M$. Hence the monomials listed in the first part of statement (3) constitute a basis of $\text{gr}(U_\chi({\ggg}_{\bbk},e))$. We complete the proof for (3).

  Now we prove the statement (2).   As in the proof of Lemma~\ref{commutative relations k2}, the induction on $|\mathbf{d}|$, then on $|\mathbf{c}|$, yields
\[\begin{array}{ll}&(\rho_\chi(\bar u^{\mathbf{c}}\bar v^\mathbf{d}))(\bar x^{\mathbf{a}'}\bar y^{\mathbf{b}'}\bar u^{\mathbf{c}'}\bar v^{\mathbf{d}'}\otimes\bar 1_\chi)\\
=&(K'\bar x^{\mathbf{a}'}\bar y^{\mathbf{b}'}\bar u^{\mathbf{c}+\mathbf{c}'}\bar v^{\mathbf{d}+\mathbf{d}'}+
\sum\limits_{i,j}\kappa_{i,j}\bar x^{\mathbf{a}'-\mathbf{e}_i}\bar y^{\mathbf{b}'}\rho_\chi([\bar u_j,\bar x_i])\bar u^{\mathbf{c}+\mathbf{c}'-\mathbf{e}_j}\bar v^{\mathbf{d}+\mathbf{d}'}\\
&+\sum\limits_{i,j}\lambda_{i,j}\bar x^{\mathbf{a}'}\bar y^{\mathbf{b}'-\mathbf{e}_i}\rho_\chi([\bar u_j,\bar y_i])\bar u^{\mathbf{c}+\mathbf{c}'-\mathbf{e}_j}\bar v^{\mathbf{d}+\mathbf{d}'}+
\sum\limits_{i,j}\omega_{i,j}\bar x^{\mathbf{a}'-\mathbf{e}_i}\bar y^{\mathbf{b}'}\rho_\chi([\bar v_j,\bar x_i])\bar u^{\mathbf{c}+\mathbf{c}'}\cdot\\
&\bar v^{\mathbf{d}+\mathbf{d}'-\mathbf{e}_j}+\sum\limits_{i,j}\nu_{i,j}\bar x^{\mathbf{a}'}\bar y^{\mathbf{b}'-\mathbf{e}_i}\rho_\chi([\bar v_j,\bar y_i])\bar u^{\mathbf{c}+\mathbf{c}'}\bar v^{\mathbf{d}+\mathbf{d}'-\mathbf{e}_j}+\text{terms of}~e\text{-degree}\leqslant |(\mathbf{a}',\mathbf{b}',\\&\mathbf{c}+\mathbf{c}',\mathbf{d}+\mathbf{d}')|_e-3)\otimes\bar 1_\chi;
\end{array}\]
it follows that
\[\begin{array}{ll}&(\rho_\chi(\bar x^{\mathbf{a}}\bar y^\mathbf{b}\bar u^\mathbf{c}\bar v^\mathbf{d}))(\bar x^{\mathbf{a}'}\bar y^{\mathbf{b}'}\bar u^{\mathbf{c}'}\bar v^{\mathbf{d}'}\otimes\bar 1_\chi)\\
=&(K\bar x^{\mathbf{a}+\mathbf{a}'}\bar y^{\mathbf{b}+\mathbf{b}'}\bar u^{\mathbf{c}+\mathbf{c}'}\bar v^{\mathbf{d}+\mathbf{d}'}
+\sum\limits_{i<j}\alpha_{i,j}\bar x^{\mathbf{a}+\mathbf{a}'-\mathbf{e}_i-\mathbf{e}_j}\bar y^{\mathbf{b}+\mathbf{b}'}\rho_\chi([\bar x_i,\bar x_j])\bar u^{\mathbf{c}+\mathbf{c}'}\bar v^{\mathbf{d}+\mathbf{d}'}
+\sum\limits_{i,j}\beta_{i,j}\cdot\\&\bar x^{\mathbf{a}+\mathbf{a}'-\mathbf{e}_i}\bar y^{\mathbf{b}+\mathbf{b}'-\mathbf{e}_j}\rho_\chi([\bar x_i,\bar y_j])\bar u^{\mathbf{c}+\mathbf{c}'}\bar v^{\mathbf{d}+\mathbf{d}'}
+\sum\limits_{i<j}\gamma_{i,j}\bar x^{\mathbf{a}+\mathbf{a}'}\bar y^{\mathbf{b}+\mathbf{b}'-\mathbf{e}_i-\mathbf{e}_j}\rho_\chi([\bar y_i,\bar y_j]) \bar u^{\mathbf{c}+\mathbf{c}'}\cdot\\& \bar v^{\mathbf{d}+\mathbf{d}'}+\sum\limits_{i,j}\kappa'_{i,j}\bar x^{\mathbf{a}+\mathbf{a}'-\mathbf{e}_i}\bar y^{\mathbf{b}+\mathbf{b}'}\rho_\chi([\bar u_j,\bar x_i])\bar u^{\mathbf{c}+\mathbf{c}'-\mathbf{e}_j}\bar v^{\mathbf{d}+\mathbf{d}'}
+\sum\limits_{i,j}\lambda'_{i,j}\bar x^{\mathbf{a}+\mathbf{a}'}\bar y^{\mathbf{b}+\mathbf{b}'-\mathbf{e}_i}\cdot\\&\rho_\chi([\bar u_j,\bar y_i])\bar u^{\mathbf{c}+\mathbf{c}'-\mathbf{e}_j}\bar v^{\mathbf{d}+\mathbf{d}'}+\sum\limits_{i,j}\omega'_{i,j}\bar x^{\mathbf{a}+\mathbf{a}'-\mathbf{e}_i}\bar y^{\mathbf{b}+\mathbf{b}'}\rho_\chi([\bar v_j,\bar x_i])\bar u^{\mathbf{c}+\mathbf{c}'}\bar v^{\mathbf{d}+\mathbf{d}'-\mathbf{e}_j}
+\sum\limits_{i,j}\nu'_{i,j}\cdot\\&\bar x^{\mathbf{a}+\mathbf{a}'}\bar y^{\mathbf{b}+\mathbf{b}'-\mathbf{e}_i}\rho_\chi([\bar v_j,\bar y_i])\bar u^{\mathbf{c}+\mathbf{c}'}\bar v^{\mathbf{d}+\mathbf{d}'-\mathbf{e}_j}+\text{terms of}~e\text{-degree}\leqslant|(\mathbf{a}+\mathbf{a}',\mathbf{b}+\mathbf{b}',\mathbf{c}\\&+\mathbf{c}',\mathbf{d}+\mathbf{d}')|_e-3)\otimes\bar 1_\chi;
\end{array}\]where all the coefficients above are in $\bbk$, and $K',K\in{\bbk}$ are in the same sense as in Lemma~\ref{commutative relations k2}.
Together with Corollary~\ref{rg}, for $1\leqslant i,j\leqslant l+q$ we have
\begin{equation}\label{tt}
\begin{array}{cl}
&[\theta_i,\theta_j](\bar 1_\chi)\\
=&(\theta_i\cdot\theta_j)(\bar 1_\chi)-(-1)^{|\theta_i||\theta_j|}(\theta_j\cdot\theta_i)(\bar 1_\chi)\\
=&([\bar Y_i,\bar Y_j]+\sum\limits_{\mbox{\tiny $\begin{array}{c}|\mathbf{i},\mathbf{j},\mathbf{f},\mathbf{g}|_e=m_i+m_j+2,\\|\mathbf{i}|+|\mathbf{j}|+|\mathbf{f}|+|\mathbf{g}|\geqslant 2\end{array}$}}\mu_{\mathbf{i},\mathbf{j},\mathbf{f},\mathbf{g}}\bar x^{\mathbf{i}}
\bar y^{\mathbf{j}}\bar u^{\mathbf{f}}\bar v^{\mathbf{g}}+\sum\limits_{|(\mathbf{i},\mathbf{j},\mathbf{f},\mathbf{g})|_e<m_i+m_j+2}\mu_{\mathbf{i},\mathbf{j},\mathbf{f},\mathbf{g}}\bar x^{\mathbf{i}}
\bar y^{\mathbf{j}}\bar u^{\mathbf{f}}\bar v^{\mathbf{g}})\\&\otimes\bar 1_\chi,
\end{array}\end{equation}
where $\mu_{\mathbf{i},\mathbf{j},\mathbf{f},\mathbf{g}}\in{\bbk}$.
Since $[\bar Y_i,\bar Y_j]=\sum\limits_{k=1}^{l+q}\alpha_{ij}^k\bar Y_k$ for $1\leqslant i,j\leqslant l+q$ by the assumption in the theorem, taking $\pi_{m_i+m_j+2,m_i+m_j}$ on both sides of \eqref{tt}, we have $$\pi_{m_i+m_j+2,m_i+m_j}(([\theta_i,\theta_j]-\sum\limits_{k=1}^{l+q}\alpha_{ij}^k\theta_k)
(\bar 1_\chi))=0.$$
On the other hand, the arguments in the proof for (3)  show that there exists a unique truncated super-polynomial $\tilde{q}_{ij}$ in $\bar Y_1,\cdots,\bar Y_{l+q+1}$ such that $$[\theta_i,\theta_j]-\sum\limits_{k=1}^{l+q}\alpha_{ij}^k\theta_k
=\tilde{q}_{ij}(\theta_1,\cdots,\theta_{l+q+1}).$$
Taking the last three equalities into account, we have that the linear part of $\tilde{q}_{ij}$ involves only those $\bar Y_1,\cdots,\bar Y_{l+q+1}$ whose weights $<m_i+m_j$. Hence there exists a truncated super-polynomial $q_{ij}$ in $l+q+1$ variables with initial form of degree at least $2$ such that
\begin{equation*}
[\theta_i,\theta_j]-\sum\limits_{k=1}^{l+q}\alpha_{ij}^k\theta_k-q_{ij}(\theta_1,\cdots,\theta_{l+q+1})\in H^{m_i+m_j+1}
\end{equation*}for $1\leqslant i,j\leqslant l+q$.

It is immediate from the assumption in \S\ref{defw} that $$[\bar v_{\frac{\sfr+1}{2}},\bar v_{\frac{\sfr+1}{2}}]\otimes\bar 1_\chi= 1\otimes\langle\bar  v_{\frac{\sfr+1}{2}},\bar v_{\frac{\sfr+1}{2}}\rangle\bar 1_\chi=1\otimes\bar 1_\chi.$$
We have further
 $$[\theta_{l+q+1},\theta_{l+q+1}](\bar 1_\chi)=2\bar v_{\frac{\sfr+1}{2}}^2\otimes\bar 1_\chi=[\bar v_{\frac{\sfr+1}{2}},\bar v_{\frac{\sfr+1}{2}}]\otimes\bar 1_\chi=1\otimes\bar 1_\chi,$$
i.e. $[\theta_{l+q+1},\theta_{l+q+1}]=\text{id}$. The statement (2) is proved.

Finally, by the equations in \eqref{relaitionsss} and \eqref{tt}, the statement (1) in the theorem follows.

Summing up, we complete the proof.
\end{proof}
\begin{rem}\label{redun}
(1) It is notable that Theorem~\ref{reduced Wg} shows that the parity of $\sfr$ plays the key role for the construction of reduced $W$-superalgebras, which makes the structure and representation theory of reduced $W$-superalgebras distinguished from that of reduced $W$-algebras (see \cite[\S3]{P2} for ordinary reduced $W$-algebras).

(2) Recall that all the indices $i,j$ in Theorem~\ref{reduced Wg}(2) are assumed to be in the set $\{1,\cdots,l+q\}$. In this case, there exist truncated super-polynomials $\bar F_{ij}$ of $l+q'$ indeterminants over $\bbk$ ($i,j=1,\cdots,l+q$) with the first $l$ indeterminants  being even, and the others being odd, such that
\begin{align}\label{explainforeq}[\theta_i,\theta_j]=\bar F_{i,j}(\theta_1,\cdots,\theta_{l+q'}),\quad i,j=1,\cdots,l+q,
\end{align}
while $\bar F_{i,j}(\theta_1,\cdots,\theta_{l+q'})\equiv\sum_{k=1}^{l+q}\alpha_{ij}^k\theta_k+q_{ij}(\theta_1,\cdots,\theta_{l+q'})(\mbox{mod }H^{m_i+m_j+1})$ for $i,j=1,\cdots,l+q$ (Note that hereinafter the super polynomials in \eqref{explainforeq} are ``formal" ones, which can be endowed with real meaning when considered as the image under ``gradation" associated with Kazhdan filtration. This point is similar to the one happening in the ordinary Lie algebra case (cf. \cite[\S2.2]{P7})).

(3) For the case when $\sfr$ is odd, if one of the indices $i,j$ happens to be $l+q+1$, then it is hard to derive an explicit formulas for $[\bar Y_i,\bar Y_{l+q+1}]$ and $[\bar Y_{l+q+1},\bar Y_j]$ since $\mathfrak{g}_{\bbk}^e\oplus\bbk\bar v_{\frac{\sfr+1}{2}}$ is not necessary a Lie superalgebra. However, by Lemma~\ref{commutative relations k2} and Theorem~\ref{reduced Wg}(3) one can still find truncated super-polynomials $\bar F_{i,l+q+1}$'s and $\bar F_{l+q+1,j}$'s in $l+q+1$ indeterminants over $\mathds{k}$ ($1\leqslant i,j\leqslant l+q$) such that
\begin{align}\label{explainforeq2}[\theta_i,\theta_{l+q+1}]=\bar F_{i,l+q+1}(\theta_1,\cdots,\theta_{l+q+1}),\quad [\theta_{l+q+1},\theta_{j}]=\bar F_{l+q+1,j}(\theta_1,\cdots,\theta_{l+q+1}),
\end{align}
where the $e$-degree for all the monomials in the $\bbk$-span of $\bar F_{i,l+q+1}$'s and $\bar F_{l+q+1,j}$'s is less than $m_i+1$ and $m_j+1$, respectively.

(4) It is notable that $[\theta_i,\theta_j]=-(-1)^{|\theta_i||\theta_j|}[\theta_j,\theta_i]$ for any $1\leqslant  i,j\leqslant  l+q'$. In particular, one can deduce that $[\theta_i,\theta_i]=0$ for $1\leqslant  i\leqslant  l$ as $\theta_i$ is always an even element in $U_\chi({\ggg}_{\bbk},e)$. Therefore, after deleting all the redundant commutating relations in \eqref{explainforeq0}, \eqref{explainforeq} and \eqref{explainforeq2}, the remaining ones are with indices $i,j$ satisfying $1\leqslant  i<j\leqslant  l+q'$ and $l+1\leqslant  i=j\leqslant  l+q'$.
\end{rem}

\begin{rem} \label{add to reduced W} More generally, one can consider a reduced $W$-superalgebra $U_\eta(\ggg_\bbk,e)$ associated with a $p$-character $\eta\in\chi+(\mmm_\bbk^\perp)_\bz$ (see Definition \ref{reduced W}).  Proposition \ref{invariant} and Theorem \ref{matrix}(1) enable us to extend the arguments in this section for the case $U_\chi(\ggg_\bbk,e)$ to the general case $U_\eta(\ggg_\bbk,e)$. Especially, we can rewrite Theorem \ref{reduced Wg}(3) as follows.

 {\emph{There exist even elements $\theta_1,\cdots,\theta_l\in U_\eta({\ggg}_{\bbk},e)_{\bar0}$ and odd elements $ \theta_{l+1},\cdots,\theta_{l+q'}\in U_\eta({\ggg}_{\bbk},e)_{\bar1}$ in the same sense as in Corollary \ref{rg}, such that
 the monomials $$\theta_1^{a_1}\cdots\theta_{l}^{a_l}\theta_{l+1}^{b_1}\cdots\theta_{l+q'}^{b_{q'}}$$ with $0\leqslant  a_k\leqslant  p-1$ for $1\leqslant k\leqslant l$ and $b_k\in\{0,1\}$ for $1\leqslant k\leqslant q'$ form a ${\bbk}$-basis of $U_\eta({\ggg}_{\bbk},e)$.}}
\end{rem}

\section{The structure of finite $W$-superalgebras over $\bbc$}\label{swc}
Maintain the notations as the previous sections. Especially, let $\ggg$ be a basic Lie superalgebra over $\bbc$ with a given nilpotent element $e\in\ggg_\bz$, $A$ be an admissible ring associated with the pair $(\ggg,e)$.

In this concluding section, we will establish the structure theory of finite $W$-superalgebra $U(\ggg,e)$ over $\bbc$, parallel to that over $\bbk$. The approach  is the procedure of  ``admissible'' via the admissible ring $A$ associated to the basic Lie superalgebra $\ggg$, which was exploited by Premet in the study of finite $W$-algebras (cf. \cite{P2}).  The PBW theorem of finite $W$-superalgebra $U({\ggg},e)$ over $\mathbb{C}$ will be present in \S\ref{cwc}.

 We need recall some primary  conventions parallel to \S\ref{modpconventions}. For example, for $k\in\mathbb{Z}_+$, define
 \begin{equation*}
 \begin{array}{llllll}
 \mathbb{Z}_+^k&:=&\{(i_1,\cdots,i_k)\mid i_j\in\mathbb{Z}_+\},&
 \Lambda'_k&:=&\{(i_1,\cdots,i_k)\mid i_j\in\{0,1\}\}
 \end{array}
 \end{equation*}with $1\leqslant j\leqslant k$.
 As in  \S\ref{defw} and \S\ref{modulo p}, there is a set of ``monomials basis" in $Q_\chi$ over $\mathbb{C}$ as follows
$$x^\mathbf{a}y^\mathbf{b}u^\mathbf{c}v^\mathbf{d}\otimes 1_\chi:=x_1^{a_1}\cdots x_m^{a_m}y_1^{b_1}\cdots y_n^{b_n}u_1^{c_1}\cdots u_s^{c_s}v_1^{d_1}\cdots v_t^{d_t}\otimes 1_\chi,$$
where $(\mathbf{a},\mathbf{b},\mathbf{c},\mathbf{d})$ runs over $\mathbb{Z}_+^m\times\Lambda'_n\times\mathbb{Z}_+^s\times\Lambda'_t$. (recall that $t=\lfloor\frac{\text{dim}\,{\ggg}(-1)_{\bar1}}{2}\rfloor$)

Assume that $\text{wt}(x_i)=k_i$, $\text{wt}(y_j)=k'_j$, i.e. $x_i\in{\ggg}(k_i)_{\bar0}$ and $y_j\in{\ggg}(k'_j)_{\bar1}$ where $1\leqslant  i\leqslant  m$ and $1\leqslant  j\leqslant  n$. Given $(\mathbf{a},\mathbf{b},\mathbf{c},\mathbf{d})\in\mathbb{Z}_+^m\times\Lambda'_n\times\mathbb{Z}_+^s\times\Lambda'_t$, set
$$|(\mathbf{a},\mathbf{b},\mathbf{c},\mathbf{d})|_e:=\sum_{i=1}^ma_i(k_i+2)+\sum_{i=1}^nb_i(k'_i+2)+\sum_{i=1}^sc_i+\sum_{i=1}^td_i,$$
and say $x^{\mathbf{a}}y^\mathbf{b}u^\mathbf{c}v^\mathbf{d}$ to have $e$-degree $|(\mathbf{a},\mathbf{b},\mathbf{c},\mathbf{d})|_e$, which is written as
$\text{deg}_e(x^{\mathbf{a}}y^\mathbf{b}u^\mathbf{c}v^\mathbf{d})=|(\mathbf{a},\mathbf{b},\mathbf{c},\mathbf{d})|_e$.

\subsection{Some Lemmas}
Parallel to Lemma~\ref{commutative relations k1}, one can easily obtain that

\begin{lemma}\label{commutative}
For any homogeneous element $w\in U({\ggg})_{i}$ ($i\in\mathbb{Z}_2$), we have
$$w\cdot x^{\mathbf{a}}y^\mathbf{b}u^\mathbf{c}v^\mathbf{d}=\sum_{\mathbf{i}\in\mathbb{Z}_+^m}\sum_{\mathbf{j}\in\Lambda_{n}'}\left(\begin{array}{@{\hspace{0pt}}c@{\hspace{0pt}}} \mathbf{a}\\ \mathbf{i}\end{array}\right)x^{\mathbf{a}-\mathbf{i}}y^{\mathbf{b}-\mathbf{j}}\cdot[wx^{\mathbf{i}}y^{\mathbf{j}}]\cdot u^\mathbf{c}v^\mathbf{d},$$
where $\mathbf{a}\choose\mathbf{i}$$=\prod\limits_{l'=1}^m$$a_{l'}\choose i_{l'}$, and $$[wx^{\mathbf{i}}y^{\mathbf{j}}]=k_{1,b_1,j_1}\cdots k_{n,b_n,j_n}(-1)^{|\mathbf{i}|}(\text{ad}y_n)^{j_n}\cdots(\text{ad}y_1)^{j_1}(\text{ad}x_m)^{i_m}\cdots(\text{ad}x_1)^{i_1}(w),$$
in which the coefficients $k_{1,b_1,j_1},\cdots, k_{n,b_n,j_n}\in\mathbb{C}$ (note that $b_1,\cdots,b_n,j_1,\cdots,j_n\in\{0,1\}$) are defined by
$$k_{t',0,0}=1, k_{t',0,1}=0, k_{t',1,0}=(-1)^{|w|+j_1+\cdots+j_{{t'}-1}}, k_{t',1,1}=(-1)^{|w|+1+j_1+\cdots+j_{{t'}-1}},$$
where $1\leqslant {t'}\leqslant n$ ($j_0$ is interpreted as $0$).
\end{lemma}
\vskip0.2cm
Let $\tilde{\rho}_\chi$ denote the representation of $U({\ggg})$ in $\text{End}Q_\chi$, then we have

\begin{lemma}\label{com c}
Let $(\mathbf{a},\mathbf{b},\mathbf{c},\mathbf{d}),~(\mathbf{a}',\mathbf{b}',\mathbf{c}',\mathbf{d}')\in\mathbb{Z}_+^m\times\Lambda'_n\times\mathbb{Z}_+^s\times\Lambda'_t$ be such that $|(\mathbf{a},\mathbf{b},\mathbf{c},\mathbf{d})|_e=A,~|(\mathbf{a}',\mathbf{b}',\mathbf{c}',\mathbf{d}')|_e=B$. Then
\[\begin{array}{ccl}(\tilde{\rho}_\chi(x^{\mathbf{a}}y^\mathbf{b}u^\mathbf{c}v^\mathbf{d}))(x^{\mathbf{a}'}y^{\mathbf{b}'}u^{\mathbf{c}'}v^{\mathbf{d}'}\otimes1_\chi)&=&
(Cx^{\mathbf{a}+\mathbf{a}'}y^{\mathbf{b}+\mathbf{b}'}u^{\mathbf{c}+\mathbf{c}'}v^{\mathbf{d}+\mathbf{d}'}+\text{terms of}~e\text{-degree}\\ &&\leqslant A+B-2)\otimes1_\chi,
\end{array}\]
where the coefficient $C\in \{-1,0,1\}\subseteq\bbc$ is subject to the following items:
\begin{itemize}
\item[(1)] $C=0$ if and only if $(\mathbf{b}+\mathbf{b}',\mathbf{d}+\mathbf{d}')\notin\Lambda'_n\times\Lambda'_t$;
\item[(2)] When $(\mathbf{b}+\mathbf{b}',\mathbf{d}+\mathbf{d}')\in\Lambda'_n\times\Lambda'_t$, then $C=(-1)^{\tau(\mathbf{b},\mathbf{d},\mathbf{b}',\mathbf{d}')}$. Here $\tau(\mathbf{b},\mathbf{d},\mathbf{b}',\mathbf{d}')$ has the same meaning as in Lemma \ref{commutative relations k2}.
\end{itemize}
\end{lemma}
\vskip0.2cm
For the proof of the above lemma, we can repeat that of Lemma~\ref{commutative relations k2} applying Lemma~\ref{commutative} in place of Lemma~\ref{commutative relations k1}.
We omit the details.

Now we can get the $\bbc$-analogue of the leading-terms lemma. Recall any $0\neq h\in U({\ggg},e)$ is uniquely determined by its value $h(1_\chi)\in Q_\chi$. For $h\neq0$ we let $n(h),~N(h)$ and $\Lambda_h^{\text{max}}$ be the same meaning as in \S\ref{maximalandweight}.

\begin{lemma}\label{hwc}
Let $h\in U({\ggg},e)\backslash\{0\}$ and $(\mathbf{a},\mathbf{b},\mathbf{c},\mathbf{d})\in \Lambda^{\text{max}}_h$. Then $\mathbf{c}=\mathbf{0}$, and $\mathbf{a}\in\mathbb{Z}_+^{l}\times\{\mathbf{0}\},~\mathbf{b}\in\Lambda'_{q}\times\{\mathbf{0}\}$. Moreover, the sequence $\mathbf{d}$ satisfies

(1) $\mathbf{d}=\mathbf{0}$ when dim\,${\ggg}(-1)_{\bar{1}}$ is even;

(2) $\mathbf{d}=\{\mathbf{0}\}_{\frac{\sfr-1}{2}}\times\Lambda'_1$ when dim\,${\ggg}(-1)_{\bar{1}}$ is odd.
\end{lemma}

\begin{proof}
Repeat verbatim the proof of Lemma~\ref{hw} but apply Lemma~\ref{com c} in place of Lemma~\ref{commutative relations k2}.
\end{proof}

\begin{rem}
In \cite[Theorem 2.5]{PS2}, Poletaeva-Serganova also formulate the similar statement to Lemma~\ref{hwc} bypassing Lemma~\ref{commutative} and Lemma~\ref{com c}.
However,  Lemma~\ref{commutative} and Lemma~\ref{com c} will play the key role for the construction theory of finite $W$-superalgebra $U({\ggg},e)$ in Theorem~\ref{PBWC}.
\end{rem}

\subsection{The PBW structure theory of finite $W$-superalgebras over $\bbc$}\label{cwc}
For $k\in\mathbb{Z}_+$ we denote by $\tilde{H}^k$ the linear span of all $0\neq h\in U({\ggg},e)$ with $n(h)\leqslant  k$. Put $\tilde{H}^{-1}=0$. It follows from Lemma~\ref{com c} that the subspaces $\{\tilde{H}^i\,|\,i\in\mathbb{Z}_+\}$ form a filtration of the algebra $U({\ggg},e)$. Moreover, Lemma~\ref{com c} implies that the graded algebra $\text{gr}(U({\ggg},e))=\bigoplus\limits_{i\geqslant 0}\tilde{H}^i/\tilde{H}^{i-1}$ is supercommutative.

Recall that $\{x_1,\cdots,x_l\}$ and $\{y_1,\cdots,y_q\}$ are $\mathbb{C}$-basis of ${\ggg}^e_{\bar{0}}$ and ${\ggg}^e_{\bar{1}}$, respectively. When dim\,${\ggg}(-1)_{\bar{1}}$ is odd, there is  $v_{\frac{\sfr+1}{2}}\in{\ggg}(-1)_{\bar{1}}\cap({\ggg}(-1)'_{\bar{1}})^{\bot }$. Set
\[Y_i:=\left\{
\begin{array}{ll}
x_i&\text{if}~1\leqslant  i\leqslant  l;\\
y_{i-l}&\text{if}~l+1\leqslant  i\leqslant  l+q;\\
v_{\frac{\sfr+1}{2}}&\text{if}~i=l+q+1 \mbox{ whenever }\sfr\mbox{ is odd}.
\end{array}
\right.
\]
By assumption it is immediate that $Y_i\in{\ggg}^e$ for $1\leqslant  i\leqslant  l+q$. The term $Y_{l+q+1}\notin{\ggg}^e$ occurs only when $\sfr=\dim\ggg(-1)_{\bar{1}}$ is odd. Assume that $Y_i$ falls in ${\ggg}(m_i)$ for $1\leqslant  i\leqslant  l+q+1$.
In the sequent arguments, we further set
$$q':=\begin{cases} q &\mbox{ if }\sfr \mbox{ is even; }\cr
q+1 & \mbox{ if }\sfr\mbox{ is odd}.\end{cases}$$
Let $\tilde{\Theta}_i$ denote the image of $\Theta_i\in U({\ggg},e)$ in $\text{gr}\,(U({\ggg},e))$. We are in the position to present the PBW structure of finite $W$-superalgebras over $\bbc$.
 We will further understand, as seeing in the last section, that the variation of the parity  of $\sfr$ gives rise to the change of the structure of finite $W$-superalgebras.  

\begin{theorem}\label{PBWC}
Let $U({\ggg},e)$ be a finite $W$-superalgebra over $\mathbb{C}$. The following PBW structural statements for $U(\ggg,e)$ hold, corresponding to the cases when  $\sfr=\dim\ggg(-1)_{\bo}$  is even and when $\sfr$ is odd, respectively.
\begin{itemize}
\item[(1)] There exist homogeneous elements $\Theta_1,\cdots,\Theta_{l}\in U({\ggg},e)_{\bar0}$ and $\Theta_{l+1},\cdots,\Theta_{l+q'}\in U({\ggg},e)_{\bar1}$ such that
\begin{equation*}
\begin{array}{ll}
&\Theta_k(1_\chi)\\=&(Y_k+\sum\limits_{\mbox{\tiny $\begin{array}{c}|\mathbf{a},\mathbf{b},\mathbf{c},\mathbf{d}|_e=m_k+2,\\|\mathbf{a}|
+|\mathbf{b}|+|\mathbf{c}|+|\mathbf{d}|\geqslant 2\end{array}$}}\lambda^k_{\mathbf{a},\mathbf{b},\mathbf{c},\mathbf{d}}x^{\mathbf{a}}
y^{\mathbf{b}}u^{\mathbf{c}}v^{\mathbf{d}}+\sum\limits_{|\mathbf{a},\mathbf{b},\mathbf{c},\mathbf{d}|_e<m_k+2}\lambda^k_{\mathbf{a},\mathbf{b},\mathbf{c},\mathbf{d}}x^{\mathbf{a}}
y^{\mathbf{b}}u^{\mathbf{c}}v^{\mathbf{d}})\otimes1_\chi
\end{array}
\end{equation*}
for $1\leqslant  k\leqslant  l+q$, where $\lambda^k_{\mathbf{a},\mathbf{b},\mathbf{c},\mathbf{d}}\in\mathbb{Q}$, and $\lambda^k_{\mathbf{a},\mathbf{b},\mathbf{c},\mathbf{d}}=0$ if $a_{l+1}=\cdots=a_m=b_{q+1}=\cdots=b_n=c_1=\cdots=c_s=
d_1=\cdots=d_{\lceil\frac{\sfr}{2}\rceil}=0$.

Additionally  set $\Theta_{l+q+1}(1_\chi)=v_{\frac{\sfr+1}{2}}\otimes1_\chi$  when $\sfr$ is odd.

\item[(2)] The monomials $\Theta_1^{a_1}\cdots\Theta_l^{a_l}\Theta_{l+1}^{b_1}\cdots\Theta_{l+q'}^{b_{q'}}$ with $a_i\in\mathbb{Z}_+,~b_j\in\Lambda'_1$ for $1\leqslant i\leqslant l$ and $1\leqslant j\leqslant q'$ form a basis of $U({\ggg},e)$ over $\mathbb{C}$.

\item[(3)] For $1\leqslant  i\leqslant  l+q'$,  the elements $\tilde{\Theta}_i=\Theta_i+\tilde{H}^{m_i+1}\in\text{gr}\,(U({\ggg},e))$ are algebraically independent and generate $\text{gr}\,(U({\ggg},e))$. In particular, $\text{gr}\,(U({\ggg},e))$ is a graded polynomial superalgebra with homogeneous generators of degrees $m_1+2,\cdots,m_{l+q'}+2$.

\item[(4)] For  $1\leqslant  i,j\leqslant  l+q'$, we have
$$[\Theta_i,\Theta_j]\in\tilde{H}^{m_i+m_j+2}.$$
Moreover, if the elements $Y_i,~Y_j\in {\ggg}^e$ for $1\leqslant  i,j\leqslant  l+q$ satisfy $[Y_i,Y_j]=\sum\limits_{k=1}^{l+q}\alpha_{ij}^kY_k$ in ${\ggg}^e$, then
\begin{equation}\label{Thetaa2}
\begin{array}{llllll}
&[\Theta_i,\Theta_j]&\equiv&\sum\limits_{k=1}^{l+q}\alpha_{ij}^k\Theta_k+q_{ij}(\Theta_1,\cdots,\Theta_{l+q'})&(\text{mod}~\tilde{H}^{m_i+m_j+1}),
\end{array}
\end{equation}where $q_{ij}$ is a super-polynomial in $l+q'$ variables in $\mathbb{Q}$ whose constant term and linear part are zero.

\item[(5)] When $\sfr$ is odd, if one of $i$ and $j$ happens to be $l+q+1$, we can find  super-polynomials $F_{i,l+q+1}$ and $F_{l+q+1,j}$ in $l+q+1$ invariants over $\mathbb{Q}$ ($1\leqslant i,j\leqslant l+q$) such that
\begin{equation}\label{il+q+1}
\begin{array}{lll}
[\Theta_i,\Theta_{l+q+1}]&=&F_{i,l+q+1}(\Theta_1,\cdots,\Theta_{l+q+1}),\\\,
[\Theta_{l+q+1},\Theta_j]&=&F_{l+q+1,j}(\Theta_1,\cdots,\Theta_{l+q+1}),
\end{array}
\end{equation}
where the $e$-degree for all the monomials in the $\bbc$-span of $F_{i,l+q+1}$'s and $F_{l+q+1,j}$'s is less than $m_i+1$ and $m_j+1$, respectively.

  Moveover, if $i=j=l+q+1$, then \begin{equation}\label{idid}
  [\Theta_{l+q+1},\Theta_{l+q+1}]=\text{id}.
  \end{equation}
\end{itemize}
\end{theorem}

In \cite[\S4]{P2}, Premet obtained the PBW theorem for the finite $W$-algebras over $\mathbb{C}$ through the procedure of ``admissible'' with aid of reduced $W$-algebras over ${\bbk}$. Since the ``modulo $p$'' version of Theorem ~\ref{PBWC} has been formulated in Theorem~\ref{reduced Wg},
and the choice of admissible ring $A$ has nothing to do with the super property of ${\ggg}$,  we can prove the theorem by the same spirit of Premet's argument for  the ordinary finite $W$-algebra case  based on the results of reduced $W$-superalgebras over ${\bbk}$ in \S\ref{rew}. Explicitly speaking, in virtue of the lemmas introduced in \S\ref{swc}, one can translate the formulas in Statement (1) of the theorem to a system of linear equations, then (1) follows by the discussion on the existence of solution for these equations over $\mathbb{Q}$ by the knowledge of field theory. The remaining consequence can be easily obtained by the same consideration as Theorem~\ref{reduced Wg} and Remark~\ref{redun}.
Now we only need to  give a sketchy proof.

\begin{proof}
Repeat verbatim the proof of Theorem 4.6 in \cite{P2} but apply Theorem~\ref{reduced Wg} (also Remark~\ref{redun}), Lemma~\ref{com c} and Lemma~\ref{hwc} in place of Theorem 3.4, Lemma 4.4 and Lemma 4.5 in \cite{P2} respectively, then the consequence follows. For more details one refers to \cite[Theorem 4.6]{P2}.
\end{proof}

\begin{rem} \label{polynomial}
Recall that we calculated the dimension of $U_\chi({\ggg}_{\bbk},e)$ for the proof of Proposition~\ref{reduced basis}. However, since the dimension of finite $W$-superalgebra $U({\ggg},e)$ over $\mathbb{C}$ is infinite, similar conclusion as Proposition~\ref{reduced basis} for the $\bbc$-algebra $U(\ggg,e)$ can not be obtained directly. Therefore, one can not establish Theorem~\ref{PBWC} directly by the same means as Theorem~\ref{reduced Wg}, bypassing the knowledge of the $\bbk$-algebra $U_{\chi}(\ggg_\bbk,e)$ in \S\ref{rew}. It is also remarkable that all the coefficients for the generators of $U({\ggg},e)$ we obtained in Theorem~\ref{PBWC}(1) are over $\mathbb{Q}$.
\end{rem}

By Theorem~\ref{PBWC}(4), there exist super-polynomials $F_{ij}$'s of $l+q'$ indeterminants   over $\bbq$ for $i,j=1,\cdots,l+q$ with the first $l$ indeterminants  being even, and the others being odd, such that
$$[\Theta_i,\Theta_j]=F_{i,j}(\Theta_1,\cdots,\Theta_{l+q'}),\quad i,j=1,\cdots,l+q,$$
while $F_{ij}(\Theta_1,\cdots,\Theta_{l+q'})\equiv\sum_{k=1}^{l+q}\alpha_{ij}^k\Theta_k+q_{ij}(\Theta_1,\cdots,\Theta_{l+q'})(\mbox{mod }\tilde H^{m_i+m_j+1})$ for $i,j=1,\cdots,l+q$, and here the super-polynomials are still under the same sense as in Remark \ref{redun}(2). Note that by Theorem~\ref{PBWC}(5), when $\sfr$ is odd there are still super-polynomials $F_{i,j}$'s of $l+q+1$ indeterminants   over $\bbq$ with $i$ or $j$ equals $l+q+1$. By the analogous arguments of \cite[Lemma 4.1]{P4}, one can prove the following result.
\begin{theorem}\label{relationc}
The finite $W$-superalgebra $U(\mathfrak{g},e)$ over $\bbc$ is generated by the $\mathbb{Z}_2$-homogeneous elements $\Theta_1,\cdots,\Theta_{l}\in U(\mathfrak{g},e)_{\bar0}$ and $\Theta_{l+1},\cdots,\Theta_{l+q'}\in U(\mathfrak{g},e)_{\bar1}$ subject to the relations
$$[\Theta_i,\Theta_j]=F_{ij}(\Theta_1,\cdots,\Theta_{l+q'})$$
with $1\leqslant i,j\leqslant l+q'$.
\end{theorem}

\subsection{Revisit to Theorem \ref{graded W}}  Define the element $\Theta\in U({\ggg},e)$ by letting $\Theta(1_\chi)=v_{\frac{\sfr+1}{2}}\otimes1_\chi$ for the case when dim\,${\ggg}(-1)_{\bar1}$ is odd. From consequences of Lemma~\ref{hwc}  and Theorem~\ref{PBWC}, Theorem~\ref{graded W} can be easily obtained. We omit the detailed proof.

\subsection{Some other definition of finite $W$-superalgebras}

Recall in \cite[Remark 70]{W} Wang defined a reduced $W$-superalgebra over ${\bbk}=\overline{\mathbb{F}}_p$ in another way, i.e.
$W'_{\chi,{\bbk}}: =(Q^\chi_\chi)^{\text{ad}\,{\mmm}'_{{\bbk}}}$,
which seems to make better sense (there this superalgebra is called the modular $W$-superalgebra).
In light of this
 (also cf. \cite{GG}), we can also define the corresponding superalgebra over $\bbc$.

\begin{defn}\label{rewcc} Let $\ggg$ be a basic Lie superalgebra over $\bbc$, and retain the data appearing as before. Define the $\bbc$-algebra
 $$W'_\chi:=(U({\ggg})/I_\chi)^{\text{ad}\,{\mmm}'}\cong Q_\chi^{\text{ad}\,{\mmm}'}
\equiv\{\bar{y}\in U({\ggg})/I_\chi \mid [a,y]\in I_\chi, \forall a\in{\mmm}'\},$$
where $\bar{y}_1\cdot\bar{y}_2:=\overline{y_1y_2}$ for all $\bar{y}_1,\bar{y}_2\in W'_\chi$.
\end{defn}

When $\sfr=\text{dim}\,{\ggg}(-1)_{\bar1}$ is even, since ${\mmm}'={\mmm}$ by the definition, the isomorphism in Theorem \ref{W-C2} shows that $U({\ggg},e)\cong W'_\chi$ as $\mathbb{C}$-algebras. However, the situation changes in the case when $\text{dim}\,{\ggg}(-1)_{\bar1}$ is odd. Since ${\mmm}$ is a proper subalgebra of ${\mmm}'$, it follows that $W'_\chi$ is a subalgebra of $Q_\chi^{\text{ad}{\mmm}}=U({\ggg},e)$. In the following, we give a precise connection between the $\bbc$-algebras $W'_\chi$ and $U(\ggg,e)$.

\begin{prop}\label{small}
For the case when $\text{dim}\,{\ggg}(-1)_{\bar1}$ is odd, the $\mathbb{C}$-algebra $W'_\chi\cong Q_\chi^{\text{ad}\,{\mmm}'}$ is a proper subalgebra of $U({\ggg},e)\cong Q_\chi^{\text{ad}\,{\mmm}}$. Moreover, we have
$$Q_\chi^{\text{ad}\,{\mmm}'}=[v_{\frac{\sfr+1}{2}},Q_\chi^{\text{ad}\,{\mmm}}].$$
\end{prop}

\begin{proof}

 At first, we claim that $Q_\chi^{\text{ad}\,{\mmm}'}$ is properly contained in $Q_\chi^{\text{ad}\,{\mmm}}$. Recall Theorem~\ref{PBWC}(1) shows that $\Theta_{l+q+1}(1_\chi)=v_{\frac{\sfr+1}{2}}\otimes1_\chi\in Q_\chi^{\text{ad}\,{\mmm}}$. By the definition we have $v_{\frac{\sfr+1}{2}}\in{\mmm}'$, and $[v_{\frac{\sfr+1}{2}},v_{\frac{\sfr+1}{2}}\otimes1_\chi]=[v_{\frac{\sfr+1}{2}},v_{\frac{\sfr+1}{2}}]\otimes
1_\chi=1\otimes\chi([v_{\frac{\sfr+1}{2}},v_{\frac{\sfr+1}{2}}])1_\chi=1\otimes1_\chi$. Thus $v_{\frac{\sfr+1}{2}}\otimes1_\chi\in
Q_\chi^{\text{ad}\,{\mmm}}$, but not in $Q_\chi^{\text{ad}\,{\mmm}'}$.

Next we consider the relation between the $\mathbb{C}$-algebras $Q_\chi^{\text{ad}\,{\mmm}'}$ and $Q_\chi^{\text{ad}\,{\mmm}}$.

(1) For any $\mathbb{Z}_2$-homogeneous element $x\in Q_\chi^{\text{ad}\,{\mmm}}$, we claim that
$$[v_{\frac{\sfr+1}{2}},x]\subseteq Q_\chi^{\text{ad}\,{\mmm}'}.$$

(i) Recall that $\Theta_{l+q+1}(1_\chi)=v_{\frac{\sfr+1}{2}}\otimes1_\chi\in Q_\chi^{\text{ad}\,{\mmm}}$. Since $x\in Q_\chi^{\text{ad}\,{\mmm}}$, it follows from the Jacobi identity that $[v_{\frac{\sfr+1}{2}},x]\in Q_\chi^{\text{ad}\,{\mmm}}$.

(ii) Since $v_{\frac{\sfr+1}{2}}\in{\ggg}(-1)_{\bar{1}}$, then $[v_{\frac{\sfr+1}{2}},v_{\frac{\sfr+1}{2}}]\in{\ggg}(-2)_{\bar0}\subseteq{\mmm}_{\bar0}$, and it follows from $x\in Q_\chi^{\text{ad}\,{\mmm}}$ that $[[v_{\frac{\sfr+1}{2}},v_{\frac{\sfr+1}{2}}],x]=0$. On the other hand,
\[\begin{array}{ccl}
[[v_{\frac{\sfr+1}{2}},v_{\frac{r+1}{2}}],x]&=&[v_{\frac{\sfr+1}{2}},[v_{\frac{\sfr+1}{2}},x]]+(-1)^{|x|}[[v_{\frac{\sfr+1}{2}},x],v_{\frac{\sfr+1}{2}}]\\
&=&[v_{\frac{\sfr+1}{2}},[v_{\frac{\sfr+1}{2}},x]]-(-1)^{|x|}\cdot(-1)^{|x|+1}[v_{\frac{\sfr+1}{2}},[v_{\frac{\sfr+1}{2}},x]]\\
&=&[v_{\frac{\sfr+1}{2}},[v_{\frac{\sfr+1}{2}},x]]+(-1)^{2|x|+2}[v_{\frac{\sfr+1}{2}},[v_{\frac{\sfr+1}{2}},x]]\\
&=&2[v_{\frac{\sfr+1}{2}},[v_{\frac{\sfr+1}{2}},x]],
\end{array}
\]
then $[v_{\frac{\sfr+1}{2}},[v_{\frac{\sfr+1}{2}},x]]=0$, i.e. $[v_{\frac{\sfr+1}{2}},x]\in Q_\chi^{\text{ad}v_{\frac{\sfr+1}{2}}}$.

Since ${\mmm}'={\mmm}\,\oplus\,\mathbb{C}v_{\frac{\sfr+1}{2}}$ as vector spaces, and (i) and (ii) show that $[v_{\frac{\sfr+1}{2}},x]\in Q_\chi^{\text{ad}\,{\mmm}'}$, then $[v_{\frac{\sfr+1}{2}},Q_\chi^{\text{ad}\,{\mmm}}]\subseteq  Q_\chi^{\text{ad}\,{\mmm}'}$ follows by the variation of $x$.

(2) We claim that the reverse of (1) is also true, i.e. $[v_{\frac{\sfr+1}{2}},Q_\chi^{\text{ad}\,{\mmm}}]\supseteq Q_\chi^{\text{ad}\,{\mmm}'}$.

As $v_{\frac{\sfr+1}{2}}\in{\mmm}'$, and $v_{\frac{\sfr+1}{2}}\otimes1_\chi\in Q_\chi^{\text{ad}\,{\mmm}}$, then for any $\mathbb{Z}_2$-homogeneous element $y\otimes1_\chi\in Q_\chi^{\text{ad}\,{\mmm}'}\subseteq Q_\chi^{\text{ad}\,{\mmm}}$, we have $[v_{\frac{\sfr+1}{2}},y\otimes1_\chi]=0$, and $yv_{\frac{\sfr+1}{2}}\otimes1_\chi\in Q_\chi^{\text{ad}\,{\mmm}}$ by the Jacobi identity. Since$$[v_{\frac{\sfr+1}{2}},yv_{\frac{\sfr+1}{2}}\otimes1_\chi]=([v_{\frac{\sfr+1}{2}},y]v_{\frac{\sfr+1}{2}}+(-1)^{|y|}y[v_{\frac{\sfr+1}{2}},v_{\frac{\sfr+1}{2}}])\otimes1_\chi=(-1)^{|y|}y\otimes1_\chi,$$i.e. $y\otimes1_\chi=[v_{\frac{\sfr+1}{2}},(-1)^{|y|}yv_{\frac{\sfr+1}{2}}\otimes1_\chi]$, it is straightforward that $[v_{\frac{\sfr+1}{2}},Q_\chi^{\text{ad}\,{\mmm}}]\supseteq Q_\chi^{\text{ad}\,{\mmm}'}$ by the arbitrary of $y$.

Summing up, we complete the proof.
\end{proof}

{\bf Acknowledgements}\quad
Most parts of the results in this article were established in 2012 when the first author was a student under the supervision of the second author, which can be found in the first author's PhD thesis \cite{Zeng} in East China Normal University. The authors would like to thank Weiqiang Wang and Lei Zhao whose work on super version of Kac-Weisfeiler property stimulated them to do the present research. The authors got much help from the discussion with Hao Chang and Weiqiang Wang. The authors express great thanks to them.

\end{document}